\documentclass[12pt,amssymb,amscd]{amsart}
\usepackage{amsmath,amssymb,graphicx,color,mathrsfs,tikz}
\usepackage{amsmath,mathrsfs}
\usepackage[titletoc]{appendix}
\usetikzlibrary{arrows}
\usepackage[margin=1.2 in, includefoot]{geometry}
\usepackage{color}
\usepackage{amssymb}
\usepackage{amsmath}
\usepackage{amsthm}
\usepackage{mathrsfs}
\usepackage[all]{xy}
\usepackage{graphicx}
\usepackage{wrapfig}

\def\fsh{{\mathcal U}_{\hbar}(\widehat{\frak{sl}}_{2}) }
\def\O {{\mathcal{O}}}

\newcommand{\Z}{{\mathbb{Z}}}
\newcommand{\Q}{{\mathbb{Q}}}

\def\Gr{\textbf{Gr}}
\def\G{\textbf{G}}
\def\Orb{\textbf{O}}

\def\M{\textbf{M}}
\def\N{\textbf{N}}
\def\h{\hbar}
\def\fp{ {\textbf{p}}  }
\def\fq{ {\textbf{q}}  }
\def\fn{ {\textbf{n}}  }

\def\fM{ {\frak{M}}  }
\def\fB{ {\frak{B}}  }  
\def\qm {{{\textnormal{\textsf{QM}}}}} 
\def\qotimes {\circledast} 

\def\tb {{\mathcal{V}}} 
\def\tbW{{\mathcal{W}}} 
\def\tbP{{\mathcal{P}}} 
\def \vss {\widehat{{\O}}_{{\rm{vir}}}} 

\def\tw {{\mathcal{W}}}

\def\qmV {{\mathscr{V}}} 

\def\qmW {{\mathscr{W}}} 
\def\qmP {{\mathscr{P}}}

\DeclareMathOperator{\Pic}{Pic}

\DeclareMathOperator{\Hom}{\mathscr{H}\text{\kern -3pt {\calligra\large om}}\,} 

\newcommand{\C}{{\mathbb{C}}}

\newcommand{\Ld}{{\Lambda}^{\!\raisebox{0.5mm}{$\scriptscriptstyle
      \bullet$}\,}}

\newcommand{\bA}{\mathsf{A}}

\newcommand{\bT}{\mathsf{T}}

\newcommand{\bG}{\mathsf{G}}

\def\be{\begin{eqnarray}}
\def\ee{\end{eqnarray}}

\def\bfig{\begin{figure}[H] }
\def\efig{\end{figure}}

\def\bc{\begin{center}}
\def\ec{\end{center}}

\newtheorem{Theorem}{Theorem}
\newtheorem{Lemma}{Lemma}
\newtheorem{Proposition}{Proposition}
\newtheorem{Corollary}{Corollary}

\newtheorem{Definition}{Definition}

\def\2{{1\over 2}}

\newcommand{\rf}[1]{(\ref{#1})}

\newcommand{\ce}{\mathcal{E}}
\newcommand{\cf}{\mathcal{F}}
\newcommand{\ch}{\mathcal{H}}
\newcommand{\cl}{\mathcal{L}}

\renewcommand{\t}{\tilde}

\def\de{\delta}
\def\al{\alpha}
\def\c1m2{\sqrt{{\bf x_1}}-c \sqrt{\bf x_2}}

\def\inv{^{-1}}
\def\<{\langle}
\def\>{\rangle}
\def\+{\dagger}

\newtheorem{Thm}{Theorem}[section]

\newtheorem{Prop}[Thm]{Proposition}

\newcommand{\rd}{}

\begin{document}
\title{Baxter $Q$-operator from quantum $K$-theory}
\author{Petr P. Pushkar}
\author{Andrey V. Smirnov}
\author{Anton M. Zeitlin}
\address{\newline
Petr P. Pushkar,\newline
Perimeter Institute for \newline
Theoretical Physics, \newline
31 Caroline Street North,\newline
Waterloo, Ontario, N2L 2Y5 Canada\newline
pushkar@math.columbia.edu}
\address{\newline
Andrey  Smirnov,\newline
Department of Mathematics, \newline
The University of North Carolina \newline
at Chapel Hill,\newline
120 E Cameron Avenue,
329 Phillips Hall, \newline
Chapel Hill, NC 27599 USA;\newline
Institute for Problems of \newline
Information Transmission \newline
Bolshoy Karetny 19, Moscow 127994, Russia\newline
asmirnov@email.unc.edu}
\address{
\newline
Anton M. Zeitlin,\newline
Department of Mathematics,\newline
Louisiana State University,\newline
Baton Rouge, LA 70803 USA;\newline
IPME RAS, V.O. Bolshoj pr., 61, 199178,\newline
St. Petersburg\newline
zeitlin@lsu.edu,\newline
http://math.lsu.edu/$\sim$zeitlin \newline
http://www.ipme.ru/zam.html  }

\begin{abstract}
We define and study the quantum equivariant $K$-theory of cotangent bundles over Grassmannians.
For every tautological bundle in the $K$-theory
we define its one-parametric deformation, referred to as quantum tautological bundle. We prove that the spectrum
of operators of quantum multiplication by these quantum classes is
governed by the Bethe ansatz equations for the inhomogeneous $XXZ$ spin chain. In addition, we prove that each such operator
corresponds to the universal elements of quantum group $\mathcal{U}_{\h}(\widehat{\mathfrak{sl}}_2)$.  In particular, we identify the Baxter operator for the $XXZ$ spin chain with the operator of quantum multiplication by the exterior algebra tautological bundle. The explicit universal combinatorial formula for this operator is found. The relation between quantum line bundles and quantum dynamical Weyl group is shown.
\end{abstract}
\maketitle

\vspace{1.2in}
\tableofcontents

\section{Introduction}
\vspace{0.5in}
\subsection{Summary of results}
A deep connection between quantum integrable systems and quantum geometry, in particular quantum cohomology and $K$-theory,   was observed in the pioneering works of N. Nekrasov  and S. Shatashvili \cite{ns1}, \cite{ns2}. It was noted, that the integrable systems associated with the quantum groups studied extensively in the 1980s by the Leningrad school, describe quantum geometry of the large class of symplectic algebraic varieties. These ideas were further developed by A. Braverman, D. Maulik, A. Okounkov \cite{bmo}, \cite{mo}, \cite{os} and other authors both in mathematics and physics, e.g. \cite{gtv}, \cite{gk1}, \cite{gk2}, \cite{koroteev}, \cite{oref}, \cite{pes}.
The substantial progress in this direction shed the light on earlier papers of A. Givental, Y. P. Lee \cite{givental}, \cite{yplee} and collaborators.

The simplest nontrivial examples of such varieties are the cotangent bundles over Grassmannians $\N_{k,n}=T^*\textbf{Gr}_{k,n}$. Results of \cite{ns1},\cite{ns2} suggest that certain \textit{quantum deformations} of cohomology and  $K$-theory rings of these varieties should be related to XXX and XXZ spin chains correspondingly. These are the integrable systems described by the Yangian $Y(\mathfrak{sl}_2)$ and the quantum group $\fsh$ respectively.
Moreover, \cite{ns1}, \cite{ns2} conjecture that the operator of  multiplication by the weighted exterior algebra of the tautological bundle in such deformed $K$-theory of $\N_{k,n}$ should be related to the so-called Baxter's $Q$-operator for the $XXZ$-spin chain.

In this paper we define the quantum $K$-theory ring of $\N_{k,n}$
 using the moduli space of quasimaps\footnote{Our definition of quantum $K$-theory ring
is alternative to the one given in \cite{givental}, where the moduli space of \textit{stable maps} were used.}.
That allows us to  give a corrected formulation and a proof of the above conjecture.  In fact, we discover that to relate quantum $K$-theory with spin chains it is not enough to consider the operators of quantum multiplication by classical $K$-theory classes. It turns out that \textit{both} the multiplication in $K$-theory and the  tautological classes  should be  deformed simultaneously.

We introduce elements $\hat{\tau}(z) \in K_{\bT} (\N_{k,n})[[z]]$ which we call \textit{quantum tautological bundles}.
In the classical limit $z\to 0$ these elements coincide with the corresponding classical bundles $\tau \in K_{\bT} (\N_{k,n})$
in the standard equivariant $K$-theory. Among the main objects we study in this paper are the operators of \textit{quantum multiplication by the quantum tautological bundles}.
We show that the spectrum of these operators is described by the Bethe ansatz equations for the $XXZ$-spin chain (see Theorem \ref{thbt}).

We use the geometric action of $\fsh$ on the equivariant $K$-theory of $\N_{k,n}$ (see e.g. \cite{vasserot}) to identify it with the space of quantum states for the $XXZ$ model. Under this identification the Baxter operator \cite{baxter} of the $XXZ$ model coincides with the operator of quantum multiplication by the weighted quantum exterior algebra of the tautological bundle (see Theorem~\ref{baxther}).

We show that the quantum tautological bundles are represented by certain universal elements from $\fsh$, which depend on $z$ but do not depend on the parameters $k$ and $n$ of the Grassmannian. In particular, we explicitly find this universal element in the case of quantum exterior powers of the tautological bundle (see Theorem \ref{Bform}). Moreover, the operator of quantum multiplication by the quantum tautological line bundle $\widehat{\O(1)}(z)$ (see (\ref{qtlb}) below) acts as a lattice element of the quantum affine dynamical Weyl group (see Theorem \ref{dymath}). This observation relates quantum $K$-theory with the theory of quantum dynamical Weyl groups of P. Etingof and A. Varchenko~\cite{ev}, which is a deformation of the standard quantum Weyl group \cite{soib}.


\subsection{ \label{definit} Classical $K$-theory of $\N_{k,n}$}
Let $\textbf{Gr}_{k,n}$ denote the Grassmannian of $k$-di\-mensional subspaces in complex $n$-dimensional vector space, viewed as a complex manifold. 

Let us choose a basis in the $n$-dimensional vector space 
such that the correspodning vectors have coordinates $(y_1,\dots,y_n)$. 
We consider the $n$-dimensional torus $\bA=(\C^{\times})^n$ acting in this vector space by scaling the coordinates in the chosen basis:
$$
(y_1,\dots,y_n)\rightarrow (y_1 a_1,\dots,y_n a_n).
$$
That necessarily induces the action of $\bA$ on $k$-subspaces, i.e., on $\textbf{Gr}_{k,n}$. 
Indeed,  a $k$-subspace representing a point in  $\textbf{Gr}_{k,n}$ is fixed under this action if and only if it is spanned by $k$ basis vectors. Thus, the fixed points $\textbf{Gr}_{k,n}^{\bA}$ are in one-to-one  correspondence with the $k$-subset of the set $\{1,2,\dots,n\}$, so that each such subset corresponds to a choice of coordinate $k$-vectors from the coordinate $n$-vectors. 

We denote by $\tb$ the rank $k$ tautological vector bundle 
over $\textbf{Gr}_{k,n}$. The total space of this vector bundle is defined as follows:
$$
\tb=\{(p,v)\in \textbf{Gr}_{k,n}\times W : v\in p\},
$$
where $W=\C^n$. In other words, the fiber of $\tb$ at a point $p\in \textbf{Gr}_{k,n}$ is the $k$-dimensional subspace of $\C^n$ represented by $p$.  We denote by 
$\tw=\textbf{Gr}_{k,n}\times W$ the trivial rank $n$ bundle, such that $\tb\subset \tw$.

The main object of this paper is the cotangent bundle of
the Grassmannian: 
$$
\N_{k,n}=T^*\textbf{Gr}_{k,n}. 
$$ 
We enlarge the torus acting on this symplectic manifold to
$\bT=\bA\times \C^{\times}_{\hbar}$, so that the new one-dimensional torus $\C^{\times}_{\hbar}$ acts on $\N_{k,n}$ by scaling the cotangent directions with $\hbar$. Since this new action contracts fibers, i.e. contracts $\N_{k,n}$ to $\textbf{Gr}_{k,n}$, the fixed point set
$\N_{k,n}^{\bT}$ is the same as $\textbf{Gr}_{k,n}^{\bA}$. 
In other words, the fixed set $\N_{k,n}^{\bT}$ is a set of
$n!/k!/(n-k)!$ points labeled by $k$-subsets 
$\{i_1,\dots,i_k\}\subset \{1,\dots, n\}.$  One can also think about $\tb$ as a bundle over $\N_{k,n}$. 

We will be interested in the equivariant $K$-theory $K_{\bT}(\N_{k,n})$. It is well known that the tautological bundle $\tb$, together with all the tensorial polynomials in $\tb$, generate this group. Localization theorem in equivariant K-theory states that the classes of fixed point 
is a basis of the localized K-theory. therefore, one can think about
the localized K-theory $K_{\bT}(\N_{k,n})_{loc}$ as a 
$n!/k!/(n-k)!$-dimensional vector space spanned by the classes of fixed points,  forgetting all other structures.

The cotangent bundle $\N_{k,n}$ has a structure of a Nakajima quiver variety. This means that it can be described as certain $GL(k)$-quotient of some subspaces of matrices. Let us describe this construction here: we will need this description in Section \ref{qms} to define quasimaps to $\N_{k,n}$.

Let $R=Hom(V,W)$ for $W=\C^{n}$ and $V=\C^{k}$ with $k\leq n$. Let
$\mu^{*}: \frak{gl}(V) \rightarrow Vect(R)$ be a map of Lie algebras induced by the canonical action of $GL(V)$ on $R$. The dual of this map is known as a moment map $\mu: T^*R \rightarrow \frak{gl}(V)^{*}$. Explicitly, the moment map has the following description. Let us identify
\be \label{cotansp}
T^*R=Hom(V,W)\oplus Hom(W,V), 
\ee
where we note that the second summand in the RHS of (\ref{cotansp}) is dual to the first. Then the value of the moment map on the pair $(A,B)\in T^*R$ with $A\in Hom(V,W)$ and $B \in Hom(W,V)$ equals to:
$$
\mu(A,B)=B A. 
$$
This describtion leads to the following well-known result 
(see Example 5.3.3  and Proposition 5.3.4 in \cite{ginzb} for $n=1$):
\begin{Theorem}
The cotangent bundle over grassmannian is isomorphic to the following $GL(V)$-quotient:	
\be \label{git}
\N_{k,n}=\mu^{-1}(0)_{s}/GL(V),
\ee
where the symbol $\mu^{-1}(0)_{s}$ denotes the intersection of the set $\mu^{-1}(0) \subset T^{*} R$ with the stable locus corresponding to injective elements in $R$:
\be \label{stpoints}
\textrm{stable points in $T^*R$} = \{ (A,B): rank(A)=k \}.
\ee
\end{Theorem} 

As a Nakajima quiver variety $\N_{k,n}$ is a symplectic resolution,
i.e. it comes with natural projective morphism to affine variety, Theorem 5.2.2 in \cite{ginzb}
\be \label{affineN}
\N_{k,n}\rightarrow \N^{0}_{k,n}:={\rm{Spec}}\Big(\C[\mu^{-1}(0)]^{GL(V)}\Big).
\ee

Now we give the description of fixed points on $\N_{k,n}$, tautological bundles, torus action and equivariant K-theory once again, this time from the perspective of Nakajima varieties.

We set a framing torus $\bA=\C^{\times}_{a_1}\times\cdots\times \C^{\times}_{a_n}$ to be an $n$-torus acting on $W$ by scaling the coordinates with characters $a_i$.  Let $\C^{\times}_{\hbar}$ be a one-torus acting on $T^*R$ by scaling the cotangent directions with character $\hbar$. We adopt the notation  $\bT=\bA\times \C^{\times}_{\hbar}$.

The action of $\bT$ on $T^*R$ induces its action on $\N_{k,n}$. The fixed set $\N_{k,n}^{\bT}$ consists of $n!/k!/(n-k)!$ isolated points representing the $k$-planes spanned by coordinate vectors. They are conveniently labeled by $k$-subsets
$\fp=\{x_1,\cdots,x_{k}\}\subset \{a_1,\cdots, a_n\}$.

We note that $\N_{k,n}$ is naturally equipped with the following tautological bundles:
\be \label{tautbun}
\tb=\mu^{-1}(0)_{s} \times V /GL(V), \ \ \tw= \mu^{-1}(0)_{s}\times W/GL(V).
\ee
Since $GL(V)$ does not act on $W$  the bundle $\tw$ is trivial, and because $A$ is injective we have $V\subset W$ and thus $\tb \subset \tw$.

Let $K_{GL(V)\times \bT}(pt) = \Z[s_1^{\pm 1},s^{\pm 1}_2,\cdots,s_k^{\pm 1},a_1^{\pm 1},\dots, a_{n}^{\pm},\hbar^{\pm}]^{\frak{S}_k}$ be the ring of symmetric Laurent polynomials in $k$ variables with coefficients in $K_{\bT}(pt)$.
For every such polynomial we denote by the same symbol $\tau \in K_{\bT}(\N_{k,n})$ the corresponding Schur functor of $\tb$.\footnote{For example, the polynomial
	$$\tau(s_1,\cdots,s_k)=(s_1+\cdots + s_k)^2 - \sum_{1\leq i_1<i_2<i_3\leq k} s_{i_1}^{-1} s_{i_2}^{-1} s_{i_3}^{-1}$$
	corresponds to $\tau(V) = V^{\otimes 2} - \Lambda^{3} V^{*}$.} 
This $K$-theory classes  $\tau$ can be uniquely represented by the symmetric Laurent polynomials in the corresponding Chern roots of $\tb$ and thus there should be no confusion in our notations.

Let us set the following notation for the disjoint union of $\N_{k,n}$ for all $k$:
$$\N(n)=\coprod\limits_{k=0}^{n}\,\N_{k,n},$$  
so that the fixed point set $\N(n)^{\bT}$ consists of total $2^n$ points.

The equivariant $K$-theory $K_{\bT}(\N(n))$ is a module over the ring of equivariant constants: $R=K_{\bT}(\cdot)=\Z[a_1^{\pm },\cdots,a_n^{\pm 1},\hbar^{\pm 1}]$. The localized $K$-theory
\be
\label{Kdec}
K_{\bT}(\N(n))_{loc}=K_{\bT}(\N(n)) \bigotimes\limits_{R} {\mathcal{A}} = \bigoplus_{k=0}^{n} K_{\bT}(\N_{k,n}) \bigotimes\limits_{R} {\mathcal{A}}
\ee
is an ${\mathcal{A}}$-vector space  (${\mathcal{A}}=\Q(a_1,\cdots,a_n,\hbar)$) of dimension $2^n$ spanned by the $K$-theory classes of fixed points $\O_{\fp}$. If $\tau$ is a $K$-theory class then the operation 
of tensor multiplication by $\tau$ (i.e., $\gamma \to \tau \otimes \gamma$) is an $\mathcal{A}$-linear operator
acting in the vector space~(\ref{Kdec}).  These operators are diagonal in the basis of fixed points:
\be
\label{clev}
\tau \otimes \O_{\fp} =\tau(a_{i_1},\cdots, a_{i_k}) \, \O_{\fp} \ \ \ \textrm{for}  \ \ \ \fp=\{i_{1},\cdots, i_{k}\}\subset \{1,\cdots, n\}.
\ee
Note, that the eigenvalue of the operator of multiplication by the class $\tau$ is given by the value of the symmetric polynomial representing this class at the corresponding fixed point, i.e. $\tau(s_1=a_{i_1},\dots,s_k=a_{i_k}) \in \mathcal{A}$. This statement can be conveniently formulated for all fixed points and tautological bundles simultaneously.
\begin{Proposition}
\label{pro1}
The eigenvalues of the operators of multiplication by tautological bundles on ${K_{\bT}(\N_{k,n})}_{loc}$ are given by the values of the corresponding Laurent polynomials $\tau(s_1,\cdots,s_k)$ evaluated at the solutions of the following equations:
\be
\label{clbeth}
\prod\limits_{j=1}^{n}(s_i-a_j) =0,  \ \ \ i=1\cdots k
\ee
with $s_i\neq s_j$.
\end{Proposition}
 The solutions of (\ref{clbeth}) with $s_i\neq s_j$  are in one-to-one correspondence with the
$k$-subsets $\{a_{i_1},\cdots,a_{i_k}\}\subset \{a_{1},\cdots, a_{n}\}$ and, therefore, with the set of the fixed points $\N_{k,n}^{\bT}$. 
Theorem \ref{thbt} provides  an elegant generalization of this statement to the case of the  quantum $K$-theory.
The system of equations (\ref{clbeth}) turns out to be the classical limit ($z\rightarrow 0$) of the so-called Bethe ansatz equations~(\ref{beth}).

\subsection{Quantum $K$-theory and Bethe ansatz}
In Section 2 we use the moduli space of \textit{quasimaps} to $\N_{k,n}$ to define certain associative, commutative, one-parametric deformation of its equivariant $K$-theory ring. We denote the deformed tensor product by $\qotimes$ and call the corresponding ring \textit{quantum K-theory} of $\N_{k,n}$. The word ``deformation'' here means that for two $K$-theory classes $A, B$  we have
$$
A \qotimes B = A {\otimes} B + \sum\limits_{d=1}^{\infty}  (A{\qotimes}_d B) \, z^d,
$$
so that if the deformation parameter is equal to zero $z\to0$ (this special case is usually referred to as \textit{classical limit}), the quantum product $\qotimes$ coincides with the classical tensor product $\otimes$. The definition of the quantum product follows closely the definition of the product in quantum cohomology: the classes $A \qotimes_d B\in K_{\bT}(\N(n))$ (quantum corrections) are given by certain  degree $d$ curve counting in $\N_{k,n}$.

Next, in Section \ref{locsec}, (Definition \ref{qtbdef}), for a tautological bundle $\tau \in K_{\bT}(N_{k,n})$ as above, we define a deformation which will be referred to as  \textit{quantum tautological bundle}:\footnote{To the best of our knowledge, this object is introduced in the present paper for the first time.}
$$
\hat{\tau}(z) = \tau+\sum\limits_{d>0}^{\infty} \,\tau_{d} z^d \in K_{\bT}(\N_{k,n})[[z]]
$$
One of the goals of this paper is to study the spectrum of operators of \textit{quantum} multiplication
by \textit{quantum} tautological bundles. The following theorem is the generalization of Proposition \ref{pro1} to the quantum level.
\begin{Theorem}
\label{thbt}
The eigenvalues of operators of quantum multiplication by $\hat{\tau}(z)$ are given by the values of the corresponding Laurent polynomials
$\tau(s_1,\cdots,s_k)$ evaluated at the solutions of the following equations:
\be
\label{beth}
\begin{array}{c}
\ \ \ \prod\limits_{j=1}^{n} \dfrac{s_i-a_j}{\hbar a_j-s_i}=z \,\hbar^{-n/2}\prod\limits_{{j=1}\atop{j\neq i}}^{k}\, \dfrac{s_i \hbar - s_j}{s_i-s_j\hbar} \, , \ \  i=1\cdots k.   \ \ \ \\
\end{array}  \ \ \
\ee
\end{Theorem}
\noindent
When $z=0$ we obtain the statement of  Proposition \ref{pro1}.

\subsection{XXZ model and Baxter $Q$-operator } A specialist can immediately recognize that (\ref{beth}) are nothing but the \textit{Bethe ansatz} equations for the so-called $XXZ$ spin chain. Let us briefly recall some basic facts about this quantum integrable system, see also \cite{resh}, \cite{kor} for a more detailed outline.

Let us consider a system of $n$ interacting magnetic dipoles (usually refered to as \textit{spins}) on a $1$-dimensional periodic lattice. Each spin  can take two possible configurations ``up'' and ``down'', such that the space of the quantum states of this system has dimension
$2^n$:
\be
\label{qsp}
{\mathcal{H}} = \C^2\otimes\C^2 \otimes \cdots \otimes \C^2.
\ee
In this system of spins only the neighboring ones (with labels $i$ and $i+1$) can interact. The energy of the interaction
is described by the following Hamiltonian:
\be
\label{XXZham}
H_{2}=-\sum\limits_{i=1}^{n}\, \sigma^{i}_x \otimes \sigma^{i+1}_x +\sigma^{i}_y \otimes \sigma^{i+1}_y +\Delta \, \sigma^{i}_z \otimes \sigma^{i+1}_z, \ \ \
\ee
where $\Delta=\hbar^{1/2}+\hbar^{-1/2}$  is the parameter of anisotropy and $\sigma^{i}_m$ are the standard Pauli matrices acting in the $i$-th factor of (\ref{qsp}).
The periodic boundary conditions
are imposed by identifying the first with $(n+1)$-th spin space. Up to a gauge transformation such identification is given by a diagonal matrix. Modulo an irrelevant scalar this matrix can be chosen to be in the following form:
$$
\left(\begin{array}{cc}
z&0\\
0&1
\end{array}\right):\ \ \C^2_{(1)} \rightarrow \C^2_{(n+1)}.
$$
This free parameter $z$, defining the periodic boundary condition will play the crucial role in this paper, namely it  is the parameter of deformation in the quantum $K$-theory.

The quantum system of spins governed by the Hamiltonian (\ref{XXZham}) is called the quantum $XXZ$ spin chain. The most important feature of this  model is its \textit{integrability}, which implies the existence of mutually commuting higher Hamiltonians $H_{n}, \ \ n=1\cdots \infty$. For example:
$$
S_z=\sum\limits_{i=1}^{n} \, \sigma^{i}_z \ \ \
$$
is the operator of total spin commuting with  (\ref{XXZham}). This operator provides the grading on the space of states:
$$
{\mathcal{H}} = \bigoplus\limits_{k=0}^{n}\, {\mathcal{H}}_k, \ \ \ {\mathcal{H}}_k= \{ v\in {\mathcal{H}} : S_z(v) = (n-2k) v \}.
$$
Obviously $\dim {\mathcal{H}}_k = n!/(k!(n-k)!)$ thus this graded sum can be identified with (\ref{Kdec}).

The Hamiltonians $H_k$ can be obtained from the asymptotic expansion of the logarithm of the {\it transfer matrix} $\mathrm{T}(x)\in {\rm End}(\mathcal{H})(z)[x]$. An important object in the study of XXZ spin chain is the so-called {\it Baxter Q-operator} $\mathrm{Q}(x)\in {\rm End}(\mathcal{H})(z)[x]$. Its eigenvalues $Q(x)$ and the eigenvalues of the transfer matrix $T(x)$ obey the following functional relation, known as TQ-relation:
\begin{eqnarray} \nonumber
{T}(x){Q}(x)=\alpha(z,x){Q}(\h x)+\delta(z,x){Q}(\h^{-1}x), 
\end{eqnarray}
where $\alpha(z,x), \delta(z,x)$ are certain rational functions. 
The  expansion of the Baxter Q-operator
$$
{\rm Q}(x)=\sum\limits_{i=0}^{\infty} x^i H^{nloc}_i
$$
produces nonlocal Hamiltonians $H^{nloc}_i$. The  original ``physical'' Hamiltonians $H_i$ of the $XXZ$ model are related to nonlocal Hamiltonians $H^{nloc}_i$ via the TQ relation. 
The physical problem is to find the joint spectrum of $H_{i}$ (equivalently, $H^{nloc}_i$). 
The solution is given by:
\begin{Theorem} (see e.g.\cite{resh})
\label{xxzdes}
The eigenvalues of the Baxter's operator ${\rm Q}(x)$ are given by $\prod\limits_{i=0}^{k}(1+x s_i)$ where
$s_i$ are the solutions of the Bethe equations (\ref{beth}).
\end{Theorem}
This means that the eigenvalues of $H^{nloc}_i$ are given by the values of $k$-th elementary symmetric function, evaluated
on the solutions of Bethe equations. In view of Theorem \ref{thbt} these are the same as the eigenvalues of operator of the quantum multiplication by the quantum $i$-th exterior power $\widehat{\Lambda^{\! i} \tb}(z)$ in the quantum $K$-theory of $\N_{k,n}$.

\subsection{Quantum group structure of the $XXZ$ model and $K$-theory of $\N_{k,n}$\label{qgst}}
We see that the spectrum of quantum tautological bundles in quantum $K$-theory of $\N_{k,n}$ coincides with spectrum of observables in the $XXZ$ model. Let us explain the connection between the quantum physics of the $XXZ$ model and quantum geometry of $\N(n)$.
It is well known that the symmetries of the $XXZ$ spin chain are described by the \textit{quantum affine group} $\fsh$ (see e.g. \cite{cp}, \cite{resh}). In particular, the Hilbert space of quantum states of the $XXZ$ model is an irreducible $\fsh$-module:
\be
\label{hilbsp}
{\mathcal{H}}_{XXZ} = \C^{2}(a_1)\otimes \C^2(a_2) \otimes \cdots \otimes \C^2(a_n),
\ee
where $\C^{2}(a_i)$ are the two-dimensional evaluation representations of $\fsh$ \cite{cp}.

It was shown in \cite{vasserot} (see also \cite{os} for an alternative construction) that there is a natural action of $\fsh$ on the equivariant $K$-theory $K_{\bT}(\N(n))$. In short, as a $\fsh$-module the $K$-theory of $\N(n)$ is isomorphic to
the Hilbert space of the $XXZ$ spin chain: $K_{\bT}(\N(n)) = {\mathcal{H}}_{XXZ}$. 
The evaluation parameters of representation $a_i$ and the ``Planck constant''  $\hbar$  are the counterparts of the equivariant characters of the torus $\bT$ on the side of equivariant $K$-theory. 
We refer to the section 5 for details regarding this identification. 
 
Our claim is that the ring of quantum tautological bundles coincides with the  ring of the nonlocal Hamiltonians of the $XXZ$ spin chain under this isomorphism. Namely, let us consider the $K$-theory class of $x$-weighted exterior algebra of $\tb$:
$$
\Lambda_x = \bigoplus_{m=0}^{\infty}\, x^m \Lambda^{\!m}( \tb) =\Lambda^{\! \bullet}_{x} \tb.
$$
\begin{Theorem} \label{baxther}
The operator of quantum multiplication by the quantum tautological bundle $\hat{\Lambda}_x(z)$ coincides with the Baxter operator for  the $XXZ$ spin chain
under the identification $K_{\bT}(\N(n)) = {\mathcal{H}}_{XXZ}$ as $\fsh$-modules:
$$
\hat{\Lambda}_x(z)\circledast\cdot ={\rm Q}(x).
$$
\end{Theorem}

\subsection{Universal formula and the dynamical quantum affine Weyl group}
The operators of quantum multiplication by  $\hat{\tau}(z)$  satisfy the following Theorem.
\begin{Theorem}
The operators of quantum multiplication by $\hat{\tau}(z)$ are the universal elements in  $\fsh[[z]]$, i.e. they do not depend on $k,n$ defining Nakajima variety. 
\end{Theorem}

As we explained above, the equivariant $K$-theory $K_{\bT}(\N(n))$ is a natural $\fsh$ module (\ref{hilbsp}).
By evaluating the universal element $\hat{\tau}(z)$ in $K_{\bT}(\N(n))$ we obtain the operator of quantum multiplication
by $\hat{\tau}(z)$ in this representation. Universality in particular means that $\hat{\tau}(z) \in \fsh[[z]]$ does not depend on $n$, i.e. on the choice of representation.
Moreover, one can check that the operators of quantum multiplication by
 \textit{classical} tautological classes $\tau$ do not have such a universal representation. This is an indication  that the quantum bundles $\hat{\tau}(z)$ are more natural objects in the context of quantum $K$-theory.

The proof of the universality theorem for the generic operators,
follows immediately from the appropriate formulas for the operators of quantum multiplication by the quantum exterior powers $\widehat{\Lambda^{\! i}\tb}(z)$ (see Section 5.7).

Let $E_r, F_r, H_r, K$ be the standard Drinfeld's generators of $\fsh$, then the following theorem holds.
\begin{Theorem}
For arbitrary $n$ and $k$ the operator of quantum multiplication by the quantum $l$-th exterior power of tautological bundle
is given by the following universal formula:
\label{Bform}
$$
\widehat{\Lambda^{\! l} \tb}(z)\circledast \cdot = \Lambda^{\! l} \tb + a_1(z) \, F_0 \Lambda^{\! l-1} \tb E_{-1}+ a_2(z) \, F_{0}^2 \Lambda^{\! l-2} \tb E_{-1}^2 + \cdots
+ a_l(z) F_{0}^l  E_{-1}^l
$$
with
$$
a_m(z)= \frac{(\h-1)^m\ \h^{\frac{m(m+1)}{2}} K^{m}}{(m)_{\h}!\prod\limits_{i=1}^
{m}(1-(-1)^nz^{-1}\h^{i} K)}, \ \ \ \textrm{for}  \ \ \ (m)_{\h}=\dfrac{1-\h^m}{1-\h}, \ \ (m)_{\h}!=(1)_{\h}\cdots (m)_\h.
\nonumber
$$
\end{Theorem}
Here $\widehat{\Lambda^{\! i} \tb}(z)$  ($\Lambda^{\! i} \tb $)  stands for the operators of quantum (classical) multiplication
by the quantum (classical) exterious powers.  Note, that
in the classical limit $z\to 0$ all $a_l(z)$ vanish and we obtain $\widehat{\Lambda^{\! l} \tb}(0)= \Lambda^{\! l} \tb$.

It is interesting to specialize this theorem to the case of the top exterior power: $\det(\tb)= \O(1)$.
It is well-known that the operator of multiplication by the line bundle
${\mathcal{O}}(1)$ in  the  classical equivariant $K$-theory of $\mathbf{N}_{k,n}$ corresponds to the lattice part of quantum affine Weyl group of $\fsh$, see e.g. \cite{os}.
In particular, it acts on Drinfeld's generators by:
$
{\O}(1) F_{m} {\O}(1)^{-1} =F_{m-1}, \ \  {\O}(1) E_{m} {\O}(1)^{-1} =E_{m+1}.
$
Using these equations and Theorem \ref{Bform} one obtains:
$$
\widehat{\mathcal{O}(1)}(z) \circledast \cdot= B(z)\, {\mathcal{O}}(1), \ \ \ B(z)=\sum\limits_{m=0}^{\infty} \dfrac{\hbar^{m(m+1)/2} (\hbar-1)^m K^m }{(m)_{\h}!\prod\limits_{i=1}^{m}(1-(-1)^nz^{-1} K \hbar^i)} F_{0}^{m} E_{0}^{m},
$$
where $\widehat{\mathcal{O}(1)}(z)$ is a quantum tautological bundle corresponding to the top exterior power tautological bundle, i.e. line bundle ${\O}(1)$ on $\mathbf{N}_{k,n}$. 
One should not worry with the infinite number of terms in the expression for $B(z)$, since we are dealing with the finite-dimensional representations of  $\fsh$ and thus $E_0, F_0$ act as nilpotent operators.

Up to a shift in the notations $B(z)$ coincides with the
element of the universal enveloping algebra, defining the action of the quantum affine dynamical Weyl group $QW_{\fsh}$ (see Proposition 14 in \cite{ev}).
The dynamical parameter $e^{\lambda}$ of the dynamical Weyl group in \cite{ev} is identified with the parameter of quantum deformation $z$
in quantum $K$-theory. We conclude with the following result:
\begin{Theorem}
\label{dymath}
The lattice element of the quantum dynamical Weyl group $QW_{\fsh}$ acts on $K_{\bT}(\N(n))$ as the operator of quantum multiplication
by the quantum line bundle $\widehat{\mathcal{O}(1)}(z)$.
\end{Theorem}

\subsection{Correspondence table}

In short, the physics of the  $XXZ$ spin chain, quantum $K$-theory of $\N(n)$ and representation theory
of $\fsh$ are different languages describing the same object. The following table is a dictionary:

\begin{small}
$$
\begin{array}{|c|c|c|}
\hline
&&\\
XXZ-\textrm{spin} \ \ \textrm{chain} & \textrm{Geometry }\ \  \textrm{of} \ \  \N_{k,n} &  \textrm{Representation} \ \  \\
&&\textrm{theory} \ \ \textrm{of} \ \ \fsh\\
\hline
&& \\
{\mathcal{H}}_{XXZ}&  K_{\bT}(\N(n)) &  \bigotimes^n_{i=1}\C^{2}(a_i) \\ &&
\\ \hline &&
\\
\textrm{inhomogeneity} & \textrm{equivariant} & \textrm{evaluation}  \ \ \textrm{module}
\\
\textrm{parameters} \ \ a_i& \textrm{characters}  \ \ a_i & \textrm{parameters} \ \ a_i\\
&&\\
\hline && \\
\textrm{anisotropy \ \ parameter }\ \  & \hbar= {\bT}   & \hbar^{1/2} - \textrm{parameter}   \\
\Delta = \hbar^{1/2} +\hbar^{-1/2} &  \textrm{weight \ \ of \ \ symplectic \ \ form}& \textrm{of \ \ the \ \ quantum \ \ group} \\
&& \\
\hline
&&\\
\textrm{Transfer \ \ matrices,} & \textrm{generating \ \ function \ \ for} & \textrm{ weighted \ \ partial \ \ traces}   \\
\textrm{Baxter} \ \ \mathcal{Q}-\textrm{operators}  & \ \  \textrm{quantum \ \  tautological \ \ bundles} & \textrm{of} \ \ \textrm{of} \ \ R-\textrm{matrices} \\
&&
\\
\hline &&\\
z-\mathrm{parameter \ \ of}&z- \mathrm{parameter \ \ of} &z - \textrm{ parameter of } \\
\mathrm{\ \ boundary \ \ condition}  &\mathrm{  \ \  quantum \ \ deformation} &\textrm {weight \ \ in \ \ the \ \ trace } \\
&& \\
\hline
\end{array}
$$
 \end{small}
\vspace*{3mm}

\noindent
\textbf{Acknowledgements.}
First of all we would like to thank Andrei Okounkov for many valuable discussions, ispiration and teaching us his understanding of geometric representation theory.
We are also grateful to M. Aganagic, D. Galakhov, M. McBreen, A. Negut, V. Toledano-Laredo and N.Yu. Reshetikhin for their interest and comments.
The work of A.V. Smirnov was supported in part by RFBR grants 15-31-20484
mol-a-ved and RFBR 15-02-04175. The work of A.M. Zeitlin was partially supported by AMS Simons travel grant and Simons Collaboration Grant, Award ID: 578501. Part of this work was done during Simons Summer Workshop 2016.

\section{Quasimaps to $\N_{k,n}$, vertex functions and diagrammatic notation \label{qms}}

This section is a pedestrian exposition of the ideas developed  in \cite{ionut,lectok}.  Most of the results of this section are given without a proof and the reference to the corresponding theorem in \cite{ionut} or \cite{lectok} is used instead.    

The material presented is primarily aimed towards the audience of quantum 
integrable systems theorists wishing to familiarize themselves with the basic 
notions of enumerative geometry.

\subsection{Overview of enumerative geometry}

Let  $\mathcal{C}$  be a complex curve. We denote a ``space of maps'' from $\mathcal{C}$ to a variety $X$ (for us, this variety will be the cotangent bundle of the Grassmannian $X=\N_{k,n}$) as follows:
\be \label{msp}
\mathcal{M}^d=\{ \textrm{maps of degree $d$}: \ \ \mathcal{C}\rightarrow X \}, \ \ \ d\in H_{2}(X,\Z).
\ee
The degree of the map $f$ is defined as $[Im(f)]=d\in H_{2}(X,\Z)$. 
In general, it is not easy to define $\mathcal{M}^d$, and several versions 
of such  moduli spaces exist. Let us postpone these questions for a moment,
and assume we are given some well-defined moduli space~(\ref{msp}).

For a point $p$ on a curve $\mathcal{C}$ there exists an evaluation map:
$$
\textrm{ev}_p : \mathcal{M}^d \rightarrow X,
$$
which sends a map $f\in \mathcal{M}^d$ to its value at the point $p$:
$$
\textrm{ev}_p: f\mapsto f(p) \in X.
$$ 
More generally, for a collection of points $p_1,\dots,p_m\in \mathcal{C}$
we have a set of evaluation maps:
\be \label{evmaps}
\textrm{ev}_{p_1}\times \dots \times \textrm{ev}_{p_m}  :\, \mathcal{M}^d\rightarrow X\times\dots\times X.
\ee
In enumerative geometry one studies the maps (\ref{evmaps}) on the level of cohomology or K-theory. 
For example, (\ref{evmaps}) induces the maps of the corresponding K-theories:
\be \label{eqmap}
\textrm{ev}_{p_1 *}\times \dots \times \textrm{ev}_{p_m *} :\, K(\mathcal{M}^d)\rightarrow K(X)\otimes \dots \otimes K(X).
\ee  
For the pushforward maps $\textrm{ev}_{p_i *}$ to be well-defined, the moduli space $\mathcal{M}^d$ has to be compact in the appropriate sense. The operation of pushforward is the analog of integration in cohomology, and the problem is that we may not always integrate differential forms over non-compact spaces as those integrals could be divergent. In practice, the $K$-theoretic pushforwards (\ref{eqmap}) are well-defined if the maps
$\textrm{ev}_{p_i}$ are \textit{proper}. But again, for a moment let us assume that $\textrm{ev}_{p_i}$ are proper and all the maps $\textrm{ev}_{p_i *}$ exist.

For simplicity, let us denote $K(X)\equiv L$ and assume that $L$ is a finite-dimensional vector space. Then, for every natural class $\tau \in K(\mathcal{M}^d)$ such as, e.g. the structure sheaf of $\mathcal{M}^d$, one obtains a rank  $m$ tensor:
$$
\textrm{ev}_{p_1 *}\times \dots \times \textrm{ev}_{p_m *}(\tau) \in L^{\otimes m}.
$$
To keep track of the degree $d$ we will organize them into a power series
\be \label{tensors}
V_{p_1,\dots,p_m}(z) = \sum\limits_{d} z^d \textrm{ev}_{p_1 *}\times \dots \times \textrm{ev}_{p_m *}(\tau) \in L^{\otimes m}[[z]].
\ee
Here $z$ stands for a collection $\{z_i\}$ of the formal parameters of the generating functions, known as K\"ahler parameters, and we assume the notation $z^d\equiv\prod_i z_i^{d_i}$, so that $d_i$ are the components of $d \in H_{2}(X,\Z)$. 

By definition, 
$V_{p_1,\dots,p_m}(z)$ are the $m$-tensors whose components are power series in $z$. The study of such tensors is the main subject of enumerative geometry. 
When $X$ and $\mathcal{M}^d$ are equipped with an action of a torus $\bT$
one can study the $\bT$-equivariant version of the map (\ref{eqmap}). In this case
the coefficients of power series $V_{p_1,\dots,p_m}(z)$ become some interesting rational functions of coordinates on $\bT$, which we refer to as {\it equivariant parameters}. We will see some explicit examples of such tensors below.

As we discuss in the Section \ref{gsec}, in the case $X=\N_{k,n}$ the corresponding vector space $L$ is the $n!/(k!(n-k)!)$-dimensional subspace of the Hilbert space of XXZ spin chain spanned by configurations with exactly $k$-spins up:
$$
L=\textrm{weight $k$ subspace of} \ \ \ \C^2(a_1)\otimes \dots \otimes \C^2(a_n).
$$ 
In this sense, the enumerative geometry of $\N_{k,n}$ provides a natural set of tensors for these subspaces of the XXZ spin models. The goal of this paper is to explain that {\bf
	most important objects in the theory of XXZ spin chains appear in this way}. As we shall see in the following,  the examples of such ``tensors'' include Baxter operators (more generally, all Hamiltonians of the XXZ model), solutions of quantum Knizhnik-Zamolodchikov equations and other standard objects in the theory of quantum integrable systems. 

The appropriate version of the moduli space $\mathcal{M}$, which makes contact to representation theory, is the moduli space of quasimaps introduced in \cite{lectok}. In the next sections we collect the known results about these moduli spaces and corresponding tensors (\ref{tensors}).


\subsection{Quasimaps}

Let us fix a rational curve $\mathcal{D}\cong \mathbb{P}^1$ and a set of distinct points $p_1,\dots,p_m \in \mathcal{D}$. The following is the genus zero version of Definition 7.2.1 in \cite{ionut}  for the target $\N_{k,n}$. 

\begin{Definition} \label{qmdef}
	A stable, genus zero, quasimap to $\N_{n,k}$ relative to $p_1,\cdots,p_m$  is given by the following data
	$$
	(\mathcal{C},p_1',\dots,p_m',P,f, \pi),
	$$	
	where
	\begin{itemize}
		\item $\mathcal{C}$ is a connected, at most nodal genus zero projective curve and $p_i'$ are nonsingular points of $\mathcal{C}$,
		\item $P$ is a principal $GL(k)$ - bundle over $\mathcal{C}$,
		
		\item f is a section of the fiber bundle 
		\be \label{qm}
		{\rm p} : P\times_{{GL(k)}} (R \oplus R^{*})  \rightarrow \mathcal{C}
		\ee
		over $\mathcal{C}$ satisfying $\mu =0$, where $R=Hom(V,W)$ - is a representation of $GL(k)$
		defined in Section \ref{definit},   
		\item $\pi: \mathcal{C} \rightarrow \mathcal{D}$ is a regular map,
		\end{itemize}
		satisfying the following conditions:
		\begin{enumerate}
		\item There is a distinguished component $\mathcal{C}_0$ of $\mathcal{C}$ such that 
		$\pi$ 	restricts to an isomorphism $\pi:\mathcal{C}_0 \cong \mathcal{D} $ and $\pi(\mathcal{C}\setminus \mathcal{C}_{0})$ is zero-dimensional (possibly empty). 
		
		\item $\pi(p_i')=p_i$.
		
		\item $f(p)$ is stable in the sense of (\ref{stpoints}) for all $p\in \mathcal{C}\setminus B$ where $B$ is a finite (possibly empty) subset of $\mathcal{C}$. 
		
		\item The set $B$ is disjoint from the nodes and points $p'_1,\dots, p'_m$. 
		
		\item $\omega_{\tilde{\mathcal{C}}}( \sum_i p'_i + \sum_j q_i) \otimes {\mathcal{L}}_{\theta}^{\epsilon}$ is ample for every rational $\epsilon >0$, where ${\mathcal{L}}_{\theta}=P\times_{{GL(k)}} \C_{\theta}$ ($\theta=\det$ is the character of ${GL(k)}$),
		$\tilde{\mathcal{C}}$ is the closure of $\mathcal{C}\setminus \mathcal{C}_0$ and $q_i$ are the nodes  $\mathcal{C}_0\cap \tilde{\mathcal{C}}$. 
		\end{enumerate}
	
\end{Definition}
{ \rd
	We say that a curve $C$ is a {\it chain of rational curves} if it is a union of components
	$$
	C=C_1\cup \dots \cup C_{l} 
	$$
	so that $C_{i}\cong\mathbb{P}^{1}$ and the intersection of neighboring components consist of one point $C_{i}\cap C_{i+1}=pt$ where $pt$ is a nodal singularity of $C$. Other components do not have common points.

	The conditions (1)-(5) of Definition \ref{qmdef} mean that the curve $\mathcal{C}$ is  a union of the distinguished component $\mathcal{C}_{0}$ and chains of rational curves attached to $\mathcal{C}_{0}$ at the points $p_{i} \in \mathcal{C}_{0}$ via a nodal singularity. The projection $\pi$ collapses the chain of rational curves attached to a point $p_i$ to the same point $p_{i}$. Away from the points $p_{i}$ on $\mathcal{D}$ the projection $\pi$ is an isomorphism:
	$$
	\pi^{-1}(\mathcal{D}\setminus \{p_1,\dots,p_m\}) \cong \mathcal{D}\setminus \{p_1,\dots,p_m\}
	$$
	 In particular, if $r\in \mathcal{D}$ is not a  relative point, i.e.,  $r\not \in \{p_1,\dots,p_m\}$ then $\pi^{-1}(r)$ consist of a single point which we  identify with $r$.  
	
	Finally,  condition (5) implies also that the total number of special points on each component of $\mathcal{C}$ (including the nodes) is at least two.  
	This means that the point $p_i'$ is located on the last component of the chain of rational curves attached to $p_i$, see Fig.~\ref{rtch}. 
	
	\begin{figure}[h!]
		\centering
		\includegraphics[width=8.5cm]{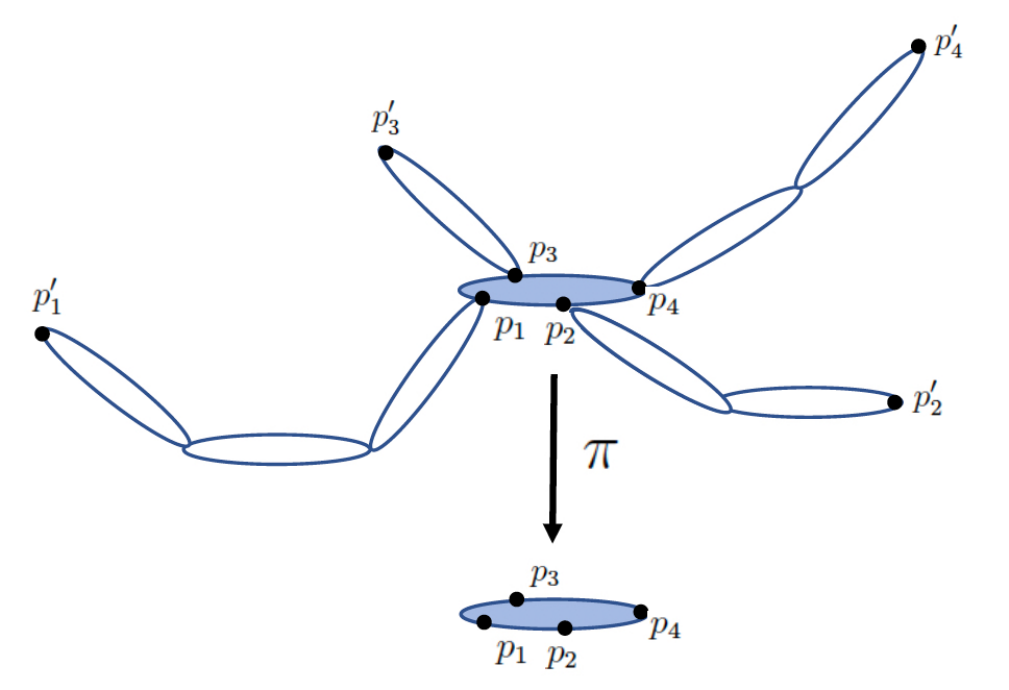}
		\caption{\label{rtch} A  example of parametrizes component with 4 relative points $p_1,p_2,p_3,p_4$ and chains of rational curves attached to it.}
	\end{figure}
	
}

Given a quasimap, we will often refer to $\mathcal{C}_0$ (or $\mathcal{D}$) as \textit{parametrized component} and to points $p_1,\dots,p_m$ as {\it relative} points.  It is possible to upgrade this definition to more general setting replacing $\mathcal{D}$ with an arbitrary curve or even a single point.

Let $(\mathcal{C},p_1',\dots,p_m',P,f, \pi)$ be a quasimap and let
$V$ be the $k$-dimensional representation of 
${GL(k)}$ as in Section \ref{definit}. Let us denote by 
\be \label{qmvdef}
\mathscr{V}=P\times_{{GL(k)}} V \rightarrow \mathcal{C}
\ee the associated rank $k$ vector bundle over $\mathcal{C}$. 

\begin{Definition} \label{defdegree}
	The degree of a quasimap $(\mathcal{C},p_1',\dots,p_m',P,f, \pi)$ is the degree the vector bundle 	$\mathscr{V}$ associated to it. 
\end{Definition}

\begin{Definition}
	Let $\qm^{d}_{{\rm{relative}}, p_1,\cdots,p_m}$ denote the stack parameterizing stable genus zero quasimaps relative to $p_1,\dots,p_m$, (i.e. the data of Definition \ref{qmdef}) of fixed degree $d$.   	
\end{Definition}

The following result is proved in \cite{ionut} (Theorem 7.2.2):
\begin{Theorem} 
	The stack $\qm^{d}_{{\rm{relative}}, p_1,\cdots,p_m}$ is a  Deligne-Mumford stack of finite type with  perfect obstruction theory. The stack $\qm^{d}_{{\rm{relative}}, p_1,\cdots,p_m}$ is proper over affine scheme $\N^0_{k,n}$ (defined in (\ref{affineN})). 	
\end{Theorem}

Let $(\mathcal{C},p_1',\dots,p_m',P,f, \pi)$ be a quasimap. By definition, $f(p_i')$ is a stable element of $R\oplus R^{*}$ for all $i=1,\dots,m$. In  particular, the corresponding ${GL(k)}$-orbit $[f(p_i')]$ is a well-defined point in 
$\N_{k,n}$.

\begin{Definition} \label{evmapsdef}
	The evaluation maps ${\rm{ev}}_{p_i}: \qm^{d}_{{\rm{relative}}, p_1,\cdots,p_m}\rightarrow \N_{k,n}$, $i=1,\dots,m$ are defined by 
	\be \label{proevs}
	{\rm{ev}}_{p_i}: (\mathcal{C},p_1',\dots,p_m',P,f, \pi)\mapsto [f(p_i')],
	\ee
	where $[f(p_i')]$ denotes  a point in $\N_{k,n}$ representing the class of ${GL(k)}$-orbit of $f(p_i')$. 
\end{Definition}

The properness of $\qm^{d}_{{\rm{relative}}, p_1,\cdots,p_m}$ over  $\N^0_{k,n}$ 
provides the following important result.
\begin{Corollary}
	The evaluation maps ${\rm{ev}}_{p_i}$, $i=1,\dots,m$  are proper. 	
\end{Corollary}
\begin{proof}
	Let us consider the following diagram: 
	\begin{center}
		\vspace{5mm}	
		\begin{tikzpicture}[node distance =5.1em]
		\node (rel) at (2.5,1.5) {$\N_{k,n}$};
		\node (nonsing) at (0,0) {$\qm^{d}_{{\rm{relative}}, p_1,\cdots,p_m}$};
		\node (X) at (5,0) {$\N^{0}_{k,n}$};
		\draw [->]  (nonsing)  edge  (X);
		\draw [->] (rel) edge  (X);
		\draw [->] (nonsing) edge node[above]{$\text{ev}_{p_i}$} (rel);
		\end{tikzpicture}
	\end{center}
	By the previous theorem $\qm^{d}_{{\rm{relative}}, p_1,\cdots,p_m}$ is proper over 
	$\N^{0}_{k,n}$ and thus $\text{ev}_{p_i}$ is also proper. 
\end{proof}

\begin{Corollary} \label{wellpush}
There exist	pushforward maps ${\rm ev}_{p_i,*}: K_{\bT}( \qm^{d}_{{\rm{relative}}, p_1,\dots,p_m} ) \rightarrow K_{\bT}(\N_{k,n})$ for $i=1,\dots,m$. 	
\end{Corollary}

\subsection{Nonsingular quasimaps\label{nonssec}}
For distinct points $r_1,\dots,r_s,p_1,\dots, p_m\in \mathcal{D}$ we consider an open subset of $\qm^{d}_{ {{{\rm{relative}},  p_1,\dots,p_m}}}$:
$$
{\qm}^{d}_{{{\rm{nonsing}}, r_1,\cdots,r_s } \atop {{{\rm{relative}},  p_1,\dots,p_m}}}\subset \qm^{d}_{ {{{\rm{relative}},  p_1,\dots,p_m}}},
$$
consisting of quasimaps such that  $f(r_i)$ are stable for all $i=1,\dots, {\rd s}$\footnote{\rd Note that $r_i$ are separate from the relative points $p_i$ 
and thus we may identify $r_i$ with points $\pi^{-1}(r_i)\in \mathcal{C}$,
see discussion after Definition 1 above. }.  Restricting the obstruction theory of $\qm^{d}$ to this open subset we obtain the following corollary.

\begin{Corollary}
	The stack ${\qm}^{d}_{{{\rm{nonsing}}, r_1,\cdots,r_s } \atop {{{\rm{relative}},  p_1,\dots,p_m}}}$ is a  Deligne-Mumford stack of finite type, with  perfect obstruction theory.	
\end{Corollary}
This moduli space is equipped with well-defined evaluation maps
$$
{\rm{ev}}_{p_i}: {\qm}^{d}_{{{\rm{nonsing}}, r_1,\cdots,r_s } \atop {{{\rm{relative}},  p_1,\dots,p_m}}}\rightarrow \N_{k,n}, \ \ \ 
{\rm{ev}}_{r_j}: {\qm}^{d}_{{{\rm{nonsing}}, r_1,\cdots,r_s } \atop {{{\rm{relative}},  p_1,\dots,p_m}}}\rightarrow \N_{k,n}
\ \ \ 
$$
for $i=1,\dots,m, j=1,\dots,s$.
Note, that ${\qm}^{d}_{{{\rm{nonsing}}, r_1,\cdots,r_s } \atop {{{\rm{relative}},  p_1,\dots,p_m}}}$ is not necessarily proper over $\N^{0}_{k,n}$. Thus, in general, these evaluation maps  do not provide push-forward maps in K-theory.

In special cases, however, it is possible to define ${\rm{ev}}_{p_i,*}$ and ${\rm{ev}}_{r_i,*}$  for localized K-theory. Let us denote by $\C^{\times}_q$  a one-dimensional torus acting  on $\mathcal{D}=\mathbb{P}^1$ in the following way. 
If $[x:y]$ denotes homogeneous coordinates on $\mathcal{D}$ then the corresponding action of $\C^{\times}_q$ is defined by:
$$
[x:y] \rightarrow [x q :y] =[x :y q^{-1}].
$$ 
Let us note that the set  $\mathcal{D}^{\C^{\times}_q}$ consists of two isolated fixed points:
\be \label{points}
p_1=[1,0] =0 \in \mathbb{P}^{1},  \ \ p_2=[0,1]=\infty \in \mathbb{P}^{1}.
\ee
This means that $\C^{\times}_{q}$-fixed quasimap may only have special marked points on $\mathcal{D}$  (i.e. nonsingular or relative points)  at $p_1$ or $p_2$. The following result is proved in Section 7.2 of \cite{lectok}:
\begin{Theorem}
	For $p_1,p_2$ as in (\ref{points}) the evaluation maps 
	$$
\mathrm{ev}_{p_1}: (\qm^{d}_{\mathrm{nonsing}, p_1 } )^{\C^{\times}_q} \rightarrow \N_{k,n}, 
	$$
		$$
	\mathrm{ev}_{p_1}: (\qm^{d}_{{\mathrm{nonsing}, p_1} \atop {\mathrm{relative}, p_2} } )^{\C^{\times}_q} \rightarrow \N_{k,n}, 
	$$
	are proper. 
\end{Theorem}
We denote by $\bG=\bT\times \C^{\times}_{q}$ the combined torus acting on these moduli spaces.
\begin{Corollary} \label{nonscorr}
There exist pushforward maps 
$$
\mathrm{ev}_{p_1,*}: K_{\bG}(\qm^{d}_{\mathrm{nonsing}, p_1 } ) \rightarrow K_{\bG}(\N_{k,n})_{loc}
$$	
$$
\mathrm{ev}_{p_1,*} \otimes \mathrm{ev}_{p_2,*}: K_{\bG}(\qm^{d}_{{\mathrm{nonsing}, p_1} \atop {\mathrm{relative}, p_2} } ) \rightarrow K_{\bG}(\N_{k,n})^{\otimes 2}_{loc}
$$ 
\end{Corollary}


\subsection{Capped K-theoretic tensors} 
Let us consider the following map (cf. Corollary \ref{wellpush}):
$$
\textrm{ev}_{p_1 *} \otimes \dots\otimes \textrm{ev}_{p_m *}: K_{\bT}\Big( {{\qm}^d}_{\text{relative }\, p_1,\dots,p_m} \Big)
\rightarrow K_{\bT}(\N_{k,n} )^{\otimes m}.
$$
Let ${\vss}^d\in K_{\bT}\Big( {{\qm}^d}_{\text{relative }\, p_1,\dots,p_m} \Big)$ denote the virtual structure sheaf of the moduli space provided by the perfect obstruction theory of ${{\qm}^d}_{\text{relative }\, p_1,\dots,p_m}$. 

\begin{Definition}
The  power series with coefficients in $K$-theory of $\N_{k,n}$	
\be \label{cten}
\hat{V}_{p_1,\dots,p_m}(z):=\sum\limits_{d=0}^{\infty}\, z^d\,  \mathrm{ev}_{p_1 *} \otimes \dots\otimes \mathrm{ev}_{p_m *}\Big({\vss}^d\Big)\in  K_{\bT}(\N_{k,n} )^{\otimes m}[[z]].
\ee
are called \rm{capped $K$-theoretic tensors}. 
\end{Definition}

Let $p \in \mathcal{D}$ be a point disjoint from $p_1,\dots,p_m$ and let $\tau$ be a Schur functor 
(with coefficients in $K_{\bT}(pt)$ as in Section \ref{definit}). The value $f(p)$ for such point is not necessary stable in the sense of (\ref{stpoints}), therefore it does not provide an evaluation map to $\N_{k,n}$, but only to the quotient stack:
$$
\textrm{ev}_{p}: {{\qm}^d}_{\text{relative }\, p_1,\dots,p_m}\rightarrow [\mu^{-1}(0)/GL(k)].
$$  
Let $\tb_{stack}=(V \times T^*R)/GL(k)$ be the element in $K$-theory class of this stack associated to $k$-dimensional $GL(k)$-module $V$. We define 
$$
\tau(\left.\qmV\right|_{p}):=\textrm{ev}_{p}^{*}( \tau(\tb_{stack}) ). 
$$

\begin{Definition}
The power series 	
\be\label{phfd}
\hat{V}^{(\tau)}_{p_1,\dots,p_m}(z):= \sum\limits_{d=0}^{\infty}\, z^d\,  \mathrm{ev}_{*,p_1} \otimes \dots\otimes \mathrm{ev}_{*,p_m}\Big({\vss^d \otimes \tau(\left.\qmV\right|_{p})}\Big) 
\ee
is called \textit{capped K-theoretic tensor with descendent $\tau$ inserted at $p$} (the notation "at $p$", will often be omitted, since we almost always use the point $p=0$).   
\end{Definition}

\subsection{Degeneration and Gluing\label{glf}}


Let  $\mathcal{D}$ be the parametrized component of a quasimap and $p_1,\dots,p_m \in \mathcal{D}$ be relative points. As we discussed above, this setting produces a tensor of rank $m$  by the pushforwards (\ref{phfd}). In the basis of fixed points of $K_{\bT}(\N_{k,n})$ we can describe such tensor explicitly by its components $T_{i_1,...,i_m} (\mathcal{D})$.

Let us  assume that $\mathcal{D}$ degenerates to a union of two rational curves $\mathcal{D}=\mathcal{D}_1\cup_p \mathcal{D}_2$ with one common nodal point $p$.  The \textit{degeneration formula} in the theory of quasimaps states that the tensors for quasimaps with domain $\mathcal{D}$ and its degeneration $\mathcal{D}_1\cup_p \mathcal{D}_2$ are equal:
$$
T_{i_1,...,i_m} (\mathcal{D})=T_{i_1,...,i_m} (\mathcal{D}_1\cup_p \mathcal{D}_2).
$$ 
This identity can be represented diagrammatically as 
\vspace{-3mm}
\begin{center}
	\begin{align} 
	\begin{tikzpicture}
	\draw [ultra thick] (-3,0) -- (1.4,0);
	\node at (1.7,0) {$=$};
	\draw [ultra thick] (4,0) to [out=0,in=240] (4.7,0.4);
	\draw [ultra thick] (5,0) to [out=180,in=300] (4.3,0.4);
	\draw [ultra thick] (5,0) -- (7,0);
	\draw [ultra thick] (2,0) -- (4,0);
	\node at (-2,-0.3) {$p_1$};
	\node at (-0.9,-0.3) {$\dots$};
	\node at (0.2,-0.3) {$p_m$};
	\node at (2.4,-0.3) {$p_1$};
	\node at (3,-0.3) {$\dots$};
	\node at (3.6,-0.3) {$p_s$};
	\node at (5.2,-0.3) {$p_{s+1}$};
	\node at (6,-0.3) {$\dots$};
	\node at (6.6,-0.3) {$p_{m}$};
	\node at (4.55,0.45) {$p$};
	\end{tikzpicture}
	\end{align}
\end{center}

Next, the statement of the \textit{gluing formula} is that the  tensors (\ref{phfd}) given by quasimaps with parametrized component  $\mathcal{D}_1\cup_p \mathcal{D}_2$ can be factored into a product of tensors given by quasimaps with parametrized components $\mathcal{D}_1$ and $\mathcal{D}_2$.   More precisely, assume that after degeneration $p_1,.\cdots,p_s \in \mathcal{D}_1$  and the rest of the points lie within the second component, i.e. $p_{s+1},\dots, p_m \in \mathcal{D}_2$. Let us consider quasimaps with parametrized domain $\mathcal{D}_1$ and relative points $p_1,\dots,p_s,p$. This gives a tensor ${T}_{i_1,\dots,i_s,a}(\mathcal{D}_1)$  with $s+1$ indices, so that the index $a$ corresponds to the new point $p$. Similarly, for the second component
$\mathcal{D}_2$ we have a tensor 
${T}_{i_{s+1},\dots,i_{m},b}(\mathcal{D}_2)$ with $m-s+1$ indices, were $b$ is the index corresponding to the relative point $p$. The main statement of the gluing theorem is the existence of the \textit{gluing matrix} ${\bold{G}} \in  K^{\otimes 2}_{\bT}(\N_{k,n})[[z]]$ such that
\be \label{gluingth}
{T}_{i_1,\dots,i_m}(\mathcal{D}_1\cup_p \mathcal{D}_2)=\sum\limits_{a,b}\, {(\bold{G^{-1})}^{{a,b}}} {T}_{i_1,\dots,i_s,a}(\mathcal{D}_1)\, {T}_{i_{s+1},\dots,i_{m},b}(\mathcal{D}_2).
\ee
As a diagram this formula can be represented in the form:
\vspace{-3mm}
\begin{center}
	\be \label{glform}
	\begin{tikzpicture}
	\draw [ultra thick] (4,0) to [out=0,in=240] (4.7,0.4);
	\draw [ultra thick] (5,0) to [out=180,in=300] (4.3,0.4);
	\draw [ultra thick] (5,0) -- (7,0);
	\draw [ultra thick] (2,0) -- (4,0);
	\node at (2.4,-0.3) {$p_1$};
	\node at (3,-0.3) {$\dots$};
	\node at (3.6,-0.3) {$p_s$};
	\node at (5.2,-0.3) {$p_{s+1}$};
	\node at (6,-0.3) {$\dots$};
	\node at (6.6,-0.3) {$p_{m}$};
	\node at (4.55,0.45) {$p$};
	\node at (8.3,0) {$=$};
	\draw [ultra thick] (3+6.6,0) -- (4+7.6,0);
	\draw [ultra thick] (4+7.5,0.2) to [out=330,in=30] (4+7.5,-0.2);
	\node at (4+8.05,0) {${\bold{G}}^{-1}$};
	\draw [ultra thick] (4+8.5,0) -- (5+9.5,0);
	\draw [ultra thick] (4+3.1+5.5,0.2) to [out=210,in=150] (4+3.1+5.5,-0.2);
	\node at (10,-0.3) {$p_1$}; 
	\node at (10.6,-0.3) {$\dots$};
	\node at (11.2,-0.3) {$p_s$}; 
	\node at (13.1,-0.3) {$p_{s+1}$}; 
	\node at (13.8,-0.3) {$\dots$};
	\node at (14.4,-0.3) {$p_m$};
	\end{tikzpicture}
	\ee
\end{center}
The left part of this pictorial identity denotes the capped tensor  (\ref{phfd}) given by  quasimaps from with parametrized component given by $\mathcal{D}_1\cup_p \mathcal{D}_2$. The right side denotes  two capped tensors for parametrized components $\mathcal{D}_1$ and $\mathcal{D}_2$ 
contracted by the matrix ${\bold{G}}^{-1}$.  For the proof of the degeneration and gluing formulas we refer the reader to the Section 6.5 of \cite{lectok}.

\subsection{More about diagrammatic notation} 
Similar to the example with the degeneration and gluing formulas in the previous section, it is often  convenient to use diagrammatic abbreviation for tensors (\ref{phfd}) given by quasimaps  with different parametrized domains.  

We will use the following  {\it pictures} representing the parametrized domains:
\vspace{0.2in}

\hspace{45pt}\begin{tikzpicture}
\draw [ultra thick] (0,0) -- (2,0);
\end{tikzpicture}  \hspace{10pt}denotes parametrized domain  ${\mathbb{P}}^1$ of  quasimaps,

\vspace{0.2in}
\hspace{45pt}\begin{tikzpicture}
\draw [ultra thick] (0,0) -- (1,0);
\draw [ultra thick] (1,0) -- (2,0);
\draw [fill] (1,0) circle [radius=0.1];
\end{tikzpicture}\hspace{10pt} denotes a marked point on parametrized domain \\
\hspace*{45mm}(no conditions on quasimaps),

\vspace{0.2in}
\hspace{45pt}\begin{tikzpicture}
\draw [ultra thick] (0,0) -- (1,0);
\draw [ultra thick] (0,0) -- (2,0);
\draw [ultra thick] (0.9,0.2) to [out=330,in=30] (0.9,-0.2);
\end{tikzpicture}\hspace{10pt} denotes a relative point on the parametrized domain,

\vspace{0.2in}
\hspace{45pt}\begin{tikzpicture}
\draw [ultra thick] (0,0) -- (0.87,0);
\draw [ultra thick] (1.1,0) -- (2,0);
\draw [very thick] (1,0) circle [radius=0.13];
\end{tikzpicture}\hspace{10pt} denotes a nonsingular point on the parametrized domain.

\vspace{0.2in}
\hspace{45pt}\begin{tikzpicture}
\draw [ultra thick] (1.3,0) -- (2,0);
\draw [ultra thick] (3,0) -- (3.6,0);
\draw [ultra thick] (2,0) to [out=0,in=240] (2.7,0.4);
\draw [ultra thick] (3,0) to [out=180,in=300] (2.3,0.4);
\end{tikzpicture}\hspace{10pt} denotes a node of a nodal parametrized domain.
\vspace{0.2in}

\begin{Definition}
A diagram associated to a picture  is a tensor (\ref{cten}) where $\vss^d$ is a virtual structure sheaf on the moduli space of quasimaps of degree $d$ with parametrized domain $\mathcal{D}$ given by the corresponding picture.
\end{Definition}
For example, we have the following Theorem:

\begin{Theorem}[Section 7.1 of \cite{lectok}]  \label{gthm}
	The gluing matrix equals to the following diagram:
	\begin{center}
		\begin{tikzpicture}
		\node at (-0.5,0) {${\bold{G}}=$};
		\draw [ultra thick] (0,0) -- (2,0);
		\draw [ultra thick] (0.1,0.2) to [out=210,in=150] (0.1,-0.2);
		\draw [ultra thick] (1.9,0.2) to [out=330,in=30] (1.9,-0.2);
		\end{tikzpicture}
	\end{center}
	\vspace{-0.1in}
\end{Theorem}	
\noindent 
By the definition above, this diagram represents a tensor (\ref{cten}) given by quasimaps
with parametrized domain $\mathcal{D}\cong \mathbb{P}^{1}$ with  two relative points $p_1, p_2$, i.e., the tensor ${\bold{G}}$ is given by the following pushforward:  
$$
{\bold{G}}=\sum\limits_{d=0}^{\infty}\, z^d {\text{ev}}_{p_1,*} \otimes {\text{ev}}_{p_2,*} \left(\vss^d \right), 
$$
where $\vss^d$ is a structure sheaf on  the moduli space ${\qm}^d_{\text{relative} \hspace{2pt}p_1,p_2}$.

\subsection{Special tensors and quantum tautological bundles \label{locsec}} 
For this subsection $\mathcal{D}=\mathbb{P}^1$ and $\bG$-fixed points $p_1$
and $p_2$ as in (\ref{points}).
The following definition is due Okounkov \cite{lectok} Section 7.2.
\begin{Definition}
The vector (tensor of rank $1$), defined by  localized $\bG$-equivariant pushforward
\be \label{capver}
V^{(\tau)}(z)=\sum\limits_{d=0}^{\infty} z^d {\rm{ev}}_{p_2, *}\Big(\vss^d \otimes \tau (\left.\qmV\right|_{p_1}) \Big) \in  K_{\bG}(\N_{k,n})_{loc}[[z]],
\ee	
where $\vss^{d}$ is the virtual structure sheaf of $\qm^{d}_{\mathrm{nonsing} \, p_2}$ and $\mathrm{ev}_{p_2}$ is the corresponding evaluation map is called {\rm{bare vertex}} with descendant $\tau$. 
\end{Definition}

Note that the bare vertex corresponds to the diagram:	
\begin{center}
	\begin{tikzpicture}
	\node at (-0.9,0) {${V^{(\tau)}}(z)=$};
	\draw [ultra thick] (0.2,0) -- (2,0);
	\draw [fill] (0.2,0) circle [radius=0.1];
	\draw [very thick] (2.17,0) circle [radius=0.13];
	\node at (0.3,-0.25) {$\tau$};
	\end{tikzpicture}
\end{center}
As we discuss in Section \ref{nonssec},  the evaluation map in (\ref{capver}) is not proper
and the result is defined only via $\bG$-localization to fixed points
by Corollary \ref{nonscorr}.  In practice this means that the coefficients of the power series (\ref{capver}) are allowed to have denominators of the form $1/(1-q)$, which are divergent in the ``delocalization" limit $q\to 1$. 
In Section~\ref{comver} we compute the bare vertex with arbitrary descendent explicitly. We also show that  the power series (\ref{capver}) are expressed via some standard $q$-hypergeometric function with the prescribed divergent behavior in the limit $q\to 1$.

The following definition is due Okounkov \cite{lectok} Section 7.4.
\begin{Definition}
	The vector defined by $\bG$-equivariant pushforward 
	\be \label{capd}
	\hat{V}^{(\tau)}(z)=\sum\limits_{d=0}^{\infty} z^d {\rm{ev}}_{p_2, *}\Big(\vss^d \otimes \tau (\left.\qmV\right|_{p_1}) \Big) \in K_{\bG}(\N_{k,n})
	\ee
where $\vss^{d}$ is the virtual structure sheaf of $\qm^{d}_{\mathrm{relative} \, p_2}$ and $\mathrm{ev}_{p_2}$ is the corresponding evaluation map is called  {\rm{capped vertex}} with descendant $\tau$.			
\end{Definition}
As a diagram, the capped vertex has the following representation:
$$
\begin{tikzpicture}
\node at (-0.9,0) {${{\hat{V}}^{(\tau)}}(z)=$};
\draw [ultra thick] (0.2,0) -- (2,0);
\draw [fill] (0.2,0) circle [radius=0.1];
\draw [ultra thick] (1.9,0.2) to [out=330,in=30] (1.9,-0.2);
\node at (0.3,-0.25) {$\tau$};
\end{tikzpicture}
$$
In contrast with the bare vertex, this tensor is defined via a proper pushforwad. 
Hence, the capped vertex is an element of the integral (non-localized) equivariant K-theory $\hat{V}^{(\tau)}(z)\in K_{\bG}(\N_{k,n})[[z]]$. In particular, it implies that the coefficients of the power series (\ref{capd}) are the  Laurent polynomials in equivariant parameter $q$. The $q\to 1$ limits of the capped vertex is therefore well-defined.

\begin{Definition} \label{qtbdef}
We call the vector 
$
\hat{\tau}(z)=\lim\limits_{q\to 1} \, \hat{V}^{(\tau)}(z)
$
a {\rm quantum tautological bundle} associated with $\tau$. 
\end{Definition}
Note, that the quantum tautological bundle $\hat{\tau}(z)$
can be defined by same formula (\ref{capd}) with $\bT$-equivariant (not $\bG$-equivariant) pushforward ${\rm{ev}}_{p_2, *}$ so that $\hat{\tau}(z)$  is an element  of $K_{\bT}(\N_{k,n})[[z]]$. The last tensor we need is defined by Okounkov \cite{lectok} Section 8.1.6:
\begin{Definition} \label{maindef1}
The rank two tensor defined by the following $\bG$-equivariant pushforward 
\be \label{capping}
\Psi(z)=\sum\limits_{d=0}^{\infty}\, z^{d} \mathrm{ev}_{p_1,*}\otimes \mathrm{ev}_{p_2,*} \Big(  \vss^d \Big) \in K^{\otimes 2}_{\bG}(\N_{k,n})_{loc}[[z]].
\ee
where $\vss^d$ denoted the virtual structure sheaf of
$\qm^{d}_{{{\mathrm{relative}\, p_1}} \atop {{\mathrm{nonsing}\, p_2}}}$ and $\mathrm{ev}_{p_i,*}$ are the corresponding push-forward maps is called {\rm capping operator}.
\end{Definition}
In the diagrammatic notations we have:
$$
	\begin{tikzpicture}
	\node at (-0.9,0) {$\Psi(z)=$};
	\draw [ultra thick] (0.2,0) -- (2,0);
	\draw [very thick] (0.11,0) circle [radius=0.13];;
	\draw [ultra thick] (1.9,0.2) to [out=330,in=30] (1.9,-0.2);
	\end{tikzpicture}
$$	
Similarly to the bare descendent vertex this is an element of the localized K-theory
due to the nonsingular condition at $p_2$. In particular, the capping operator is  divergent in the limit $q\to 1$. 

Using the standard pairing on the equivariant K-theory:
\be
\label{bf}
{(\mathcal{F},\mathcal{G})}=\chi(\mathcal{F}\otimes\mathcal{G}),
\ee
one can think about this tensor as operator acting from the first copy of  $K_{\bG}(\N_{k,n})_{loc}$ to the second. The name ``capping operator'' is justified by the following theorem.

\begin{Theorem}[\cite{lectok}, Section 7.4]
	The capping operator maps the bare vertex with descendants to the capped vertex with descendants:
	\vspace{0.2in}
	\begin{center}
		\begin{tikzpicture}
		\draw [ultra thick] (0,0) -- (2,0);
		\draw [fill] (2,0) circle [radius=0.1];
		\draw [ultra thick] (0.1,0.2) to [out=210,in=150] (0.1,-0.2);
		\node at (2.3,0) {$\tau$};
		\node at (2.9,0) {$=$};
		\draw [ultra thick] (3.4,0) -- (4.4,0);
		\draw [ultra thick] (0.1+3.4,0.2) to [out=210,in=150] (0.1+3.4,-0.2);
		\draw [very thick] (4.53,0) circle [radius=0.13];
		\draw [very thick] (5,0) circle [radius=0.13];
		\draw [ultra thick] (5.13,0) -- (6.13,0);
		\draw [fill] (6.13,0) circle [radius=0.1];
		\node at (6.43,0) {$\tau$};
		\end{tikzpicture}
	\end{center}
	or, explicitly, the pushforwards   (\ref{capver}),(\ref{capd}), (\ref{capping}) are related as follows:
\be \label{capth}
	{{\hat{V}}^{(\tau)}}(z)=\Psi(z){{{V}}^{(\tau)}}(z). 
\ee
\end{Theorem}
\begin{proof}
The formula (\ref{capth}) can be obtained from computing the capped vertex (\ref{capd}) by $\C^{\times}_{q}$-localization. First, we analyze the fixed stack $(\qm_{{\mathrm{relative}},p_2})^{\C^{\times}_{q}}$.
Let us consider the quasimap  ${\bf{m}} \in (\qm_{{\mathrm{relative}},p_2})^{\C^{\times}_{q}}$.  The restriction of the quasimap data to the parametrized component $\left.{\bf{m}}\right|_{\mathcal{C}_0}$ is an element of $\qm_{{\mathrm{nonsing}},p_2}^{\C^{\times}_{q}}$ because it is both non-singular at $p_2$ and $\C^{\times}_{q}$-fixed.

We denote by $\qm^{\sim}_{p_2',p_2}$ the stack of qusimaps with ``parametrized domain'' given by a single points $p_2$ with one relative point $p_2$ (same point). 
In other words, this is the stack of quasimaps from Definition \ref{qmdef} with $\mathcal{C}=\mathcal{D}=p_2$.

The restriction of the quasimap data  $\left.{\bf{m}}\right|_{\tilde{\mathcal{C}}}$
(here ${\tilde{\mathcal{C}}}$ is the closure of $\mathcal{C}\setminus \mathcal{C}_0$) is a quasimap nonsingular at $p_2, p_2'\in {\tilde{\mathcal{C}}}$. 
In addition $\pi({\tilde{\mathcal{C}}})=p_2$ and thus $\left.{\bf{m}}\right|_{\tilde{\mathcal{C}}} \in \qm^{\sim}_{p_2',p_2}$. 

In reverse, two quasimaps from $\qm_{{\mathrm{nonsing}},p_2}^{\C^{\times}_{q}}$ and $\qm^{\sim}_{p_2',p_2}$, whose values at $p_2$ agree, define a quasimap from  $(\qm_{{\mathrm{relative}},p_2})^{\C^{\times}_{q}}$. We conclude:
\be \label{qfixed}
(\qm_{{\mathrm{relative}},p_2})^{\C^{\times}_{q}}=\qm^{\sim}_{p_2',p_2}\times_{\N_{k,n}} (\qm_{{\mathrm{nonsing}},p_2})^{\C^{\times}_{q}},
\ee 
where $\times_{\N_{k,n}}$ denotes a fiber product
with respect to the evaluation maps $\textrm{ev}_{p_2}$.

 For quasimaps in $\qm^{\sim}_{p_2,p_2'}$ the values $f(p_2)$ and $f(p_2')$ are stable and thus we have two well-defined evaluation maps:  
$\mathrm{ev}_{p_2}(f)=f(p_2)$ and $\mathrm{ev}_{p_2'}(f)=f(p_2')$ (note a small difference between this and Definition \ref{evmapsdef}).

Let $N_{p_2}$ denote the virtual normal bundle to (\ref{qfixed}). This normal bundle has the form $N_{p_2}=\psi_{p_2}\otimes \C_q$ where $\psi_{p_2}$ is a line bundle on $\qm^{\sim}_{p_2',p_2}$ and $\C_q$ is a trivial $\C^{\times}_q$-equivariant line bundle over the fixed stack corresponding to the character  $q$  see Section 7.1 of \cite{lectok}. The localization formula then gives:
\begin{small} 
\begin{align}\label{qloc}
\begin{array}{l}
\hat{V}^{(\tau)}(z)=\\
\Big(\sum_{d} z^d  \mathrm{ev}_{p_2',*} \otimes \mathrm{ev}_{p_2,*} \Big(\qm^{\sim}_{p_2,p_2'},\dfrac{\vss}{1-N^{\vee}_{p_2}}\Big)\Big) \times 
 \Big(\sum_{d} z^d \mathrm{ev}_{p_2,*} \Big(\qm^{\C^{\times}_{q}}_{{\mathrm{nonsing}},p_2},\vss \otimes \tau (\left.\qmV\right|_{p_1}) \Big)\Big).
 \end{array}
\end{align}
\end{small}
Here the first and the second factor represent rank two and rank one $K$-theory-valued tensors and $\times$ denotes the contraction of these tensors over the indices corresponding to $\mathrm{ev}_{p_2,*}$. The second factor is exactly the bare vertex with descendent $\tau$:
$$
{V}^{(\tau)}(z)=\sum_{d} z^d \mathrm{ev}_{p_2,*} \Big(\qm^{\C^{\times}_{q}}_{{\mathrm{nonsing}},p_2},\vss \otimes \tau (\left.\qmV\right|_{p_1}) \Big).
$$ 
Similarly we analyze the capping operator (\ref{capping}) and the corresponding moduli space $\qm^{d}_{{{\textrm{relative}\, p_2}} \atop {{\textrm{nonsing}\, p_1}}}$. For a $\C^{\times}_{q}$-fixed quasimap  ${\bf{m}}\in(\qm^{d}_{{{\textrm{relative}\, p_2}} \atop {{\textrm{nonsing}\, p_1}}})^{\C^{\times}_2}$ the set of singularities 
$B\cap {\mathcal{C}_0}$ must be a subset of the fixed set $\{p_1,p_2\}$. However, by Definition \ref{qmdef} such a quasimap is also non-singular at $\{p_1,p_2\}$. Thus, being restricted to ${\mathcal{C}}_0$, section $f$ is a section with no poles and zeros, and therefore is a constant. We conclude that the restriction of the quasimap data $\left.{\bf{m}}\right|_{{\mathcal{C}}_0}$ is trivial 
and the quasimap $\bf{m}$ is completely determined by its restriction $\left.{\bf{m}}\right|_{\tilde{\mathcal{C}}}$.   
Therefore,
$$
(\qm^{d}_{{{\textrm{relative}\, p_2}} \atop {{\textrm{nonsing}\, p_1}}})^{\C^{\times}_q}=\qm^{\sim}_{p_2',p_2}.
$$
The $\C^{\times}_q$-localization for the capping operator gives:
$$
\Psi(z)=\sum_{d} z^d  \mathrm{ev}_{p_2',*} \otimes \mathrm{ev}_{p_2,*} \Big(\qm^{\sim}_{p_2,p_2'},\dfrac{\vss}{1-N^{\vee}_{p_2}}\Big),
$$
which is the first factor in (\ref{qloc}). 
\end{proof}
Let us note that (\ref{capth}) is quite nontrivial: both the capping operator and the bare vertex on the right side of this equality are elements of the localized K-theory, while the capped vertex on the left is a class in the integral K-theory. This means that the divergent terms of the capping operator and the bare vertex in the limit $q\to 1$ cancel each other, so that the limit of the product is well-defined. This observation plays an important role in the computations of Section 4.

In Section \ref{comver} we use localization to explicitly
compute the bare vertex. The capping operator,
and thus (by the previous theorem) the capped vertex, can be computed as the fundamental solution of certain $q$-difference equations which we discuss below.

\subsection{Quantum difference equation}

Let us consider quasimaps with parametrized domain $\mathcal{D}=\mathbb{P}^1$ and let $p_1, p_2\in \mathcal{D}$ be the fixed points of $\mathbb{C}^{\times}_q$-action as in (\ref{points}).

\begin{Theorem}[Theorem 8.1.16 of \cite{lectok}]\label{Mop}
	
	The capping operator $\Psi(z)$ is the fundamental solution matrix of the quantum difference equation:
	\be
	\label{difference}
	\Psi(qz)=\M(z) \Psi(z) \O(1)^{-1},
	\ee
	where $\O(1)$ is the operator of classical multiplication by the corresponding line bundle on $\N_{k,n}$ and $\M(z)\in K_{\bG}(\N_{k,n})^{\otimes 2}[[z]]$ is defined by
	\begin{align}
	\label{qdeop}
	\M(z)=\sum\limits_{d=0}^{\infty} z^d {\rm{ev}}_{p_1 , \ast}\otimes {\rm{ev}}_{p_2, \ast} \left(\qm^{d}_{{\rm{relative}}\, p_1,p_2},\vss\otimes \det H^{\bullet}\left(\qmV\otimes \pi^{\ast}(\mathcal{O}_{p_1})\right)\right)\bold{G}^{-1},
	\end{align}
	where $\pi:\mathcal{C}\to {\mathcal{D}}$ is a part of the quasimap data, $\mathcal{O}_{p_1}$ is a class of
	point $p_1\in {\mathcal{D}}$ and ${\rm{ev}}_{p_1 , \ast}, {\rm{ev}}_{p_2, \ast}$ are $\bG$-equivariant push-forward maps. 
\end{Theorem}

\begin{proof}
Consider the $q$-equivariant coherent sheaf $\mathscr{F}$ over $\mathcal{D}$. We have 
\begin{eqnarray}
\det H^{\bullet}
(\mathscr{F}\otimes (\mathcal{O}_{p_1}-\mathcal{O}_{p_2}))=q^{{\rm deg}(\mathscr{F})},
\end{eqnarray}
To show that it is enough to check it for $\mathcal{O}$ and $\mathcal{O}_{p_i}$.  Since ${\rm deg} (\pi_*\mathscr{V})={\rm deg} (\mathscr{V})$, we have:
\begin{eqnarray}
\det H^{\bullet}
(\mathscr{V}\otimes \pi^*(\mathcal{O}_{p_1}-\mathcal{O}_{p_2}))=q^{{\rm deg}(\mathscr{V})}.
\end{eqnarray}
By Definition \ref{defdegree} we have $d=\deg \mathscr{V}$ and thus we can write: 
\begin{eqnarray}\label{qshift}
(z q)^d= z^d \det H^{\bullet}
(\mathscr{V}\otimes \pi^*(\mathcal{O}_{p_1}-\mathcal{O}_{p_2})).
\end{eqnarray}
Now let us have a look at the capping operator:
\be 
\Psi(z)=\sum\limits_{d=0}^{\infty}\, z^{d} \textrm{ev}_{p_1*} \otimes \textrm{ev}_{p_2*} \Big( \qm^{d}_{{{\textrm{relative}\, p_1}} \atop {{\textrm{nonsing}\, p_2}}}, \vss \Big) .
\ee
By (\ref{qshift}) we have:
$$
\Psi(z q)=\sum\limits_{d=0}^{\infty}\, z^{d} \textrm{ev}_{p_1*} \otimes \textrm{ev}_{p_2*}\Big( \qm^{d}_{{{\textrm{relative}\, p_1}} \atop {{\textrm{nonsing}\, p_2}}}, \vss \otimes \det H^{\bullet}
(\mathscr{V}\otimes \pi^*(\mathcal{O}_{p_1}-\mathcal{O}_{p_2})) \Big). 
$$
As in Section \ref{glf} we consider a degeneration of the parametrized component into a union of two rational curves ${\mathcal{D}}={\mathcal{D}}_1\cup_p {\mathcal{D}}_2$ with one node $p$. Then, gluing formula  (\ref{gluingth}) gives
$$
\begin{array}{l}
\Psi(z q)=\Big(\sum\limits_{d=0}^{\infty}\, z^{d} \textrm{ev}_{p_1} \otimes \textrm{ev}_{p} \Big( \qm^{d}_{{{\textrm{relative}\, p_1}} \atop {{\textrm{relative}\, p}}}, \vss \otimes \det H^{\bullet}
(\mathscr{V}\otimes \pi^*(\mathcal{O}_{p_1})) \Big)\Big)  \times \bold{G}^{-1}\\
\times \Big(\sum\limits_{d=0}^{\infty}\, z^{d} \textrm{ev}_{p} \otimes \textrm{ev}_{p_2} \Big( \qm^{d}_{{{\textrm{relative}\, p}} \atop {{\textrm{nonsing}\, p_2}}}, \vss \otimes \det H^{\bullet}
(\mathscr{V}\otimes \pi^*(-\mathcal{O}_{p_2})) \Big)\Big).
\end{array}
$$
The first two factors on the right side give $\M(z)$ by definition (\ref{qdeop}). For the last factor we note that $\pi$ is locally an isomorphism at the point $p_2$
because $p_2$ is not a relative point. Thus 
$H^{\bullet}(\mathscr{V}\otimes \pi^*(\mathcal{O}_{p_2}))=\mathscr{V}_{p_2}=\tb_{f(p_2)}$
where $\tb$ is the tautological bundle over $\N_{k,n}$. The last equality holds by the definition of the evaluation map.
We conclude that $H^{\bullet}(\mathscr{V}\otimes \pi^*(\mathcal{O}_{p_2}))=\textrm{ev}^{*}_{p_2}(\tb)$ and 
$\det H^{\bullet}
(\mathscr{V}\otimes \pi^*(\mathcal{O}_{p_2}))=\textrm{ev}^{*}_{p_2}(\det \tb)$.

In the K-theory of $\N_{k,n}$ we have 
$\det\tb ={\mathcal{O}}(1)$ and thus 
$$
\sum\limits_{d=0}^{\infty}\, z^{d} \textrm{ev}_{p} \otimes \textrm{ev}_{p_2} \Big( \qm^{d}_{{{\textrm{relative}\, p}} \atop {{\textrm{nonsing}\, p_2}}}, \vss \otimes \det H^{\bullet}
(\mathscr{V}\otimes \pi^*(-\mathcal{O}_{p_2})) \Big)=\Psi(z) {\mathcal{O}}(1)^{-1}. 
$$
Finally, $\Psi(0)=\mathcal{O}_{\N_{k,m}}$ which is a boundary condition for the fundamental solution.
\end{proof}
By definition, (\ref{qdeop}) is a power series $\M(z)=\sum_{d=0}^{\infty} z^d \M_{d}$. It  was computed explicitly for Nakajima varieties in \cite{os}, in particular the following Theorem holds:
\begin{Theorem}[\cite{os}]
For the given $k, n$ the power series $\M(z)=\sum_{d=0}^{\infty} z^d \M_{d}$ is a Taylor expansion of a rational function, i.e., $\M(z) \in K_{\bG}(\N_{k,n})(z)$.	
\end{Theorem}

Once  $\M(z)$ is known explicitly, the corresponding $q$-difference equation (\ref{difference}) turns into a system of linear equations for unknown coefficients $\Psi_i$:
$$
\Psi(z)=\mathcal{O}_{\N_{k,m}}+\Psi_1 z+\Psi_2 z^2+\dots. 
$$ 
This provides the most efficient tool for computing the capping operators $\Psi(z)$.
For the case $\N_{k,n}$, which we investigate in this paper, the  operator $\M(z)$ 
is described explicitly in the Section 7.3.6 of \cite{os}.

\section{Quantum K-theory ring of $\N_{k,n}$}

\subsection{Multiplication in quantum K-theory \label{muldef}}

Recall that we denote by $\bold{G}$ the gluing matrix  given by Theorem \ref{gthm}.  Using the standard scalar pairing in
$K_{\bT}(\N_{k,n})$, namely (\ref{bf}),  
 we can think about this tensor 
as a linear operator, i.e.,  $\bold{G} \in End(K_{\bT}(\N_{k,n}))[[z]]$.

Let ${\qm}^{d}_{ \,p_1,p_2,p_3}$ be the moduli space of quasimaps from ${\mathbb{P}}^1$ with 3 relative points 
and let $\vss^d$ be the virtual structure sheaf on this moduli space.  Given a class $\mathcal{F}\in K_{\bT}(\N_{k,n})$, one can construct the following tensor:
\be
\label{qprod}
\mathcal{F}\circledast:= 
\left(\sum\limits_{d=0}^{\infty}\,z^d {\text{ev}}_{p_1\ast}\times {\text{ev}}_{p_3\ast} \left(\text{ev}^{\ast}_{p_2}({\bold{G}}^{-1}\mathcal{F})\otimes \vss^d \right)\right) {\bold{G}}^{-1}.
\ee
By definition, $\mathcal{F}\circledast$ is a rank two tensor, which, thanks to the scalar product in K-theory, can be identified with the linear operator:
$$
\mathcal{F}\circledast\in End({K_{\bT}(\N_{k,n})})[[z]].
$$
\begin{Definition}
We call $\mathcal{F}\circledast$ the operator of quantum multiplication by a class $\mathcal{F}$. 	
\end{Definition}
As a diagram, this operator can be represented  in the following form:
\vspace{-0.1in}
\begin{center}
	\be \label{qmuldef}
	\begin{tikzpicture}
	\node at (-1,-0) {$\mathcal{F} \circledast=$}; 
	\draw [ultra thick] (0,0) -- (2,0);
	\draw [ultra thick] (0.1,0.2) to [out=210,in=150] (0.1,-0.2);
	\draw [ultra thick] (1.9,0.2) to [out=330,in=30] (1.9,-0.2);
	\draw [ultra thick] (0.9,0.2) to [out=330,in=30] (0.9,-0.2);
	\node at (1,-0.4) {${\bold{G}}^{-1}\mathcal{F}$};
	\node at (2.5,-0) {${\bold{G}}^{-1}$};
	\end{tikzpicture}
	\ee
\end{center}
\vspace{0.1in}
Note, that the moduli space of degree zero quasimaps is isomorphic to $\N_{k,n}$, since a degree zero quasimap maps the entire curve to a single point in $\N_{k,n}$. This implies that the zeroth coefficient
of the power series (\ref{qprod}) has the form: 
\be \label{relcla}
\mathcal{F} \circledast  =\mathcal{F} \otimes +\dots ,
\ee
where $\mathcal{F} \otimes$ is the operator of tensor multiplication by $\mathcal{F}$ in the equivariant $K$-theory, i.e. for $\mathcal{G} \in {K_{\bT}(\N_{k,n})}$
$$
\mathcal{F} \otimes: \mathcal{G} \mapsto \mathcal{F} \otimes \mathcal{G} 
$$
and dots in (\ref{relcla}) stand for the terms vanishing in $z\to 0$ limit (called {\it classical limit}).
\begin{Definition} \label{maindef2}
	We call  $QK_{\bT}(\N_{k,n}):=K_{\bT}(\N_{k,n})[[z]]$ endowed with the multiplication~(\ref{qprod}) the   quantum equivariant $K$-theory ring of $\N_{k,n}$.
\end{Definition}
Let us note note these definitions of quantum product and quantum $K$-theory ring are new and introduced in this paper for the first time. These definitions apply directly to varieties for which quasimaps are defined, for example, Nakajima varieties.

\noindent
We are now ready to prove the following result.
\begin{Theorem}
	The quantum $K$-theory ring $QK_{\bT}(\N_{k,n})$ is a commutative, associative unital algebra.
\end{Theorem}
\begin{proof}
	Every statement of this theorem, except existence of the unity, follows directly from our definitions and the gluing formula.  Commutativity of this algebra follows from the construction of multiplication operator, by switching points $p_2$ and $p_3$, since our definition (\ref{qprod}) is symmetric with respect to $p_1\leftrightarrow
	p_3$. 
	
	Associativity of this ring follows from the fact that operators of quantum multiplication by two different sheafs $\mathcal{F}$ and $\mathcal{G}$ commute. The picture proof of this is as follows:
	\vspace{0.1in}
	\begin{center}
		\begin{tikzpicture}
		\draw [ultra thick] (0,0) -- (2,0);
		\draw [ultra thick] (0.1,0.2) to [out=210,in=150] (0.1,-0.2);
		\draw [ultra thick] (1.9,0.2) to [out=330,in=30] (1.9,-0.2);
		\draw [ultra thick] (0.9,0.2) to [out=330,in=30] (0.9,-0.2);
		\node at (1,-0.4) {${\bold{G}}^{-1}\mathcal{F}$};
		\node at (2.45,-0) {${\bold{G}}^{-1}$};
		\node at (3.1,0) {$\times$};
		\draw [ultra thick] (3.5,0) -- (5.5,0);
		\draw [ultra thick] (3.6,0.2) to [out=210,in=150] (3.6,-0.2);
		\draw [ultra thick] (5.4,0.2) to [out=330,in=30] (5.4,-0.2);
		\draw [ultra thick] (4.4,0.2) to [out=330,in=30] (4.4,-0.2);
		\node at (4.5,-0.4) {${\bold{G}}^{-1}\mathcal{G}$};
		\node at (5.95,-0) {$\bold{G}^{-1}$};
		\node at (6.65,-0) {$=$};
		\node at (6.65,0.3) {$\bold{1}$};
		\draw [ultra thick] (6.2,-1) -- (8.2,-1);
		\draw [ultra thick] (6.3,-0.8) to [out=210,in=150] (6.3,-1.2);
		\draw [ultra thick] (8.2,-1) to [out=0,in=240] (8.9,-0.6);
		\draw [ultra thick] (9.2,-1) to [out=180,in=300] (8.5,-0.6);
		\draw [ultra thick] (9.2,-1) -- (11.2,-1);
		\draw [ultra thick] (7.4,-0.8) to [out=330,in=30] (7.4,-1.2);
		\draw [ultra thick] (9.9,-0.8) to [out=330,in=30] (9.9,-1.2);
		\draw [ultra thick] (11.1,-0.8) to [out=330,in=30] (11.1,-1.2);
		\node at (7.5,-1.4) {${\bold{G}}^{-1}\mathcal{F}$};
		\node at (10,-1.4) {${\bold{G}}^{-1}\mathcal{G}$};
		\node at (11.65,-1) {$\bold{G}^{-1}$};
		\node at (12.35,-1) {$=$};
		\node at (12.35,-0.7) {$\bold{2}$};
		\draw [ultra thick] (0+1.8,-2) -- (3+1.8+0.3,-2);
		\draw [ultra thick] (0.1+1.8,-1.8) to [out=210,in=150] (0.1+1.8,-2.2);
		\draw [ultra thick] (1.9+1.8+0.3,-1.8) to [out=330,in=30] (1.9+1.8+0.3,-2.2);
		\draw [ultra thick] (0.9+1.8,-1.8) to [out=330,in=30] (0.9+1.8,-2.2);
		\draw [ultra thick] (2.9+1.8+0.3,-1.8) to [out=330,in=30] (2.9+1.8+0.3,-2.2);
		\node at (3.45+1.8+0.3,-2) {$\bold{G}^{-1}$};
		\node at (1+1.8,-2.4) {${\bold{G}}^{-1}\mathcal{F}$};
		\node at (2+1.8+0.3,-2.4) {${\bold{G}}^{-1}\mathcal{G}$};
		\node at (4.15+1.8+0.3,-2) {$=$};
		\node at (4.15+1.8+0.3,-2+0.3) {$\bold{3}$};
		\draw [ultra thick] (4.6+1.8+0.3,-2) -- (7.6+1.8+0.6,-2);
		\draw [ultra thick] (4.7+1.8+0.3,-1.8) to [out=210,in=150] (4.7+1.8+0.3,-2.2);
		\draw [ultra thick] (5.5+1.8+0.6,-1.8) to [out=330,in=30] (5.5+1.8+0.6,-2.2);
		\draw [ultra thick] (6.5+1.8+0.6,-1.8) to [out=330,in=30] (6.5+1.8+0.6,-2.2);
		\draw [ultra thick] (7.5+1.8+0.6,-1.8) to [out=330,in=30] (7.5+1.8+0.6,-2.2);
		\node at (8.05+1.8+0.6,-2) {$\bold{G}^{-1}$};
		\node at (5.6+1.8+0.3,-2.4) {${\bold{G}}^{-1}\mathcal{G}$};
		\node at (6.6+1.8+0.6,-2.4) {${\bold{G}}^{-1}\mathcal{F}$};
		\node at (8.75+1.8+0.6,-2) {$=$};
		\node at (8.75+1.8+0.6,-2+0.3) {$\bold{4}$};
		\draw [ultra thick] (6.2-6.2,-1-2.2) -- (8.2-6.2,-1-2.2);
		\draw [ultra thick] (6.3-6.2,-0.8-2.2) to [out=210,in=150] (6.3-6.2,-1.2-2.2);
		\draw [ultra thick] (8.2-6.2,-1-2.2) to [out=0,in=240] (8.9-6.2,-0.6-2.2);
		\draw [ultra thick] (9.2-6.2,-1-2.2) to [out=180,in=300] (8.5-6.2,-0.6-2.2);
		\draw [ultra thick] (9.2-6.2,-1-2.2) -- (11.2-6.2,-1-2.2);
		\draw [ultra thick] (7.4-6.2,-0.8-2.2) to [out=330,in=30] (7.4-6.2,-1.2-2.2);
		\draw [ultra thick] (9.9-6.2,-0.8-2.2) to [out=330,in=30] (9.9-6.2,-1.2-2.2);
		\draw [ultra thick] (11.1-6.2,-0.8-2.2) to [out=330,in=30] (11.1-6.2,-1.2-2.2);
		\node at (7.5-6.2,-1.4-2.2) {${\bold{G}}^{-1}\mathcal{G}$};
		\node at (10-6.2,-1.4-2.2) {${\bold{G}}^{-1}\mathcal{F}$};
		\node at (11.65-6.2,-1-2.2) {$\bold{G}^{-1}$};
		\node at (12.35-6.2,-1-2.2) {$=$};
		\node at (12.35-6.2,-1-2.2+0.3) {$\bold{5}$};
		\draw [ultra thick] (0+5.7,0-4.2) -- (2+5.7,0-4.2);
		\draw [ultra thick] (0.1+5.7,0.2-4.2) to [out=210,in=150] (0.1+5.7,-0.2-4.2);
		\draw [ultra thick] (1.9+5.7,0.2-4.2) to [out=330,in=30] (1.9+5.7,-0.2-4.2);
		\draw [ultra thick] (0.9+5.7,0.2-4.2) to [out=330,in=30] (0.9+5.7,-0.2-4.2);
		\node at (1+5.7,-0.4-4.2) {${\bold{G}}^{-1}\mathcal{G}$};
		\node at (2.45+5.7,-0-4.2) {${\bold{G}}^{-1}$};
		\node at (3.1+5.7,0-4.2) {$\times$};
		\draw [ultra thick] (3.5+5.7,0-4.2) -- (5.5+5.7,0-4.2);
		\draw [ultra thick] (3.6+5.7,0.2-4.2) to [out=210,in=150] (3.6+5.7,-0.2-4.2);
		\draw [ultra thick] (5.4+5.7,0.2-4.2) to [out=330,in=30] (5.4+5.7,-0.2-4.2);
		\draw [ultra thick] (4.4+5.7,0.2-4.2) to [out=330,in=30] (4.4+5.7,-0.2-4.2);
		\node at (4.5+5.7,-0.4-4.2) {${\bold{G}}^{-1}\mathcal{F}$};
		\node at (5.95+5.7,-0-4.2) {$\bold{G}^{-1}$};
		\end{tikzpicture}
	\end{center}
	
	\vspace{0.1in}
	\noindent Here the equalities $\bold{1},\bold{2},\bold{4}, \bold{5}$ are the examples of degeneration and gluing formulas which we discussed in  Section \ref{glf}. The equality $\bold{3}$ is a deformation of the parametrized component. The existence and
	properties of multiplicative identity element in $QK_{\bT}(\N_{k,n})$ is discussed in the next subsection.
\end{proof}

\subsection{Multiplicative identity element of $QK_{\bT}(\N_{k,n})$ }
Until now, most of the definitions and statements about the quantum equivariant $K$-theory were analogous to ones in  quantum cohomology. However, there is one major difference related to the structure of the multiplicative identity element in the ring $QK_{\bT}(\N_{k,n})$.
In quantum cohomology, the element representing the multiplicative identity with respect to the quantum product coincides with the multiplicative identity of the classical theory, i.e. it is given by the fundamental class. In the quantum $K$-theory it is not true anymore: the multiplicative identity in the quantum $K$-theory ring \textit{does not always coincide} with
the structure sheaf $\O_{\N_{k,n}}$.\footnote{More precisely, the identity element of the quantum K-theory ring coincides with its classical version $\O_{\N_{k,n}}$ in the case $n\geq 2k$. We, however, will not discuss this property of the quantum K-theory ring in this paper.}

Let $\hat{\bold{1}}(z)\in QK_{\bT}(\N_{k,n})$  be the quantum tautological bundle from Definition \ref{qtbdef} for $\tau={\mathcal{O}}_{\N_{k,n}}$:
\be
\label{mulid}
\hat{\bold{1}}(z)= \sum\limits_{d=0}^{\infty} z^d \textrm{ev}_{p_2, *}\Big(\qm^{d}_{\textrm{relative} \, p_2}, \vss^d   \Big).
\ee

\begin{Theorem}
	$\hat{\bold{1}}(z)$ is the multiplicative identity of the quantum $K$-theory ring, i.e.,
	$\hat{\bold{1}}(z) \circledast \alpha =\alpha$ for all $\alpha \in QK_{\bT}(\N_{k,n})$.
\end{Theorem}

\begin{proof}
	We start from the identity $\bold{Id}=\bold{G}\, \cdot \bold{G}^{-1}$. Thanks to the Theorem \ref{gthm}, we can represent  it via the following diagram: 
	\vspace{0.1in}
	\begin{center}
		\begin{tikzpicture}
		\node at (-0.6,0) {$\bold{Id}=$};
		\draw [ultra thick] (0,0) -- (2,0);
		\draw [ultra thick] (0.1,0.2) to [out=210,in=150] (0.1,-0.2);
		\draw [ultra thick] (1.9,0.2) to [out=330,in=30] (1.9,-0.2);
		\node at (2.45,-0) {${\bold{G}}^{-1}$};
		\node at (2.45+0.7,-0) {$=$};
		\draw [ultra thick] (0+2.9+0.7,0) -- (2+2.9+0.7,0);
		\draw [ultra thick] (0.1+2.9+0.7,0.2) to [out=210,in=150] (0.1+2.9+0.7,-0.2);
		\draw [ultra thick] (1.9+2.9+0.7,0.2) to [out=330,in=30] (1.9+2.9+0.7,-0.2);
		\draw [fill] (1+2.9+0.7,0) circle [radius=0.1];
		\node at (1+2.9+0.7,-0.3) {$1$};
		\node at (2.45+3.6,-0) {${\bold{G}}^{-1}$};
		\node at (2.45+0.7+3.6,-0) {$=$};
		\draw [ultra thick] (0+5.75+0.5,0-1) -- (2+5.75+0.5,0-1);
		\draw [ultra thick] (1+5.75+0.5,0.2-1) -- (1+5.75+0.5,-1-1);
		\draw [ultra thick] (0.1+5.75+0.5,0.2-1) to [out=210,in=150] (0.1+5.75+0.5,-0.2-1);
		\draw [ultra thick] (1.9+5.75+0.5,0.2-1) to [out=330,in=30] (1.9+5.75+0.5,-0.2-1);
		\draw [fill] (1+5.75+0.5,-1-1) circle [radius=0.1];
		\node at (1+5.75+0.5,-1.3-1) {$1$};
		\node at (2+5.75+0.5+0.45,0-1) {${\bold{G}}^{-1}$};
		\node at (9.4,-1.65) {$=$};
		\draw [ultra thick] (0+9.85,0-1) -- (2+9.85,0-1);
		\draw [ultra thick] (1+9.85,-0.7-1) -- (1+9.85,-1.7-1);
		\draw [fill] (1+9.85,-1.7-1) circle [radius=0.1];
		\node at (1+9.85,-0.4-1) {$\bold{G}^{-1}$};
		\draw [ultra thick] (0.8+9.85,-0.8-1) to [out=60,in=120] (1.2+9.85,-0.8-1);
		\draw [ultra thick] (0.1+9.85,0.2-1) to [out=210,in=150] (0.1+9.85,-0.2-1);
		\draw [ultra thick] (1.9+9.85,0.2-1) to [out=330,in=30] (1.9+9.85,-0.2-1);
		\draw [ultra thick] (0.9+9.85,0.2-1) to [out=330,in=30] (0.9+9.85,-0.2-1);
		\node at (1+9.85,-2-1) {$1$};
		\node at (12.3,-1) {$\bold{G}^{-1}$};
		\end{tikzpicture}
	\end{center}
	\vspace{0.1in}
	The second equality is the degeneration of the parametrized domain $\mathcal{D}$ in to a nodal curve given by  union of $\mathcal{D}$ with rational curve $\mathbb{P}^1$. 
	This domain is represented by the third term in the above equalities. The last equality is the gluing formula (\ref{glform}). Comparing with the definition (\ref{qmuldef}) we see the last diagram is the operator of quantum multiplication by the class $\hat{\bold{1}}(z)$.
\end{proof}

\subsection{Quantum tautological line bundle}
Let $\widehat{\O(1)}(z)$ be the quantum tautological bundle from  Definition \ref{qtbdef} for $\tau=\O(1)=\det\tb$:
As an example of tautological class $\tau$, one can consider the line bundle over $\N_{k,n}$ given by $\tau=\mathcal{O}(1)=\det \tb$:
\be \label{qtlb}
\widehat{\O(1)}(z):=\sum\limits_{d=0}^{\infty} z^d  {\rm{ev}}_{p_2, *}\Big(\qm^d_{{\rm{relative}} \, p_2},\vss^d \otimes \det (\left.\qmV\right|_{p_1}) \Big) \in QK_{\bT}(\N_{k,n}).
\ee
The following Theorem relates it to the operator of quantum difference equation $\M(z)$ defined by (\ref{qdeop}): 

\begin{Theorem}
	\label{Mth}
	Under the specialization $q=1$   the operator $\M(z)$ coincides with the operator of quantum multiplication by the quantum line bundle:
	$$
	\M(z)|_{q=1}=\widehat{\O(1)}(z) \circledast.
	$$
\end{Theorem}
\begin{proof}
	Let us consider the operator $\M(z)$ defined by (\ref{qdeop}).
	Here the moduli space in question is the stack of quasimaps with parametrized domain ${\mathcal{D}}={\mathbb{P}}^1$, relative to the points $p_1$ and $p_2$. The bilinear form makes it an operator on $K_{\bG}(\N_{k,n})$ same way it did in the case of quantum multiplication.
The specialization $q=1$ corresponds to the non ${\mathbb{C}}^{\times}_q$-equivariant case. In this case any two points on ${\mathcal{D}}$ are isomorphic, so we can replace the sheaf $\pi^{\ast}(\mathcal{O}_{p_1})$ with $\pi^{\ast}(\mathcal{O}_{p_3})$, where $p_3$ is any other point on  ${\mathcal{D}}$. After such a replacement, $\det H^{\bullet}\left(\qmV\otimes \pi^{\ast}(\mathcal{O}_{p_1})\right)$ becomes $\det \qmV|_{p_3}$, since the point $p_3$ is not a relative point and, therefore, the map $\pi$ is an isomorphism over this point. In the diagrammatic notation this operator is
	\vspace{0.1in}
	\begin{center}
		\begin{tikzpicture}
		\draw [ultra thick] (0,0) -- (2,0);
		\draw [ultra thick] (0.1,0.2) to [out=210,in=150] (0.1,-0.2);
		\draw [ultra thick] (1.9,0.2) to [out=330,in=30] (1.9,-0.2);
		\node at (2.5,-0) {${\bold{G}}^{-1}$};
		\draw [fill] (1,0) circle [radius=0.1];
		\node at (1,-0.4) {$\det\tb$};
		\end{tikzpicture}
	\end{center}
where  $\det \tb =\O(1)$ denotes the determinant of tautological bundle over $\N_{k,n}$. Next, we use the gluing formula to represent this operator in the following form:

\vspace{0.1in}
\begin{center}
	\begin{tikzpicture}%
	\draw [ultra thick] (0+2.9+0.7,0) -- (2+2.9+0.7,0);
	\draw [ultra thick] (0.1+2.9+0.7,0.2) to [out=210,in=150] (0.1+2.9+0.7,-0.2);
	\draw [ultra thick] (1.9+2.9+0.7,0.2) to [out=330,in=30] (1.9+2.9+0.7,-0.2);
	\draw [fill] (1+2.9+0.7,0) circle [radius=0.1];
	\node at (1+2.9+0.7,-0.3) {$\O(1)$};
	\node at (2.45+3.6,-0) {${\bold{G}}^{-1}$};
	\node at (2.45+0.7+3.6,-0) {$=$};
	\draw [ultra thick] (0+5.75+0.5,0-1) -- (2+5.75+0.5,0-1);
	\draw [ultra thick] (1+5.75+0.5,0.2-1) -- (1+5.75+0.5,-1-1);
	\draw [ultra thick] (0.1+5.75+0.5,0.2-1) to [out=210,in=150] (0.1+5.75+0.5,-0.2-1);
	\draw [ultra thick] (1.9+5.75+0.5,0.2-1) to [out=330,in=30] (1.9+5.75+0.5,-0.2-1);
	\draw [fill] (1+5.75+0.5,-1-1) circle [radius=0.1];
	\node at (1+5.75+0.5,-1.3-1) {$\O(1)$};
	\node at (2+5.75+0.5+0.45,0-1) {${\bold{G}}^{-1}$};
	\node at (9.4,-1.65) {$=$};
	\draw [ultra thick] (0+9.85,0-1) -- (2+9.85,0-1);
	\draw [ultra thick] (1+9.85,-0.7-1) -- (1+9.85,-1.7-1);
	\draw [fill] (1+9.85,-1.7-1) circle [radius=0.1];
	\node at (1+9.85,-0.4-1) {$\bold{G}^{-1}$};
	\draw [ultra thick] (0.8+9.85,-0.8-1) to [out=60,in=120] (1.2+9.85,-0.8-1);
	\draw [ultra thick] (0.1+9.85,0.2-1) to [out=210,in=150] (0.1+9.85,-0.2-1);
	\draw [ultra thick] (1.9+9.85,0.2-1) to [out=330,in=30] (1.9+9.85,-0.2-1);
	\draw [ultra thick] (0.9+9.85,0.2-1) to [out=330,in=30] (0.9+9.85,-0.2-1);
	\node at (1+9.85,-2-1) {$\O(1)$};
	\node at (12.3,-1) {$\bold{G}^{-1}$};
	\end{tikzpicture}
\end{center}
\vspace{0.1in}
Thus, in accordance with our definition (\ref{qmuldef})
	the last diagram represents the operator of quantum multiplication by
	\vspace{0.1in}
	\begin{center}
		\begin{tikzpicture}
		\draw [ultra thick] (0,0) -- (2,0);
		\draw [fill] (2,0) circle [radius=0.1];
		\draw [ultra thick] (0.1,0.2) to [out=210,in=150] (0.1,-0.2);
		\node at (2.6,0) {$ \ \ \O(1)$};
		\end{tikzpicture}
	\end{center}
	\vspace{0.1in}
	
\noindent
i.e. by the quantum line bundle $\widehat{\O(1)}(z)$. 
\end{proof}
\section{Eigenvalues and eigenvectors of quantum multiplication}

\subsection{Eigenvectors of quantum multiplication}

The localized $K$-theory $K_{\bT}(\N_{k,n})_{loc}$ is a vector space over the field ${\mathcal{A}}$ with a natural basis given by the $K$-theory classes of fixed points ${\mathcal{O}}_{\fp}$ (where $\fp=\{i_1,\cdots,i_k\}\subset \{1,\cdots,n\}$ is a $k$-subset labeling the elements of $\N_{k,n}^{\bT}$),
see Section \ref{definit} for notations. The operators of classical multiplication by $\mathcal{F}\in K_{\bT}(\N_{k,n})$ act on this vector space as ${\mathcal{A}}$-linear operators, i.e. $\mathcal{F}\in End_{{\mathcal{A}}}(K_{\bT}(\N_{k,n})_{loc})$. These operators are diagonal in the basis of fixed points ${\mathcal{O}}_{\fp}$:
$$
\mathcal{F}\otimes {\mathcal{O}}_{\fp}=\left.\mathcal{F}\right|_{\fp} {\mathcal{O}}_{\fp}
$$ 
where $\left.\mathcal{F}\right|_{\fp} \in {{\mathcal{A}}}$ denotes the restriction of K-theory class to the corresponding fixed point $\fp$. 

The operators of quantum multiplication by quantum (or classical) $K$-theory classes are defined as power series 
$\mathcal{F}\circledast \in End_{{{\mathcal{A}}}}(K_{\bT}(\N_{k,n})_{loc})[[z]]$ with the zero degree terms given by classical multiplication:
$$
\mathcal{F}\circledast=\mathcal{F}\otimes + O(z)
$$
Thus, we may look for the classes $\psi_{\fp}(z) \in K_{\bT}(\N_{k,n})[[z]]$ and power series $v_{\fp}(z)\in{{\mathcal{A}}}[[z]]$ with fixed degree zero terms:
\be \label{norpow}
\psi_{\fp}(z)={\mathcal{O}}_{\fp}+O(z), \ \ \ v_{\fp}(z)=\left.\mathcal{F}\right|_{\fp} +O(z)
\ee 
such that the following formal equalities of power series hold: 
$$
\mathcal{F}\circledast \psi_{\fp}(z)= v_{\fp}(z) \psi_{\fp}(z).
$$
We will call the power series $v_{\fp}(z)$ the  eigenvalues 
of quantum multiplication by $\mathcal{F}$. 
The power series $\psi_{\fp}(z)$ normalized as in (\ref{norpow}) do not depend on $\mathcal{F}$, because the quantum multiplication is commutative. Thus, we will refer to $\psi_{\fp}(z)$ as eigenvectors of quantum multiplication.

The goal of this section is to describe $\psi_{\fp}(z)$ and to obtain explicit formulas for $v_{\fp}(z)$. In particular, we will show that the eigenvalues of the operators of quantum multiplication by the quantum tautological classes are given by the symmetric polynomials of the roots of Bethe equations. Thus, this operators have exactly the same eigenvalues as the nonlocal Hamiltonians of the XXZ spin chain.

\subsection{The capping operator and the bare vertex in the limit $q\to 1$}

Let us recall that the capping operator $\Psi(z)$ is the rank two tensor defined by (\ref{capping}). It is an element of $\bG$-localized K-theory due to 
nonsingular conditions at one of the fixed points. In particular, it becomes singular at $q=1$. In this section we describe the asymptotic behavior of the capping operator in the limit $q\to 1$.

The operator of classical multiplication by the line bundle $\O(1)$ in the basis of fixed points is represented by the diagonal matrix with the eigenvalue $\lambda^{0}_\fp=a_{i_1}\cdots a_{i_k}$ corresponding to a fixed point $\fp=\{i_1,\cdots,i_k\}$. The quantum line bundle (\ref{qtlb})  by definition is a $z$-deformation of this class:
$$
\widehat{\O(1)}(z)=\O(1)+\dots,
$$
where dots stand for terms vanishing in $z\to 0$ limit.  Thus, the eigenvalues of the quantum line bundle $\widehat{\O(1)}(z)$ have the form $\lambda_\fp(z)=\lambda^{0}_\fp+\lambda^{1}_\fp z +\cdots$.
Let us denote 
$$l_\fp(z)\equiv\lambda_\fp(z)/\lambda^{0}_{\fp}=1+l^{1}_{\fp} z+ \cdots$$ the normalized eigenvalues $l_\fp(z)$ of the operator of the quantum multiplication by the quantum line bundle $\widehat{\O(1)}(z)$.

Let us consider the $q$-difference equation (\ref{difference}) in the basis of \textit{fixed points}.
By definition, the capping operator is a $K_{\bT}(\N_{k,n})^{\otimes 2}$ -valued power series of the form:
$$
\Psi(0)=1+\Psi_1 z +\Psi_2 z^2 +\dots,
$$
i.e. the first term is the identity matrix.  At $z=0$ the equation  (\ref{difference}) holds trivially because
$\M(0)=\O(1)$. The higher terms $\Psi_i$ for $i>1$ are fixed by (\ref{difference}) and can be computed explicitly by solving corresponding linear problem
using known operator $\M(z)$. The following result describes $q\to 1$ behavior of the fundamental  solution matrix $\Psi(z)$.
\begin{Proposition}
	Let $\Psi_{{\fp}}(z)$ be the $\fp$-th column of the matrix $\Psi(z)$ in the basis of fixed points.
	In the limit $q\mapsto 1$ the capping operator has the following asymptotic behavior:
	\be
	\label{solution}
	\Psi_{\it{\fp}}(z) =\exp \Big({\frac{1}{q-1}\int d_q z \ln (l_\fp(z)) } \Big)\Big(\psi_\fp(z) + \cdots\Big),
	\ee
	where $\psi_{\fp}(z)$ are the column eigenvectors of the operators of quantum multiplication corresponding to the fixed point $\fp$ and dots stand for the terms vanishing in the limit $q\to 1$.
	\footnote{Let ${(k)}_q:=\frac{q^k-1}{q-1}$ stand for the ``$q$-number'' for a formal parameter $q\in \C^{\times}$ with $|q|<1$.
		Recall that the Jackson $q$-integral is the formal sum:
		$$
		\int d_q z f(z)=(1-q) \sum\limits_{n=0}^{\infty} f(z q^n).
		$$
		The Jackson integral is well defined as an operation on power series without constant term:
		$$
		\int d_q z: \sum\limits_{k=1}^{\infty}a_k z^k \mapsto \sum\limits_{k=1}^{\infty} \frac{a_k}{{(k)}_q}z^k.
		$$
		It is clear from this definition that the $q$-integral solves corresponding $q$-difference equation:
		$$
		F(z)=\int d_q z f(z) \ \ \ \Rightarrow \ \ \ \frac{F(qz)-F(z)}{q-1} = f(z),
		$$
		which can be also written in the following exponential form:
		$$
		F(z)=\exp \Big(\frac{1}{q-1} \int d_q z \ln \lambda(z) \Big)  \ \ \  \Rightarrow \ \ \ F(z q) =\lambda(z) F(z).
		$$
	}
\end{Proposition}
\begin{proof}
	By Theorem \ref{Mth} at $q=1$ the operator 	$\M(z)$ coincides with the operator of quantum multiplication by a tautological line bundle. We denote by $\psi_{\fp}(z)$ the eigenbasis of this operator.


The difference equation in the fixed point basis can be rewritten in the following way, denoting $q=e^h$:
\begin{eqnarray}
\lambda^0_{\mathbf{p}}\Psi_{\fp}(e^hz)=\mathbf{M}(z)\Psi_{\fp}(z), \quad  	\mathbf{M}(z)=\widehat{\mathcal{O}}(1)(z)+\dots
\end{eqnarray}	
where dots stand for $O(h)$ when $h\to 0$. The vector solutions to the difference equation above are the deformations of the formal solutions to the difference equation 	
\begin{eqnarray}\nonumber 
\lambda^0_{\mathbf{p}}\Phi_{\fp}(e^hz)=\widehat{\mathcal{O}}(1)(z)\Phi_{\fp}(z)
\end{eqnarray}	
which have the form
$$
\Phi_p=\exp \Big({\frac{1}{q-1}\int d_q z \ln (l_\fp(z)) } \Big)\psi_\fp(z), 
 $$
 so that $\Psi_p=\exp \Big({\frac{1}{q-1}\int d_q z \ln (l_\fp(z)) } \Big)(\psi_\fp(z)+\dots)$, where $\dots$ stand for $O(h)$. 
\end{proof}
\begin{Corollary} \label{idcorr}
	In the limit $q \to 1$ the coefficients of bare vertex in the basis of fixed points have the following asymptotic expansion:
	$$
	V^{(1)}_{\fp}(z)=(v_\fp(z) +\cdots) \exp \Big(-{\frac{1}{q-1}\int d_q z \ln (l_\fp(z)) } \Big),\\
	$$
	where dots  stand for the terms vanishing at $q\to 1$ and $v_k(z)$ are the coefficients of multiplicative identity element of the quantum $K$-theory ring in the basis of eigenvectors of quantum multiplication:
	\be
	\label{expansion}
	\hat{\bold{1}}(z) = \sum\limits_{\fp} v_{\fp}(z) \psi_{\fp}(z).
	\ee
\end{Corollary}
\begin{proof}
	Recall that the multiplicative identity in the quantum $K$-theory is given by $\hat{\bold{1}}(z)=\lim\limits_{q\to 1}\hat{V}^{(1)}(z)$. Hence, we have:
	\be
	\hat{\bold{1}}(z)=\lim\limits_{q\to 1} \Psi(z) V^{(1)}(z) =\lim\limits_{q\to 1} \sum\limits_{\fp}\,
	\Psi_{\fp}(z) V^{(1)}_{\fp}(z).\nonumber
	\ee
The existence of the latter limit follows from the fact that the capped vertex is an element of the integral $K$-theory and thus is a Laurent polynomial in $q$.  This means that exponentially divergent terms from (\ref{solution}) are canceled by the corresponding contributions from $V^{(1)}_{\fp}(z)$. However, due to the linear independence of the eigenvectors $\psi_{\fp}(z)$ that is possible only if these coefficients have the form:
	$$
	V^{(1)}_{\fp}(z)=w_\fp(z) \exp \Big(-{\frac{1}{q-1}\int d_q z \ln (l_\fp(z)) } \Big).
	$$
	Taking $v_{\fp}(z)=w_{\fp}(z,q=1)$, we obtain the desired result.
\end{proof}
\subsection{Eigenvalues of the quantum multiplication by quantum tautological bundles}

Let us consider the operators of quantum multiplication by quantum tautological classes, i.e., the operators $\hat{\tau}(z)\circledast$. As operators of quantum multiplication these operators are also diagonal in the basis $\psi_{\fp}(z)$.  Let $\tau_{\fp}(z)$ denote the eigenvalue corresponding to the eigenvector $\psi_{\fp}(z)$.
\begin{Proposition} The coefficients of the bare vertex function  have the following  $q\to 1$ asymptotic
	in the fixed points basis:
	\be
	\label{bvwd}
	V^{(\tau)}_{\fp}(z)=(\tau_{\fp}(z) v_\fp(z) +\cdots) \exp \Big(-{\frac{1}{q-1}\int d_q z \ln (l_\fp(z)) } \Big),
	\ee
	where dots stand for the terms vanishing in the limit $q\to 1$ and coefficients $v_\fp(z)$ are as stated in Corollary \ref{idcorr}.
\end{Proposition}
\begin{proof}
	Following the argument of Corollary \ref{idcorr}, we obtain that the following limit exists:
	$$
	\hat{\tau}(z)=\lim\limits_{q\to 1} \Psi(z) V^{(\tau)}(z) =\lim\limits_{q\to 1} \sum\limits_{\fp}\,
	\Psi_{\fp}(z) V^{(\tau)}_{\fp}(z).
	$$
	Thus, the coefficients of the descendent vertex must be of the following form:
	$$
	V^{(\tau)}_{\fp}(z)=w_\fp(z,q) \exp \Big(-{\frac{1}{q-1}\int d_q z \ln (l_\fp(z)) } \Big)
	$$
	for some $w_\fp(z,q)$ regular at $q=1$. Therefore, $\hat{\tau}(z)=\sum\limits_{\fp} w_\fp(z,1) \psi_{\fp}(z)$.
	However, at the same time
	$$
	\hat{\tau}(z)=\hat{\tau}(z) \circledast \hat{1}(z) =\hat{\tau}(z)  \circledast \sum\limits_{\fp} v_{\fp}(z)  \psi_{\fp}(z) =\sum_{\fp} \tau_{\fp}(z) v_{\fp}(z)  \psi_{\fp}(z)
	$$
	and thus, the result follows from linear independence of $\psi_{\fp}(z)$.
\end{proof}
We conclude that the eigenvalues of the operator of quantum multiplication by $\hat{\tau}(z)$ can be computed from the
asymptotics of the bare vertex functions.
\begin{Corollary}
	The eigenvalues of the operator of quantum multiplication by $\hat{\tau}(z)$ can be expressed as follows:
	\be
	\label{eval}
	\tau_{\fp}(z)=\lim\limits_{q \rightarrow 1 } \dfrac{V^{(\tau)}_{\fp}(z)}{V^{(1)}_{\fp}(z)}
	\ee
\end{Corollary}
In the next section the  bare vertex functions with descendants $V^{(\tau)}_{\fp}(z)$  are  computed using equivariant localization in K-theory. It turns out that these functions are given by some standard $q$-hypergeometric series expansions, so that the limit (\ref{eval}) provides an explicit formula for the eigenvalues of quantum tautological bundles. Namely, these eigenvalues coincide with the symmetric polynomials of the solutions of Bethe equations. 

\subsection{Virtual tangent bundle and the structure sheaf on $\N_{k,n}$}
Before we proceed to the actual vertex computation, one has to discuss the version of virtual structure sheaf more suitable for computations as well as teh virtual tangent bundle, discussed in higher generality for Nakajma varieties in \cite{lectok}, Section 6.1. 

The moduli spaces of nonsingular quasimaps ${{\qm}^d}_{\text{nonsing}\, p}$ and relative quasimaps ${{\qm}^d}_{\text{relative}\, p}$ above have a perfect deformation-obstruction theory \cite{ionut}. This allows one to construct a tangent virtual bundle $T^{\textrm{vir}}$, a virtual structure sheaf $\vss$ and a virtual canonical bundle over these moduli spaces. We will define multiplication in the quantum  $K$-theory using this data. Without going into detail of the construction of this virtual sheaf, we state the formula of the reduced virtual tangent bundle.  Let $(\qmV,\qmW)$ be data defining a point on ${{\qm}^d}_{\text{nonsing}\, p}$. We define the fiber of the reduced virtual tangent
bundle to ${{\qm}^d}_{\text{nonsing }\, p}$ at this point to be equal to:
\be
T^{\textrm{vir}}_{(\qmV,\qmW)} {{\qm}^d}_{\text{nonsing  p}}=H^{\bullet}(\qmP\oplus\hbar\, \qmP^{*} ) - T_{f(p)} \N_{k,n}
\ee
where $T_{f(p)} \N_{k,n}$ is a normalizing term, the essence of which will be explained later, and  $\qmP$ is the polarization bundle on the curve $\mathcal{C}$:
$$
\qmP=\qmW^{*} \otimes \qmV - \qmV^{*} \otimes \qmV.
$$
The symmetrized virtual structure sheaf is defined by:
\be
\label{virdef}
{\vss}=\O_{\textrm{vir}}\otimes {\mathscr{K}^{1/2}_{\textrm{vir}}} q^{\deg(\qmP)/2},
\ee
where $\O_{\textrm{vir}}$ is a standard structure sheaf and $\mathscr{K}_{\textrm{vir}}={\det}^{-1}T^{\textrm{vir}} \qm^{d}$ is the virtual canonical bundle.

Since we will be using the symmetrized virtual structure sheaf, we need to adjust the standard bilinear form on $K$-theory. Let $(\cdot,\cdot)$ be given by the following formula:
\be
{(\mathcal{F},\mathcal{G})}=\chi(\mathcal{F}\otimes\mathcal{G}\otimes{K}^{-1/2}),
\ee
where $K$ is the canonical class.

\subsection{Explicit formula for the bare vertex with descendants\label{comver}}
It will be convenient to adopt the following notations:
\be
\phi(x)=\prod^{\infty}_{i=0}(1-q^ix),\quad  \{x\}_{d}=\dfrac{(\hbar/x,q)_{d}}{(q/x,q)_{d}} \, (-q^{1/2} \hbar^{-1/2})^d, \ \ \textrm{where}  \ \ (x,q)_{d}=\frac{\varphi(x)}{\varphi(q^dx)}.\nonumber
\ee
\begin{Proposition} \label{prover}
	Let $\fp=\{x_1,\cdots,x_k\} \subset \{a_1,\cdots,a_n\}$ be a $k$-subset defining a torus fixed point $\fp\in \N^{\bT}_{k,n}$.
	Then the coefficient of the vertex function at this point is given by the following $q$-hypergeometric function:
	$$
	V^{(\tau)}_{\fp}(z) = \sum\limits_{d_1,\cdots, d_k\in {\mathbb{Z}_{\ge 0}}}\, z^{d} q^{n d/2}\, \prod\limits_{i,j=1}^{k}\{x_i/x_j\}^{-1}_{d_i-d_j}   \prod\limits_{i=1}^{k} \prod\limits_{j=1}^{n} \{x_i/a_j\}_{d_i} \tau(x_1 q^{-d_1},\cdots, x_k q^{-d_k}),
	$$
	where $d=\sum^k_{i=1}d_i$.
\end{Proposition}
\begin{proof}
	If $\tb$, $\tbW$ are the tautological bundles on $\N_{k,n}$, then the tangent bundle has the form:
	$$
	T \N_{k,n} = \tbP + \hbar \tbP^* \ \  \textrm{for} \ \  \tbP=\tbW^{*}\otimes \tb - \tb^{*} \otimes \tb.
	$$
	Recall that the degree $d$ quasimaps to $\N_{k,n}$ are given by a pair of bundles: rank $k$, degree $d$ bundle
	$\qmV$ and rank $n$ trivial bundle $\qmW$ on ${\mathbb{P}}^1$. Let us consider the set of $\bT$-fixed points  $(\qmV,\qmW) \in  (\qm^{d}_{\text{nonsing} \, p_2})^{\bT}$ such that the value of the evaluation map at $p_2$ is $\fp \in \N_{k,n}^{\bT}$.
	The virtual tangent space is a representation of the torus $\bT$. We denote
	its $\bT$-character by:
	\be \label{pcontr}
	\chi(d)=\textrm{char}_{\bT}\Big( T^{vir}_{(\qmV,\qmW)} \qm^{d}  \Big).
	\ee
	Localization in $K$-theory gives the following formula for the equivariant pushforward:
	$$
	V^{(\tau)}_{\fp}(z)=\sum\limits_{d=0}^{\infty}\sum\limits_{(\qmV,\qmW) \in  (\qm^{d}_{\text{nonsing} \, p_2})^{\bT}}\, \hat{s}(  \chi(d) )\, z^{d} q^{\deg(\qmP)/2} \tau (\left.\qmV\right|_{p_1}),
	$$where the sum runs over the $\bT$-fixed quasimaps which take value $\fp$ at the nonsingular point $p_2$.  We use notation $\hat{s}$ for the Okounkov's roof function defined by
	$$
	\hat{s}(x)=\dfrac{1}{x^{1/2}-x^{-1/2}}, \ \ \ \hat{s}(x+y)=\hat{s}(x)\hat{s}(y).
	$$
	Note, that the tangent weight contribute to vertex via the roof function $\hat{s}(x)$ because the symmetrized virtual structure sheaf (\ref{virdef})
	is defined together with a shift on the square root of canonical bundle ${\mathscr{K}^{1/2}_{\textrm{vir}}}$.
	Thus, our goal is to compute (\ref{pcontr}).  The \textit{reduced} virtual tangent space to $\qm^{d}_{\textrm{nonsing} \, p_2}$
	at such point is given by\footnote{We use the reduced virtual tangent space which differs from standard one by subtracting $T_{\fp} \N_{k,n}$. This term does not depend on $d$ and thus produces a simple multiple in the vertex function. This is the multiple normalizing the  vertex such that $V^{(\tau)}_{\fp}(0)=\tau$. }:
	\be
	\label{vtan}
	T^{vir}_{(\qmV,\qmW)} \qm^{d}=H^{\bullet}(\qmP\oplus\hbar\, \qmP^{*} ) - T_{\fp} \N_{k,n},
	\ee
	where $\qmP$ is the polarization bundle $ \qmP=\qmW^{*} \otimes \qmV - \qmV^{*} \otimes \qmV.$
	
	The following Lemma drastically simplifies the computation of the contribution of $\textrm{char}_{\bT}\Big( T^{vir}_{(\qmV,\qmW)} \qm^{d}  \Big)$ to the localization formula.
	\begin{Lemma}
		\label{chlemma}
		Let $\qmP$ be a polarization  bundle on $\mathbb{P}^1$  corresponding to a  $\bT$-fixed point on
		$\qm^{d}_{\textrm{nonsing} \, p_2}$. It splits into a sum of $\bT$-equivariant line bundles $\qmP=\bigoplus_{i} a_i q^{-d_i} \O(d_i)$
		with for some characters $a_i$ of the framing torus $\bA$. The cohomology of these line bundles
		have the following  $\bT$-characters:
		\be
		\label{charH}
		&&\textrm{char}_{\bT}\Big(H^{\bullet}(a_i q^{-d_i} \O(d_i))\Big)=a_i \dfrac{q^{-d_i-1}-1}{q^{-1}-1} = \nonumber\\
		&&\begin{cases}
			a_i+a_i q^{-1}+\cdots +a_i q^{-d_i} & {\emph{if}}\quad  d_i>0\\
			0 & {\emph{if}}\quad  d_i=-1\\
			-a_iq-a_iq^2-\cdots -a_i q^{-d_i-1} & {\emph{if}}\quad  d_i<-1.\\
		\end{cases}
		\ee
	\end{Lemma}
	
	\begin{proof}
		It is clear that the tautological bundles $\qmV$ and $\qmW$ representing $\bT$-fixed quasimap split into the sum of line bundles equivalently. It means that $\qmP=\bigoplus_{i} x_i \O(d_i)$
		for some $\bT$ -characters $x_i$. Since the quasimap is nonsingular at $p_2=\infty$ the corresponding section should not vanish at $p_2$. The only such section
		of $\O(d_i)$ is $z^{d_{i}}$. The torus $\bT$ acts on sections by $z\to q z$. By assumption, this section must be $\bT$-fixed. It is possible only if $x_i= a_i q^{-d_i}$ for some
		character $a_i$ of framing torus $\bA$, which does not act on $\mathbb{P}^1$. Finally, if $d_i\geq 0$ then only zeroth cohomology group $H^{0}(\O(d_i))$ is nontrivial and is spanned by global sections $1,z,\cdots, z^{d_i}$. Thus, we obtain (\ref{charH}). For  $d_i<0$, applying the Serre duality one obtains same result.
	\end{proof}

	By Lemma \ref{chlemma}, the polarization bundle representing a fixed point on the moduli space of quasimaps splits to a sum of linear subbundles. By multiplicativity
	of the roof function it is enough to compute the contribution of a single line subbundle $x \O(d) \subset \qmP(d)$ to the weight of the fixed point.
	According to Lemma \ref{chlemma} the contribution of such bundle to (\ref{vtan}) is given by:
	\be
	&&x \dfrac{q^{-d-1}-1}{q^{-1}-1} + x^{-1} \hbar \dfrac{q^{d-1}-1}{q^{-1}-1}-x-x^{-1} \hbar=\nonumber\\
	&&\begin{cases}
		x q^{-1} (1+q^{-1}+\cdots + q^{-d+1})-x^{-1} \hbar
		(1+q+\cdots +q ^{d-1})\quad {\rm if} \quad d>0,\\
		x^{-1}\hbar q^{-1} (1+q^{-1}+\cdots + q^{d+1})-x
		(1+q+\cdots +q ^{-d-1})\quad {\rm if} \quad d<0.
	\end{cases}
	\ee
	Applying the roof function we find:
	\be\nonumber
	\begin{cases}
		\hat{s}\Big(x q^{-1} (1+q^{-1}+\cdots + q^{-d+1})-x^{-1} \hbar
		(1+q+\cdots +q ^{d-1})\Big)=\{x\}_{d} \quad {\rm if} \quad d>0,\\
		\hat{s}\Big(x^{-1}\hbar q^{-1} (1+q^{-1}+\cdots + q^{d+1})-x
		(1+q+\cdots +q ^{-d-1})\Big)=\{x\}_{d} \quad {\rm if} \quad d<0.
	\end{cases}
	\ee

	It is clear that the fixed points  on the moduli space of quasimaps taking value $\fp=\{x_1,\cdots,x_k\}$ in the nonsingular point, correspond to the bundles of the form:
	\be
	\label{tbf}
	\qmV=\O(d_1) q^{-d_1} x_1 \oplus \cdots \O(d_k) q^{-d_k} x_k, \ \ \ \qmW=\O a_1 \oplus \cdots \oplus \O a_n,
	\ee
	where $d_1+\cdots+d_k=d$. Hence, the terms $\qmW^{*}\otimes \qmV$  and  $-\qmV^{*}\otimes \qmV$ in the polarization produce the following contributions:
	$$
	\qmW^{*} \otimes \qmV \longrightarrow \prod\limits_{j=1}^{n} \{x_i/a_j\}_{d_i} \ \ \ -\qmV^{*} \otimes\qmV \longrightarrow  \prod\limits_{i,j=1}^{k}\{x_i/x_j\}^{-1}_{d_i-d_j}.
	$$
	Note, that $\deg(\qmP) = n d$. That gives the polarization term $q^{n d/2}$ in the vertex. Finally, from (\ref{tbf}) we obtain
	$\tau (\left.\qmV\right|_{p_1})=\tau(x_1 q^{-d_1},\cdots,x_k q^{-d_k})$ concluding the computation.
\end{proof}

\subsection{Integral representation for bare vertex function with descendent}
Similarly to  standard $q$-hypergeometric series, the vertex function has a  Mellin - Barnes type integral representation: 
\begin{Prop}
	\be
	\label{verint}
	&&V^{(\tau)}_{\fp}(z)= \nonumber\\
	&& \dfrac{1}{2 \pi i \alpha_p} \int\limits_{C_p} \prod\limits_{i=1}^{k} \dfrac{d s_i}{s_i} \, e^{-\frac{\ln(z_{\sharp}) \ln(s_i) }{\ln(q)}} \,  \prod\limits_{i,j=1}^{k} \dfrac{\varphi\Big( \frac{s_i}{ s_j}\Big)}{\varphi\Big(\frac{q}{\hbar} \frac{s_i}{ s_j}\Big)} \prod\limits_{i=1}^{n}\prod\limits_{j=1}^{k} \dfrac{\varphi\Big(\frac{q}{\hbar} \frac{s_j}{ a_i }\Big)}{\varphi\Big(\frac{s_j}{ a_i }\Big)}
	\tau(s_1,\cdots,s_k).
	\ee
	Here  the contour of integration $C_p$, corresponding to a fixed point $p=\{x_1,\cdots,x_k\}\subset\{a_1,\cdots,a_n\}$
	is a positively oriented contour enclosing the poles at $s_i=q^{-d_i} x_i$ for $i=1, \dots, k$, $d_i\in \mathbb{Z}_{\ge 0}$. We  also  used a shifted degree counting parameter $z_{\sharp}=(-1)^n  \hbar^{n/2}z$ and $\alpha_p$ is a normalization constant:
	$$
	\alpha_p=\prod\limits_{i,j=1}^{k} \dfrac{\varphi\Big(\frac{x_i}{x_j}\Big)}{\varphi\Big(\frac{q x_i}{\hbar x_j}\Big)}\,
	\prod\limits_{i=1}^{n}\prod\limits_{j=1}^{k} \dfrac{\varphi\Big(\frac{q x_j}{\hbar a_i}\Big)}{\varphi\Big(\frac{x_j}{a_i}\Big)} \prod\limits_{i=1}^{k}  e^{-\frac{\ln(z_{\sharp}) \ln(x_i) }{\ln(q)}}.
	$$
\end{Prop}
\begin{proof}In order to prove this Proposition one has to
	evaluate residues at zeroes of the function $\varphi\Big(\frac{s_j}{ a_i }\Big)$. One can see, for example, that the contribution of one fraction $\dfrac{\varphi\Big(\frac{q}{\hbar} \frac{s_j}{ a_i }\Big)}{\varphi\Big(\frac{s_j}{ a_i }\Big)} $ is:
	\be
	\frac{\Big(1-q^{-d_{j}+1}\frac{x_j}{\hbar a_i}\Big)\Big(1-q^{-d_{j}+2}\frac{x_j}{\hbar a_i}\Big)\cdots  \Big(1-\frac{x_j}{\hbar a_i}\Big)}{\Big(1-q^{-d_{j}}\frac{x_j}{a_i}\Big)\Big(1-q^{-d_{j}+1}\frac{x_j}{a_i}\Big)\cdots  \Big(1-q^{-1}\frac{x_j}{ a_i}\Big)}=
	(-1)^{d_j}\Big(\frac{q}{\hbar}\Big)^{d_j/2}\{x_j/a_i\}_{d_j}.\nonumber
	\ee
	These extra coefficients provide shifts in the $z$-variable as well as necessary extra $q$-contributions. Combining contributions from all other terms we obtain the statement of the  Proposition.
\end{proof}
Such integral representation is convenient for the computation of the $q\to 1$ asymptotical behavior of the vertex function. It is well known
that in this limit a single term in the $q$-hypergeometric series dominates. In other words, in this limit, the integral (\ref{verint})
converges to its saddle point approximation and we arrive to the following Proposition:
\begin{Prop}
	At $q\to 1$ the saddle point of the integral (\ref{verint}) is determined by Bethe equations:
	\be
	\label{baeq}
	\prod\limits_{j\neq i} \,\dfrac{s_i-s_j \hbar}{s_i \hbar-s_j} \prod\limits_{j=1}^{n} \, \dfrac{s_i-a_j}{a_j \hbar - s_i}=z \hbar^{-n/2}, \ \ i=1,\dots,k.
	\ee
\end{Prop}
\begin{proof}
	Let $\Phi$ denote the integrand in (\ref{verint}). The saddle point is defined by the equations:
	$s_i \partial_{s_i} \ln (\Phi) =0$ for $i=1,\dots,k$.
	Let us prove the following Lemma:
	\begin{Lemma}
		Asymptotically, in the limit $q\to 1$ we have $$ x \frac{\partial \ln \varphi(x) }{ \partial x} = -\frac{\ln(1-x)}{\ln(q)} +o(\ln(q)).$$
	\end{Lemma}
	\begin{proof}To show that, one has to expand the expression for $x \frac{\partial \ln \varphi(x) }{ \partial x}$:
		\be
		&& x \frac{\partial \ln \varphi(x) }{ \partial x} = -\sum^{\infty}_{i=0}\frac{q^ix}{1-q^ix}=-\sum^{\infty}_{i=0}\sum^{\infty}_{m=0}q^{i(m+1)}x^{m+1}=-\sum^{\infty}_{m=0}\frac{x^{m+1}}{1-q^{m+1}}=\nonumber\\
		&&-\sum^{\infty}_{m=0}\frac{x^{m+1}}{1-e^{(m+1)\ln(q)}}=-\frac{\ln(1-x)}{\ln(q)} +o(\ln(q)).\nonumber
		\ee
	\end{proof}
	\noindent Hence, we obtain the following equations for the saddle point:
	\be \nonumber
	-z_{\sharp} +\sum_{j=1}^{n} \Big(\ln(1-\frac{s_i}{a_j})-  \ln(1-\frac{s_i}{\hbar a_j})\Big)+\\ \nonumber
	\sum\limits_{j\neq i}\Big(
	-\ln(1-\frac{s_i}{s_j})+\ln(1-\frac{s_j}{s_i})+\ln(1-\frac{s_i}{\hbar s_j})-\ln(1-\frac{s_j}{\hbar s_i})  \Big)=0,
	\ee
	which after exponentiation gives (\ref{baeq}).
\end{proof}

Let us summarize the results of this section:

\begin{Theorem}\label{symmeig}
	The operators of quantum multiplication by quantum tautological classes:
	$$
	{\mathcal{F}} \rightarrow \hat{\tau}(z) \circledast {\mathcal{F}}
	$$	
	are diagonal in the basis $\psi_{\fp}(z)$. The corresponding eigenvalues are given
	by values of the symmetric polynomial $\tau(s_1,\cdots,s_k)$  at the solutions of the Bethe equations (\ref{baeq}), corresponding to $\fp$.
\end{Theorem}
\begin{proof}
	Formula (\ref{eval}) states that the eigenvalue is given by the limit:
	$$
	\tau_{\fp}(z)=\lim\limits_{q \rightarrow 1 } \dfrac{V^{(\tau)}_{\fp}(z)}{V^{(1)}_{\fp}(z)}.
	$$
	Let $\Phi^{(\tau)}(s_1,\cdots,s_k)$ denote the integrand in (\ref{verint}). In the limit $q\to 1$
	the vertex  functions are divergent with leading term given by the saddle point approximation.
	By the previous proposition is means that
	$$
	\tau_{\fp}(z)=\dfrac{\Phi^{(\tau)}(s_1,\cdots,s_k)}{\Phi^{(1)}(s_1,\cdots,s_k)},
	$$
	where $s_i$ satisfy the saddle point equations (\ref{baeq}). The latter ratio is $\tau(s_1,\cdots,s_k)$.
\end{proof}

The last theorem says that the operators of quantum multiplication by quantum classes have the same eigenvalues as the nonlocal Hamiltonians of XXZ spin model. Therefore, to identify these operators with nonlocal Hamiltonians, we need to prove that the corresponding eigenbases are equivalent, i.e., to prove that $\psi_{\fp}(z)$  is the basis of Bethe vectors. Proving this last step occupies the rest of the paper.


\section{$\fsh$ action on $K_{\bT}(\N(n))$ \label{gsec}}
In this section we recall the construction of $\fsh$ action on $K_{\bT}(\N(n))$. We describe the isomorphism of $\fsh$ modules $K_{\bT}(\N(n))=\C^2(a_1)\otimes\cdots\otimes \C^2(a_n)$. Explicitly, this isomorphism maps the standard basis
of $\C^2(a_1)\otimes\cdots\otimes \C^2(a_n)$ to the K-theoretic stable basis of  $K_{\bT}(\N(n))$. 

\subsection{Geometric $R$-matrices and RTT-procedure}
Let us consider the variety $\N(1)$. By definition, it is a union of two components corresponding to the Grassmannians of zero- and one-dimensional hyperplanes in one-dimensional vector space. Geometrically, both components are points and thus $K_{\bT}(\N(1))$ is a two-dimensional vector space over the field $\mathcal{A}$. 
Let us denote by $a$  the equivariant parameter of $\bA\cong\C^{\times}$ and by
$
\mathbb{F}(a):=K_{\bT}(\N(1))
$
the corresponding two-dimensional vector space.

The set of $\bA$-fixed points of $\N(n)$ has the following description:
$$
\N(n)^{\bA}=\N(1)^n=\N(1)\times \cdots \times \N(1). 
$$
This is a general property of the Nakajima quiver varieties
known as \textit{tensor product structure}, see Section 2.4 of \cite{mo}. After localization this implies the following isomorphism of $\mathcal{A}$-vector spaces:
\be \label{tpvect}
K_{\bT}(\N(n)^{\bA})=\mathbb{F}(a_1)\otimes \cdots \otimes \mathbb{F}(a_n).
\ee 
For a permutation $\sigma \in \frak{S}_{n}$ we fix a chamber in the real Lie algebra of $\bA$ given in coordinates by $a_{\sigma(1)}>\cdots>a_{\sigma(n)}$. Let the \textit{slope} $s$ be a generic element of $\Pic(\N(n))\otimes~\mathbb{R}\cong~\mathbb{R}$.
For every such choice of $\sigma$ and $s$ Maulik and Okounkov defined canonical maps:
$$
\textrm{Stab}_{\sigma}^{s}:\ \ K_{\bT}(\N(n)^{\bA}) \longrightarrow K_{\bT}(\N(n)),
$$
known as K-theoretic stable envelope, see Section 2.1 in \cite{os} or Section 9 in \cite{lectok} for the definition of this map. After localization, $\textrm{Stab}_{\sigma}^{s}$ become isomorphisms of vector spaces, in particular these maps are invertible.  The explicit formulas for this map in the case of $\N(n)$ and more generally, for flag varieties, can be found in Section 5 of \cite{rtv}.

For a given transposition $\sigma=(k,m)$ one can define the following linear operators
$$
\mathcal{R}^{s}_{k,m}:=(\textrm{Stab}_{(k,m)}^{s})^{-1}
\circ \textrm{Stab}^{s}_{\textrm{id}},
$$
known as geometric $R$-matrices. By  definition, they satisfy the quantum Yang-Baxter equations:
\be \label{qybe}
\mathcal{R}^{s}_{i,j} \mathcal{R}^{s}_{i,k} \mathcal{R}^{s}_{j,k}=\mathcal{R}^{s}_{j,k} \mathcal{R}^{s}_{i,k} \mathcal{R}^{s}_{i,j}
\ee
for an arbitrary triple $i,j,k$ and a slope $s$. 

By RTT-procedure \cite{DF}, \cite{frt} the set of $R$-matrices
endows the vector spaces $\mathbb{F}(a_i)$ with a structure of 
modules over certain Hopf  algebra, which we denote by $\mathcal{U}(\widehat{\frak{g}})$, see Section 3 of \cite{os} for the definitions. The Hopf structure gives the tensor product (\ref{tpvect}) a structire of module over $\mathcal{U}(\widehat{\frak{g}})$. 

\subsection{Identifications} 
Let $\C^{2}(a)$ denote two-dimesnional evaluation module
of quantum affine algebra $\fsh$.  Let $e_1,e_2$ be the standard basis of $\C^{2}(a)$ corresponding to $1$ and $0$ weight vectors. In the language of spin chains 
these correspond to ``spin up'' and ``spin down'' states.

Let us consider a $\fsh$-module given by the tensor product of two-dimensional evaluation modules 
$
\C^2(a_{1})\otimes \cdots \otimes \C^2(a_{n}).
$
The \textit{standard  basis} of weight $k$-subspace in this module is labeled by $k$-subsets ${\bf p}=\{i_1,\dots, i_k\} \subset \mathbf{n}=\{1,\dots, n\}$. Such a subset corresponds to the basis vector $v_{\bf p}=e_{k_1}\otimes\cdots \otimes e_{k_n}$
with $e_1$ at the positions $i\in {\bf p}$ and $e_2$ at  $i\in {\bf n}\setminus {\bf p}$. 

\begin{Proposition} \label{procan}
Let $s_{0}$ be a slope in ``anticanonical alcove'', i.e., $s_0 = \epsilon [{\mathcal{O}}(1)]$ with $\epsilon \in (-1,0)$. Then the geometric $R$-matrix $\mathcal{R}^{s_0}_{l,l+1}$ evaluated in the basis of fixed points of (\ref{tpvect}) coincides with the trigonometric $\fsh$ $R$-matrix $R_{l,l+1}(a_{l}/a_{l+1})$ evaluated in the standard basis of $v_{\bf p}$. 	
\end{Proposition}

\begin{proof}
For an arbitrary  slope the geometric $R$-matrices $\mathcal{R}^s_{l,l+1}$ were computed explicitly in Section 5.3 of \cite{rtv}. (For $n=2$ see also Section 7.1.8  in \cite{os}). In case of  $s=s_{0}$, one observes that the corresponding $R$-matrices coincide precisely with well known trigonometric  $R$-matrices of $\fsh$ evaluated in the standard basis. 
\end{proof}

\begin{Corollary} 	

1) $\mathcal{U}(\widehat{\frak{g}})\cong \fsh$.

2) For every $n$ the map: 
$$
 \C^{2}(a_1) \otimes\cdots
\otimes \C^{2}(a_n) \to K_{\bT}(\N(n)^{\bA})
$$
sending a standard basis vector $e_{{\bf p}}$ to the class of the corresponding fixed point in~$K_{\bT}(\N(n)^{\bA})$,
is an isomorphism of $\fsh$-modules. In particular $\mathbb{F}(a)=\C^{2}(a)$.
\end{Corollary}	
\begin{proof}
The RTT procedure constructs the corresponding quantum groups from $R$-matrices. Thus, 1) follows from the previous proposition. The action of $\fsh$ in $\C^{2}(a_1) \otimes\cdots
\otimes \C^{2}(a_n)$ in the standard basis coincides with its action on $K_{\bT}(\N(n)^{\bA})$ in the basis of fixed points because the corresponding $R$-matrices coincide in these bases.  
\end{proof}

\begin{Definition}
In this article we fix an isomorphism of $\fsh$-modules	
$$
K_{\bT}(\N(n))_{loc}=\C^{2}(a_1) \otimes\cdots
\otimes \C^{2}(a_n) 
$$
given by the stable envelope map 
\be \label{fisom}
\textrm{Stab}_{\mathrm{id}}^{s_{0}}:\ \ K_{\bT}(\N(n)^{\bA}) \longrightarrow K_{\bT}(\N(n)). 
\ee
where $\mathrm{id}$ denotes the trivial permutation and $s_{0}$
is a slope from anticanonical alcove as in Proposition \ref{procan}. 
\end{Definition}
{ \bf In particular, the standard basis of the Hilbert space of 
$XXZ$ spin chain $\C^{2}(a_1) \otimes\cdots
\otimes \C^{2}(a_n)$ is identified with the basis of K-theoretic stable envelope classes of the corresponding fixed points. }

The fixed isomorphism of modules is not canonical. It requires  a choice of chamber and slope. In the next subsection we discuss generic choices.

\subsection{Other choices of chambers and slopes} 
Let us choose a generic chamber $\sigma$ and consider an isomorphism 
$$
\textrm{Stab}_{\sigma}^{s_{0}}:\ \ K_{\bT}(\N(n)^{\bA}) \longrightarrow K_{\bT}(\N(n))_{loc}
$$
instead of the isomorphism (\ref{fisom}). Two identifications are related by the operator
$$
{\mathcal{R}}_{\sigma}: \C^{2}(a_1) \otimes\cdots
\otimes \C^{2}(a_n) \longrightarrow \C^{2}(a_{\sigma(1)}) \otimes\cdots
\otimes \C^{2}(a_{\sigma(n)})
$$
given by
$$
{\mathcal{R}}_{\sigma}=(\textrm{Stab}^{s_{0}}_{\sigma})^{-1} \circ \textrm{Stab}^{s_{0}}_{\mathrm{id}}.
$$  
If $\sigma=s_{l_1}\cdots s_{l_m}$ is the decomposition into a product of transpositions then, by the definitions of the geometric $R$-matrices ${\mathcal{R}}_{\sigma}={\mathcal{R}}_{l_1,l_1+1}\cdots {\mathcal{R}}_{l_m,l_{m}+1}$ so that the result does not depend on the choice of the decomposition due to QYBE (\ref{qybe}). We conclude that the different choices of the chambers correspond simply to different orderings of tensor factors in the tensor product 
of evaluation modules. 
 The explicit isomorphism between different tensor products of evaluation modules are given by $\fsh$ $R$-matrices, as usual.

Second, let us discuss changing the slope.  Let $\mathcal{O}(m)$ be the operator be an operator acting in $K_{\bT}(\N(n))$ by classical tensor product $\gamma \to \mathcal{O}(m)\otimes \gamma$.  The following Proposition follows directly from the definition of K-theoretic stable envelope.  	
\begin{Proposition}\cite{os}
Let $s=\epsilon [\mathcal{O}(1)]$ with $\epsilon \in (m-1,m)$ for $m\in \Z$, then
$$
\mathrm{Stab}^{s}_{\sigma}(\gamma)=\mathcal{O}(m)^{-1}\circ \mathrm{Stab}^{s_{0}}_{\sigma}(\gamma) \circ \mathcal{O}(m)
$$ 	
where $s_0$ is as in Proposition \ref{procan}. 
\end{Proposition}

We conclude that the actions of $\fsh$ on $K_{\bT}(\N(n))$ corresponding to different choices of slopes are related by an internal automorphisms of this representation given by operators of classical multiplication by the line bundles $\mathcal{O}(m)$. Computing the action of this automorphism on the generators of $\fsh$ one can find that 
conjugation by $\mathcal{O}(m)$ increases the loop index of Drinfeld's generators of $\fsh$ by $m$ units, see (\ref{ocong}) below. This means that these automorphisms act as the lattice elements of the quantum affine Weyl group 
of $\fsh$, the well known automorphisms of this algebra.  

\subsection{Description of $\fsh$ through generators and relations}
Let us recall that the algebra $\fsh$ with zero central charge is an associative algebra with $1$ generated over $\C(\hbar^{1/2})$ by elements
$E_{k},F_{k}, H_{m},K$ ($k\in \Z$, $m \in \Z \setminus \{0\}$) satisfying the following relations:
\be
\begin{array}{l} \label{sl2rel}
	K K^{-1}=K^{-1} K=1\\
	{[}H_{m},H_n{]}={[}H_m,K^{\pm 1}{]}=0\\
	K E_{m} K^{-1}=\hbar E_{m}, K F_{m} K^{-1}=\hbar^{-1 } F_{m}\\
	{[}E_m, F_l{]}=\dfrac{\psi^{+}_{m+l}-\psi^{-}_{m+l}}{\h^{1/2}-\h^{-1/2}} \\
	{[} H_{k}, E_{l} {]}=\dfrac{{[}2k{]}_{\h^{1/2}}}{k} E_{k+l}, \,
	{[} H_{k}, F_{l} {]}=-\dfrac{{[}2k{]}_{\h^{1/2}}}{k} F_{k+l},
\end{array}
\ee
where $[n]_{\h}=\frac{\h^n-\h^{-n}}{\h-\h^{-1}}$ and
$$
\begin{array}{l}
\sum\limits_{m=0}^{\infty} \psi^{+}_m z^{-m}= K \exp\Big( (\h^{1/2}-\h^{-1/2})\sum\limits_{k=1}^{\infty} H_{k} z^{-k} \Big)\\
\sum\limits_{m=0}^{\infty} \psi^{-}_{-m} z^{m}= K^{-1} \exp\Big( -(\h^{1/2}-\h^{-1/2})\sum\limits_{k=1}^{\infty} H_{-k} z^{k} \Big).
\end{array}
$$

\subsection{Construction of the action}
The explicit action of the above generators on $K_{\bT}(\N(n))$ can be computed (see e.g. \cite{DF}) by $RTT$-procedure from the Gauss decomposition of geometric R-matrices.\footnote{\rd 
Such relation between the Drinfeld generators and the geometric R-matrices was also discussed in Sections 7.2.1-7.2.2 of \cite{os}. The action of Drinfeld generators described by formulas (\ref{secform}) below is obtained from equality between the geometric R-matrix and the  universal R-matrix in the given representations.}
An equivalent approach was suggested earlier by Vasserot \cite{vasserot} and Nakajima  \cite{nak}.  Though we do not need the explicit formulas 
for the action of generators of $\fsh$ on the K-theory,
it will be convenient to recall these formulas here for future references.

Let $\Gr(k,n)$ be the Grassmannian of $k$-planes in  $\C^n$. We think of it as the zero section of the cotangent bundle $\Gr_{k,n}\subset \N_{k,n}$.
We set $\G^{k}_{k+1}=\Gr_{k,n}\times \Gr_{k+1,n}$ and  $\fM^{k}_{k+1} = T^{*} \G^{k}_{k+1} = \N_{k,n}\times \N_{k+1,n}$.
Let $\pi_1$, $\pi_2$ be natural projections from $\fM^{k}_{k+1}$ to the first and second factor respectively.
We consider a $\textrm{GL}(n)$-orbit in $\G^{k}_{k+1}$:
$$
\Orb^{k}_{k+1}=\{V_1 \times V_2 \in \G^{k}_{k+1} : \dim V_1=k, \ \ \dim V_2=k+1, \ \  V_1 \subset V_2  \}
$$
Let $\fB^{k}_{k+1}$ be a Lagrangian subvariety of $\fM^{k}_{k+1}$ given by the conormal to $\Orb^{k}_{k+1}$. As in previous sections
we set $\tb_1, \tb_2$ to be the tautological bundles on  $\N_{k,n}$ and $\N_{k+1,n}$ and $\tw$ to be a trivial rank $n$ bundle on these varieties.
We define the set of bundles $e_{e},f_{r} \in K_{\bT}(\fB^{k}_{k+1})$ labeled by $r \in \Z$:
\be \nonumber
\begin{array}{l}
	e_r=(-1)^{\textrm{rk}(\tw)+1} (\tb_2/\tb_1)^{r+\textrm{rk}(\tw)-2 \textrm{rk}(\tb_1)} \otimes \h^{- \textrm{rk} (\tb_1)/2}, \\
	f_r=\dfrac{\det \tb_1^2}{\det \tw} \otimes (\tb_2/\tb_1)^{r}\otimes \h^{(\textrm{rk}(\tw)-2 \textrm{rk}(\tb_1) -1)/2},
\end{array}
\ee
where $\h$ stands for trivial line bundle with the corresponding equivariant structure.
These line bundles define the correspondences $E_r,F_r \in \textrm{End}(K_{\bT}(\N(n)))$ acting by rising and lowering the Grassmannian index $k$:
$$
K_{\bT}(\N_{k+1,n}) \stackrel{E_r}{ \longrightarrow } K_{\bT}(\N_{k,n}), \ \ \
K_{\bT}(\N_{k,n}) \stackrel{F_r}{ \longrightarrow } K_{\bT}(\N_{k+1,n}).
$$
Explicitly, these operators are defined by:
\be \label{efdef}
\begin{array}{l}
	E_{r}(\alpha)= \pi_{1,*}(\pi^{*}_{2}(\alpha) \otimes e_r) \in K_{\bT}(\N_{k-1,n}), \\
	\\ F_{r}(\alpha)= \pi_{2,*}(\pi^{*}_{1}(\beta) \otimes f_r) \in K_{\bT}(\N_{k+1,n})
\end{array}
\ee
for a class $\alpha \in  K_{\bT}(\N_{k,n})$. Let us further consider the following complex of equivariant bundles on $\N_{k,n}$:
$$
C^{\bullet} = \h  \tb \rightarrow \tw \rightarrow \h^{-1} \tb^{\vee}
$$
and define
$$
\psi^{+}(z)=K \Ld_{z^{-1}}( C^{\bullet}), \ \ \psi^{-}(z)=K^{-1} \Ld_{z}( C^{\bullet *}),  \ \ \textrm{with}  \ \ K=\h^{(rk(\tb^{\vee}) -rk(\tb))/2}.
$$
These classes determine a map 
\be \label{psidef}
K_{\bT}(\N_{k,n}) \stackrel{\psi^{\pm}(z)}{ \longrightarrow } K_{\bT}(\N_{k,n}), \ \ \ \alpha \mapsto \psi^{\pm}(z)\otimes \alpha.
\ee
The relations (\ref{sl2rel}) for the constructed generators are checked in Section 11.4 
	of \cite{nak} by an explicit computation. A similar calculation can be found in an earlier work \cite{vasserot}).
Thus we have the following Proposition. 
\begin{Proposition}
	The operators $E_r,F_r, \psi^{\pm}(z)$ defined explicitly by (\ref{efdef}) and (\ref{psidef}) satisfy the relations (\ref{sl2rel}) and thus generate the structure of $\fsh$-module on $K_{\bT}(\N_{k,n})$. 	
\end{Proposition}

Such explicit formulas for the action of $\fsh$ on $K_{\bT}(\N(n))$ can be computed from above definitions as a simple exercise on localization in equivariant $K$-theory. For future references we give these formulas here:
\begin{Proposition}
	In the basis of fixed points $\O_{\fp}$ of $K_{\bT}(\N(n))$  for  $\fp=\{i_1,\cdots,i_k\}\subset \fn=\{1,2,\cdots,n\}$ the action of the above generators is described explicitly by the following formulas:
	\be
	\label{secform}
	\begin{array}{l}
		K ({\mathcal{O}}_{\fp} ) = \h^{(n-2 |\fp|)/2} {\mathcal{O}}_{\fp}, \\
		\\
		H_{m} ({\mathcal{O}}_{\fp} ) = \dfrac{[m]_\mathfrak{\h^{1/2}}}{m} \Big(  \h^{-m/2} \sum\limits_{i \in \fn\setminus \fp} a_i^{-m} - \h^{m/2} \sum\limits_{i \in \fp} a_i^{-m} \Big) {\mathcal{O}}_{\fp},\\
		\\
		E_{r} ({\mathcal{O}}_{\fq} ) = \sum\limits_{s \in \fq}\, \left(a^{-r-1}_s \,
		\dfrac{ \prod\limits_{j \in \fq  \setminus \{s\} }\, (a_j- \h a_s)}{
			\prod\limits_{j \in \fn \setminus \fq}\, (a_j-a_s)}
		
		\right)\, {\mathcal{O}}_{\fq\setminus \{s\}},\\
		\\
		F_r ({\mathcal{O}}_{\fp}) = \sum\limits_{s \in \fn \setminus \fp}\,\left( \h^{(n-2 k -1)/2} a_s^{-r+1} \dfrac{\prod\limits_{j \in \fn\setminus (\fp\cup \{s\})}\,(a_j-a_s \h^{-1} )}{\prod\limits_{j \in \fp}\,(a_j-a_s)} \right) \, {\mathcal{O}}_{\fp \cup \{s\}},
	\end{array}
	\ee
	where $[n]_\h=\frac{\h^{n}-\h^{-n}}{\h-\h^{-1}}$.
\end{Proposition}

\noindent
{\bf Note:}
Let us stress here that (\ref{secform}) describes the action 
of $\fsh$ on $K_{\bT}(\N(n))$ in the basis given by classes of fixed points. Note, in particular, that the Cartan generators $H_{m}$ are diagonal in this basis. 
This basis is related to the standard basis of 
$\C^{2}(a_1)\otimes\cdots\otimes\C^{2}(a_n)$ by the $K$-theoretic stable envelope map (\ref{fisom}). 

The operators of classical multiplication by $K$-theory classes are diagonal in the basis of fixed points. In particular, ${\mathcal{O}}(1)\otimes {\mathcal{O}}_{\fp}=(\prod\limits_{i\in \fp} a_i ){\mathcal{O}}_{\fp}$. Formulas (\ref{secform})  imply that
\be \label{ocong} \ \ \ \ \ \ 
{\mathcal{O}}(1) E_{r} {\mathcal{O}}(1)^{-1}=E_{r+1}, \ \ \   {\mathcal{O}}(1) F_{r} {\mathcal{O}}(1)^{-1}=F_{r-1}, \ \ \ {\mathcal{O}}(1) H_{r} {\mathcal{O}}(1)^{-1}=H_{r}.
\ee
This means that conjugation by  ${\mathcal{O}}(1)$ acts on generators in this representation as a lattice element of the affine Weyl group.

\subsection{Other choices of generators}
In the previous section we constructed the action of  $\fsh$ on $K_{\bT}(\N_{k,n})$ in its Drinfeld realization. In this section we give a different description. Quantum affine algebra $\fsh$ can be constructed via the following Chevalley generators: $e_{\pm\al}$, $e_{\pm(\de-\al)}$, $K^{\pm 1}_{\alpha}={\h}^{h_{\alpha}/2}$, $k^{\pm 1}_{\de-\al}={\h}^{h_{\de-\al}/2}$, satisfying the commutation relations
\begin{eqnarray}
k_{\gamma}k_{\gamma}^{-1}&\!\!=\!\!&
k_{\gamma}^{-1}k_{\gamma}^{}=1~,\qquad\quad\;
[k_{\gamma}^{\pm 1},k_{\gamma'}^{\pm 1}]=0~,\qquad
\nonumber\\
k_{\gamma}^{}e_{\pm\alpha}^{}k_{\gamma}^{-1}&\!\!=\!\!&
{\h}^{\pm(\gamma,\alpha)/2}e_{\pm\alpha}~,
\qquad k_{\gamma}^{}e_{\pm(\delta-\alpha)}k_{\gamma}^{-1}=
{\h}^{\pm(\gamma,\delta-\alpha)/2}e_{\pm(\delta-\alpha)},\nonumber
\end{eqnarray}
\begin{eqnarray}
[e_{\alpha},e_{-\delta+\alpha}]&\!\!=\!\!&0~,\qquad\qquad\qquad\quad
[e_{-\alpha},e_{\delta-\alpha}]=0~,
\\[7pt]
[e_{\alpha},e_{-\alpha}]&\!\!=\!\!&[h_{\alpha}]_{\sqrt{\h}}~,\qquad\qquad\;\;\,
[e_{\delta-\alpha},e_{-\delta+\alpha}]=[h_{\delta-\alpha}]_{\sqrt{\h}}~,\nonumber
\end{eqnarray}
\begin{eqnarray}\nonumber
[e_{\pm\alpha}^{},[e_{\pm\alpha}^{},[e_{\pm\alpha}^{},
e_{\pm(\delta-\alpha)}^{}]_{\sqrt{\h}}]_{\sqrt{\h}}]_{\sqrt{\h}}&\!\!=\!\!&0~,\quad\;\;
\\\nonumber
[[[e_{\pm\alpha},e_{\pm(\delta-\alpha)}]_{\sqrt{\h}},e_{\pm(\delta-\alpha)}]_{\sqrt{\h}},
e_{\pm(\delta-\alpha)}]_{\sqrt{\h}}&\!\!=\!\!&0~,\quad\;\;
\end{eqnarray}
where ($\gamma\!=\!\alpha,\delta-\alpha$),  ($\alpha, \alpha$)=$-$($\delta-\al$, $\al$) and
$[h_\beta]_{\sqrt{\h}}\!:=\!(k_\beta\!-\!k_\beta^{-\!1})/(\sqrt{\h}\!-\!\sqrt{\h}^{-\!1})$.
The brackets $[\cdot,\cdot]$ and $[\cdot,\cdot]_{h}$ are the ${\h}$-commutator:
\begin{eqnarray}
[e_{\beta}^{},e_{\beta'}^{}]_{\sqrt{\h}}&\!\!=\!\!&e_{\beta}^{}e_{\beta'}^{}-
{\h}^{(\beta,\beta')/2}e_{\beta'}^{}e_{\beta}^{}~.
\end{eqnarray}
From now on for simplicity, we assume that ($\alpha, \alpha$)=2. It is important to mention that Serre relations remain intact under the transformation $\h\to \h^{-1}$.

The Hopf algebra structure on $\mathcal{U}_{\h}(\widehat{\mathfrak{sl}}_2)$ is given by the following coproduct $\Delta_{\sqrt{\h}}$ and antipode $S_{\sqrt{\h}}$:
\begin{eqnarray}
\begin{array}{rcccl}
\Delta_{\sqrt{\h}}(k_\gamma^{\pm 1})&\!\!=\!\!&k_\gamma^{\pm 1}\otimes k_\gamma^{\pm 1}~,
\qquad\qquad\qquad S_{\sqrt{\h}}(k_\gamma^{\pm 1})&\!\!=\!\!&k_\gamma^{\mp 1}~,
\\[7pt]
\Delta_{\sqrt{\h}}(e_{\beta}^{})&\!\!=\!\!&e_{\beta}^{}\otimes 1
+ k_{\beta}^{-1}\otimes e_{\beta}^{}~,
\qquad\quad S_{\sqrt{\h}}(e_{\beta}^{})&\!\!=\!\!&-k_{\beta}e_{\beta}^{}~,
\\[7pt]
\Delta_{\sqrt{\h}}(e_{-\beta}^{})\!\!&=\!\!&
e_{-\beta}^{}\otimes k_{\beta}+1 \otimes e_{-\beta}^{}~,
\qquad S_{\sqrt{\h}}(e_{-\beta}^{})&\!\!=\!\!&-e_{-\beta}^{}k_{\beta}^{-1}~,
\end{array}
\end{eqnarray}
where $\beta=\alpha,\,\delta-\alpha$.\\
Notice, that our definition corresponds to the standard one (see e.g. \cite{ktolst}) when $\sqrt{\h}=q$, standard deformation parameter for quantum groups.

The relation between the Chevalley relaization and the Drinfeld one is given via the following formulas:
\be \label{drinf}
&&e_{\al_0}=F_0K^{-1},\quad e_{\al_1}=E_{-1},\quad
e_{-\al_0}= K E_0,\quad
e_{-\al_1}=F_{1} ,\\
&&k_{\al_1}=k_{\al_0}^{-1}=K.\nonumber
\ee

The celebrated universal R-matrix is an element in the tensor product $\bf{b}_{+}\otimes \bf{b}_{-}$, where $\bf{b}_{\pm}$ are upper and lower Borel subalgebras of $\mathcal{U}_{{\h}}(\widehat{\mathfrak{sl}}_2)$, which satisfies the following relations with respect to the coproduct $\Delta_{{\sqrt{\h}}}$ and opposite coproduct $\tilde{\Delta}_{\sqrt{\h}}=\sigma\Delta_{_{\sqrt{\h}}}$ and $\sigma(a \otimes b)=b\otimes a$:
\begin{eqnarray}
&&\tilde{\Delta}_{\sqrt{\h}}(a)\!\!=\!\!R\Delta_{\sqrt{\h}}(a)R^{-1} \qquad\quad\;\;
\forall\,\,a \in U_{\sqrt{\h}}(\widehat{\mathfrak{sl}}_2)~,
\nonumber\\
&&(\Delta_{\sqrt{\h}}\otimes{\rm id})R\!\!=\!\!R^{13}R^{23}~,\qquad
({\rm id}\otimes\Delta_{\sqrt{\h}})R=R^{13}R^{12}.
\end{eqnarray}
The relations above can be understood in the following way:
$R^{12}=\sum a_{i}\otimes b_{i}\otimes{\rm id}$,
$R^{13}=\sum a_{i}\otimes{\rm id}\otimes b_{i}$,
$R^{23}=\sum {\rm id}\otimes a_{i}\otimes b_{i}$ if $R$ has the form
$R=\sum a_{i}\otimes b_{i}$ For more information on the structure of the R-matrix, see Appendix.

\section{Representation theory of $\fsh$, spin chains and Baxter Q-operator}

\subsection{$Q$-operator from oscillator  representations.} 
In this section we describe the oscillator representations of $\bf{b}_-\subset \fsh$ (now known as prefundamental representations \cite{fh}) that serve as building blocks for evaluation modules, defined in this section as well. Namely, all finite-dimensional evaluation representations of $\bf{b}_-$ (see below) can be reproduced within the Grothendieck ring generated by prefumdamental representations, as we will see below. 

Using that, we define transfer matrices and Baxter Q-operators following \cite{blz} (see also \cite{qaf}, \cite{qrw}, \cite{qk}, \cite{qdm}, \cite{qdkk}).

First, we introduce the deformed oscillator algebra $\mathcal{B}$ with the generators $H, \mathcal{E}_{\pm}$:
\begin{equation}
{\h}^{\frac{1}{2}} \mathcal{E}_+\mathcal{E}_- - {\h}^{-\frac{1}{2}}\mathcal{E}_-\mathcal{E}_+=\frac{1}{{\h}^{\frac{1}{2}}-{\h}^{-\frac{1}{2}}},\quad  [H, \mathcal{E}_{\pm}]=\pm 2\mathcal{E}_{\pm}.
\end{equation}
The algebra  $\mathcal{B}$ has the following Fock space representations:
\begin{equation}
\mathbf{F}_{\pm}=
\{{\rm span}\{\mathcal{E}^k_{\mp}|0\rangle_{\pm}\};~ \mathcal{E}_{\pm}|0\rangle_{\pm}=0,~ H|0\rangle_{\pm}=0\}
\end{equation}

The representations of the lower Borel subalgebra ${\bf b}_{-}$ of $\mathcal{U}_{{\h}}(\widehat{\mathfrak{sl}}_2)$, which we are interested in, 
can be described by the following homomorphisms $\rho_{\pm}(x): \mathbf{b}_-\to\mathcal{B}$:
\begin{equation}
 \rho_{\pm}(x): \left\{
    \begin{array}{rl}
      &h_{\alpha}\to \pm H,\\
      &h_{\delta-\alpha}  \to \mp H,\\
      &e_{-\al}  \to \mathcal{E}_{\mp},\\
      &e_{\al-\de} \to x\mathcal{E}_{\pm},
    \end{array} \right.
\end{equation}
The above homomorphisms thus automatically produce representations of $\mathbf{b}_-$ on $\mathbf{F}_{\pm}$. In the following we will refer to the corresponding representations as $\rho_{\pm}(x)$ as well.

We are interested in the decomposition of the tensor product
\begin{equation}
\rho_{-}(x{\h}^{-\frac{n+1}{2}})\otimes \rho_{+}(x{\h}^{\frac{n+1}{2}})\label{prodqosc}
\end{equation}
for $n\in \mathbb{Z}$.
In order to describe the components of the decomposition of this tensor product in the Grothendieck ring we need to introduce the evaluation representations $\pi_n^+(x)$ of ${\bf b}_{-}$, associated with 
Verma modules of $\mathcal{U}_{{\h}}(\mathfrak{sl}_2)$, constructed as follows.  
First, let us consider the {\it evaluation} homomorphisms $\pi(x): {\bf b}_{-}\to \mathcal{U}_{{\h}}(\mathfrak{sl}_2)$: 
\begin{eqnarray}
 \pi(x): \left\{
    \begin{array}{rl}
      &h_{\alpha} \to \pm \mathcal{H},\\
      &h_{\delta-\alpha}\to \mp \mathcal{H},\\
      &e_{-\al}  \to \mathcal{F},\\
      &e_{\al-\de} \to x\mathcal{E},
    \end{array} \right.
\end{eqnarray} 
where the generators $\mathcal{E}, \mathcal{F}, \mathcal{H}$ of $\mathcal{U}_{{\h}}(\mathfrak{sl}_2)$ satisfy standard commutation relations:
$
[\mathcal{E}, \mathcal{F}]=\frac{q^\mathcal{H}-q^\mathcal{-H}}{q-q^{-1}}.
$
Thus, given the representation $V$ of $\mathcal{U}_{{\h}}(\mathfrak{sl}_2)$, one can produce representation $\pi_V$ for  ${\bf b}_{-}$. The representations of special interest are Verma modules $V_n^+=\{{\rm span}(v_k=\mathcal{F}^kv_0);~\mathcal{E}v_0=0, \mathcal{H}v_0=nv_0\}$,  
where the action of generators is given by the following formulas:
\begin{eqnarray}
\label{pin}
&&\mathcal{H}v_k=(n-2k)v_k,\nonumber\\
&&\mathcal{F}v_k=v_{k+1}, \\
&&\mathcal{E}v_k=[k]_{\sqrt{\h}}[n-k+1]_{\sqrt{\h}}v_{k-1}.\nonumber
\end{eqnarray}
We also are be interested in the finite-dimensional modules $V_n\equiv V_n^+/V^+_{n-2}$. We will refer to the corresponding evaluation modules of $\mathbf{b}_-$ as $\pi^{+}_{n}(x)$ and $\pi_{n}(x)$ correspondingly. We note here, that evaluation homomorphism can be extended to the full algebra $\fsh$ and the modules $\pi^{+}_{n}(x), \pi_{n}(x)$ serve as representations for the full algebra $\fsh$. 

Let us also denote 1-dimensional representations of ${\bf b}_-$ with eigenvalue of $h_{\alpha}$ equal to $s$ as $\omega_{s}$. It is clear that $\omega_s\otimes \omega_{s'}=\omega_{s+s'}$.
Then we have the following Proposition \cite{blz}.
\begin{Proposition}
Decomposition of the product \rf{prodqosc} in the Grothendieck ring of ${\bf b}_{-}$ is given by the following expresion:
\begin{eqnarray}
\rho_{-}(x\h^{-\frac{n+1}{2}})\cdot \rho_{+}(x{\h}^{\frac{n+1}{2}})=\omega_{-n}(1-\omega_{-2})\inv\pi_n^{+}(x),\label{grot}
\end{eqnarray}
where $(1-\omega_{-2})^{-1}$ is understood as the geometric series expansion.
\end{Proposition}
\begin{proof} Here we give a sketch of the proof, for the details we refer to the paper \cite{blz}. Let us write  the coproduct of the generators of $\bf{b}_-$ in the tensor product:
\begin{eqnarray}
&&\bar{\mathcal{H}}\equiv\Delta_{\rho_-(xy^{-1})\otimes  \rho_-(xy)}(h_{\alpha})=
1\otimes H-H\otimes 1\nonumber\\
&&\bar{\mathcal{F}}\equiv\Delta_{\rho_-(xy^{-1})\otimes  \rho_-(xy)}(e_{-\alpha})=
\mathcal{E}_+\otimes {\h}^{\frac{H}{2}}+1 \otimes \mathcal{E}_{-}=b_++a_+\\
&&\bar{\mathcal{E}}\equiv=x^{-1}\Delta_{\rho_-(xy^{-1})\otimes  \rho_-(xy)}(e_{\alpha-\delta})=
y^{-1}\mathcal{E}_-\otimes {\h}^{-\frac{H}{2}}+1 \otimes y\mathcal{E}_{+}=b_-+a_-,\nonumber
\end{eqnarray}
where we decomposed each of the coproducts into two terms denoted as $b_{\pm}, a_{\pm}$ in the order of their appearance and $y$ stands for $\h^{\frac{n+1}{2}}$. Then the following commutation relations are satisfied:
\begin{eqnarray}
&&a_{\sigma_1}b_{\sigma_2}=\h^{\sigma_1\sigma_2}b_{\sigma_2}a_{\sigma_1},\\
&&{\h}^{\frac{1}{2}} a_-a_+ - {\h}^{-\frac{1}{2}} a_+a_-=\frac{y^2}{{\h}^{\frac{1}{2}}-{\h}^{-\frac{1}{2}}}, \quad
{\h}^{\frac{1}{2}} b_+b_- - {\h}^{-\frac{1}{2}} b_-b_+=\frac{y^{-2}}{{\h}^{\frac{1}{2}}-{\h}^{-\frac{1}{2}}}.
\end{eqnarray}
Then, introducing the basis vectors $|\rho^{(m)}_k\rangle=(a_++b_+)^k
(a_+-\gamma b_+)^m|0\rangle_-\otimes |0\rangle_+$. One can show that for generic
$\gamma$ they span the total space $\rho_{-}(x\h^{-\frac{n+1}{2}})\otimes \rho_{+}(x\h^{\frac{n+1}{2}})$. These vectors are of special nature, namely:
\begin{eqnarray}
&&\bar{\mathcal{H}}|\rho^{(m)}_k\rangle = -2(k+m)|\rho^{(m)}_k\rangle,\nonumber\\
&&\bar{\mathcal{F}}|\rho^{(m)}_k\rangle=|\rho^{(m)}_{k+1}\rangle, \\
&&\bar{\mathcal{E}}|\rho^{(m)}_k\rangle=[k]_{\sqrt{\h}}[n-k+1]_{\sqrt{\h}}|\rho^{(m)}_{k-1}\rangle +
c_k^{m}|\rho^{(m-1)}_{k}\rangle .\nonumber
\end{eqnarray}
We see that up to a shift in $\mathcal{H}$ and extra coefficients $c^m_k$ this gives the decomposition in terms of representations. On the level of the Grothendieck group the
latter problem is irrelevant and the former problem can be corrected by multiplication on 1-dimensional representations of appropriate weight, forming the geometric series.\end{proof}

Now we describe the weighted traces of R-matrices in those representations. First of all, from now on, we assume that whenever we write the trace of the universal R-matrix, which belongs to $\bf{b}_-\otimes \bf{b}_+$ over certain representation of its $\bf{b}_-$ part, its ${\bf b_+}$-part is considered in some finite-dimensional representation $V$, so that the following trace
\begin{eqnarray}
\tilde{Q}_{\pm}(x)=tr_{\rho_{\pm}(x)}\Big[ (I\otimes\rho_{\pm}(x))R(I\otimes Z^{\h_{\alpha}})\Big]
\end{eqnarray}
is well-defined as an element of ${\rm End}(V)[[x, Z^{-1}]]$,  and we will treat it as such in the following (see \cite{fh}).  
However, as we will see in the end of this section, one can also define it analytically for $Re\{\ln(Z)\}>0$ and then analytically continue.

We would like to normalize this operator, so that at $x=0$ it is equal 1, therefore, we have to introduce
\begin{eqnarray}
&&\mathcal{Z}_{\pm}(h_{\alpha})=tr_{\rho_{\pm}(x)}\Big[ (I\otimes\rho_{\pm}(x))\mathcal{K}(I\otimes Z^{\h_{\alpha}})\Big]=\nonumber\\
&&tr_{\rho_{\pm}(x)}\Big[ (I\otimes\rho_{\pm}(x)){\h}^{\frac{h_{\alpha}\otimes h_{\alpha}}{4}}Z^{I\otimes \h_{\alpha}}\Big]=\sum^{\infty}_{s=0}{\h}^{-\frac{h_{\alpha}s}{2}}Z^{-2s}=\frac{1}{1-\h^{-\frac{h_{\alpha}}{2}}Z^{-2}},
\end{eqnarray}
again, understood as an element of ${\rm End}(V)[[ Z^{-1}]]$ so that $Q_{\pm}(x)\equiv\mathcal{Z}_{\pm}^{-1}(h_{\alpha})\tilde{Q}_{\pm}(x)$.

Then we have the following theorem, which is a consequence of the Proposition we proved above. This is the direct analogue of the similar Theorem from \cite{blz} in the case of finite-dimensional representation $V$ of $\mathbf{b}_+$.

\begin{Theorem} The product of two $Q$-operators, acting in certain finite-dimensional representation $V$, gives a trace of the R-matrix in $\pi_n^+(x)$ representation:
\begin{equation}
\h^{\frac{n{h_{\alpha}}}{4}}ZQ_{+}(\h^{\frac{n+1}{2}}x)Q_{-}(\h^{-\frac{n+1}{2}}x)=
tr_{\pi^{+}_n(x)}\Big[ R(I\otimes {Z}^{h_{\alpha}})\Big]\Big (1-\hbar^{-\frac{h_{\alpha}}{2}}Z^{-2}\Big ),
\end{equation}
where the trace in the RHS is understood as an element in ${\rm End}(V)Z^{n}[[Z^{-1},x]]$.
\end{Theorem}
\begin{proof} To prove that, it is enough to see that
\begin{equation}
tr_{\omega_s}\Big[ R(I\otimes {Z}^{h_{\alpha}})\Big]={\h}^{\frac{sh_{\alpha}}{4}}Z^s.
\end{equation}
Then formula \rf{grot} corrected by normalization coefficients gives the result of the Theorem.\end{proof}

Let us introduce the following notation for the following traces, again, understood as power series from ${\rm End}(V)Z^{n}[[Z^{-1},x]]$ for certain finite-dimensional representation $V$ of $\bf{b}_+$:
\begin{eqnarray}
&&T_n^{+}(x)\equiv tr_{\pi^{+}_n(x)}\Big[ R(I\otimes {Z}^{h_{\alpha}})\Big],
\nonumber\\
&&T_n(x)\equiv tr_{\pi_n(x)}\Big[ R(I\otimes {Z}^{h_{\alpha}})\Big]=T_n^{+}(x)-T_{-n-2}^{+}(x).
\end{eqnarray}
There are two direct consequences of the Theorem we proved. One is known as the {\it quantum Wronskian relation} between $Q_{\pm}(x)$-operators, illustrating that they are not independent. The second, known as the Baxter TQ-relation, expresses the dependence of $T_1$ on  $Q_{\pm}$- operators (see e.g. \cite{blz}). 
\begin{Proposition}\label{wr}
Given $Q_{\pm}$, $T_1$-operators acting in certain finite-dimensional representation $V$, the following relations forhold:
\begin{eqnarray}
&&Z{\h}^{\frac{h_{\alpha}}{4}}Q_{+}(\h^{\frac{1}{2}} x)Q_{-}(\h^{-\frac{1}{2}}x)-Z^{-1}{\h}^{-h_{\alpha}/4}Q_{+}(\h^{-\frac{1}{2}}x)Q_{-}(\h^{\frac{1}{2}} x)=Z{\h}^{\frac{h_{\alpha}}{4}}-Z^{-1}{\h}^{-\frac{h_{\alpha}}{4}}\\
&&\nonumber\\
&&T_1(x)Q_{\pm}(x)={\h}^{\pm\frac{h_{\alpha}}{4}}Z^{\pm 1}Q_{\pm}({\h}x)+{\h}^{\mp\frac{h_{\alpha}}{4}}Z^{\mp 1}Q_{\pm}(\h^{-1}x)
\end{eqnarray}

\end{Proposition}
\begin{proof} To prove the first relation it is enough to use  the statement of the Theorem in the case of $n=0$ and use the fact that $T_0(x)=1$. The second follows from the Theorem above, when $n=1$,  and the application of the Wronskian relation.
\end{proof}

\subsection{XXZ spin chains and algebraic Bethe ansatz}
As we already discussed, one can consider the evaluation representations $\pi_k(x)$ for the upper Borel subalgebra $\bf{b}_{+}$ as well, namely,
\begin{eqnarray}
\pi_n(x): e_{\alpha}\to \mathcal{E}, ~ e_{\delta-\alpha}\to x\mathcal{F}, ~ h_{\alpha}\to \mathcal{H},\quad  h_{\delta-\alpha}\to -\mathcal{H},
\end{eqnarray}
where $\mathcal{E},\mathcal{F},\mathcal{H}$ satify the standard commutation relations of $\mathcal{U}_{\sqrt{\h}}(sl(2))$ and their action on the corresponding $n+1$-dimensional module is given by the formulas \rf{pin}.
The normalized universal R-matrix with ${\bf b}_+$ being represented via evaluation homomorphism and ${\bf b}_-$ considered in 2-dimensional representation $\pi_1$, is given by the following expression, see e.g. \cite{kt}:
\begin{eqnarray}
(\pi_n(x)\otimes\pi_1(1))R=\phi_n(x)
\left( {\begin{array}{cc}
{\h}^{\frac{\ch}{4}}-{\h}^{-\frac{1}{2}} x{\h}^{-\frac{\ch}{4}} & (\h^{\frac{1}{2}}-{\h}^{-\frac{1}{2}})\cf{\h}^{-\frac{\ch}{4}}\\
x(\h^{\frac{1}{2}}-\h^{-\frac{1}{2}})\ce{\h}^{\frac{\ch}{4}}&{\h}^{-\frac{\ch}{4}}-{\h}^{-\frac{1}{2}} x{\h}^{\frac{\ch}{4}}
\end{array}}
\right),
\end{eqnarray}
where $\phi_n(x)=\exp(\sum^{\infty}_{l=1}\frac{{\h}^{\frac{n+1}{2}}+{\h}^{-\frac{n+1}{2}}}{1+{\h}^{l}}\frac{x^l}{l}$).
For convenience it is sometimes useful to consider the variable $u$, so that
$u^{-2}=x$. Then the normalized and Cartan-conjugated R-matrix transforms into the following symmetric expression:
\begin{eqnarray}\label{lop}
&&\cl(u)=u{\h}^{\frac{1}{4}}\phi_n(u^{-2})^{-1}u^H\Big[(\pi_1(u^{-2})\otimes\pi_1(1))R\Big]u^{-H}=\nonumber\\
&&\left( {\begin{array}{cc}
uk{\h}^{\frac{1}{4}}-\h^{-\frac{1}{4}}u^{-1}k^{-1}& {\h}^{-1/2}u^{-1}f\\
\h^{-1/2}u^{-1}e&\h^{\frac{1}{4}} uk^{-1}-\h^{-\frac{1}{4}} u\inv k
\end{array}}
\right).
\end{eqnarray}
Here $e=\ce{\h}^{\frac{\ch}{4}}(\h^{\frac{1}{2}}-{\h}^{-\frac{1}{2}})$, $f={\h}^{-\frac{\ch}{4}}\mathcal{F}(\h^{\frac{1}{2}}-{\h}^{-\frac{1}{2}})$, $k={\h}^{\frac{\ch}{4}}$, which satisfy simple commutation relations:
\begin{eqnarray}
ke={\h}^{\frac{1}{2}} ek, \quad kf={\h}^{-\frac{1}{2}}fk,\quad  ef-fe=(\h^{\frac{1}{2}}-{\h}^{-\frac{1}{2}})(k^2-k^{-2})
\end{eqnarray}
The operator $\cl(u)$ is known as the {\it $\mathcal{L}$-operator}  for the XXZ spin chain and the reason is as follows.

Consider the tensor product $\pi_1(\xi_1^2/u^2)\otimes \dots\otimes \pi_1(\xi_n^2/u^2)\otimes \pi_1(1)$, where each of the 2-dimensional evaluation modules is refered as $site$ of the lattice and the last site is considered as auxilliary.
Note that each of $\pi_1(\xi_1^2/u^2)\cong \mathbb{C}^2={\rm span}_{\mathbb{C}}(\nu_0, \nu_1)$, where $\nu_0$ and $\nu_1$ are correspondingly highest and lowest weight vectors with respect to representations of $e,f,k$-algebra.
Then  the {\it monodromy matrix} acting in this space is the following product:
\begin{equation}\label{mon}
\mathcal{T}(u)\equiv\mathcal{L}_1(u/\xi_1)\dots\mathcal{L}_N(u/\xi_n)\left( {\begin{array}{cc}
Z& 0\\
0&Z^{-1}
\end{array}}
\right)
\end{equation}
where the i-th $\cl$-operator acts in the tensor product of the i-th cite and the auxilliary module.  The traces over auxiliary modules commute
\begin{eqnarray}\label{int}
[tr_{\pi_1(1)}\mathcal{T}(u_1), tr_{\pi_1(1)}\mathcal{T}(u_2)]=0
\end{eqnarray}
because of the relation $(\Delta\otimes I)R=R^{13}R^{23}$. Namely, up to multiplication by the function from \rf{lop} $tr_{\pi_1(1)}\mathcal{T}(z)$  coincides with $T_1(x)$ from the previous section, where ${\bf b}_{+}$ is represented in the physical space $\mathbf{H}\equiv\pi_1(\xi_1^2)\otimes \dots\otimes \pi_1(\xi_n^2)$. Then the Yang-Baxter equation $R^{12}R^{13}R^{23}=R^{23}R^{13}R^{12}$ implies the commutativity condition on the traces.

The relation \rf{int} is known as $integrability$ $condition$, and the operator-valued coefficients of the expansion of $\ln (tr\mathcal{T}(u))$ are known as the Hamiltonians of the XXZ model.
The so-called Algebraic Bethe Ansatz provides an effective method for finding eigenvalues and eigenvectors of $tr\mathcal{T}(u)$ and therefore, solve the problem of simultaneous diagonalization of Hamiltonians.

Below we illustrate the key steps of the method (for the details we refer the reader to section 4.2 of \cite{resh}).

It is convenient to write down the expression for $\mathcal{T}(u)$ as a matrix in 
the representation $\pi_1(1)$ with anticommutative coefficients as follows:
\begin{eqnarray}
\mathcal{T}(u)=\left( {\begin{array}{cc}
A(u)& B(u)\\
C(u)&D(u)
\end{array}}
\right),
\end{eqnarray}
so that $tr_{\pi_1(1)}(\mathcal{T}(u))=A(u)+D(u)\in End(\mathbf{H})[u,u^{-1}, Z, Z^{-1}]$.
Then the Yang-Baxter equation produces commutation relations between $A, B$ and $D, B$ for different values of parameter $z$. Let us denote the highest weight vector in the product of our $N$ cites as
\begin{eqnarray}
\Omega_{+}\equiv \nu_0\otimes...\otimes \nu_0
\end{eqnarray}
Then the following Theorem is satisfied, see e.g. \cite{fad}, \cite{resh}.
\begin{Theorem}\label{theig}
Vectors $\{B(v_1)\dots B(v_k)\Omega_{+}\}$ are the eigenvectors of $tr_{\pi_1(1)}(\mathcal{T}(u))$ with eigenvalues
\begin{eqnarray}
\label{transeig}
&&\Lambda(u|\{v_i\},Z)=\nonumber\\
&&\alpha(u)\prod_{i=1}^k\frac{v_i u\inv{\h}^{\frac{1}{2}}-v_i\inv u {\h}^{-\frac{1}{2}}}{v_i u\inv-v_i\inv u }+\delta(u)\prod_{i=1}^k\frac{v_i\inv u{\h}^{\frac{1}{2}}-v_iu\inv {\h}^{-\frac{1}{2}}}{v_i \inv u-v_i u \inv}
\end{eqnarray}
so that
\begin{eqnarray}
\alpha(u)=Z\prod_{i=1}^n(u \xi_i^{-1}{\h}^{\frac{1}{2}}-u\inv \xi_i{\h}^{-\frac{1}{2}}), \quad \delta(u)=Z^{-1}\prod_{i=1}^n(u \xi_i^{-1}-u\inv \xi_i)
\end{eqnarray}
are the eigenvalues of $A(u)$ and $D(u)$ on $\Omega_+$ and parameters $v_i$ are subject to Bethe equations:
\begin{eqnarray}\label{beq}
\prod_{\alpha=1}^n\frac{v_i{\xi_\alpha}^{-1}{\h}^{\frac{1}{2}}-
v_i^{-1}\xi_\alpha {\h}^{-\frac{1}{2}}
}{v_i\xi_\alpha^{-1}-v_i^{-1}\xi_\alpha} =
Z^{-2}\prod_{j=1, i\neq j}^k \frac{v_iv_j^{-1}{\h}^{\frac{1}{2}}-v_i^{-1}v_j{\h}^{-\frac{1}{2}}} {v_iv_j^{-1}{\h}^{-\frac{1}{2}}-v_i^{-1}v_j{\h}^{\frac{1}{2}}}
\end{eqnarray}
\noindent ii) Vectors $\{B(v_1)\dots B(v_k)\Omega_{+}\}$ indexed by the solutions of Bethe equations, span the weight subspace $\mathbf{H}_k\subset\mathbf{H}$ corresponding to the eigenvectors of $\mathcal{H}$ with eigenvalue $n-2k$. 
\end{Theorem}

For our considerations we will redenote $\xi^2_{\alpha}\equiv a_{\alpha}$, $v_i^{2}=s_i$, and we remind that $u^{-2}=x$ so that we can rewrite

\begin{eqnarray}
&&tr\mathcal{T}(u)=\nonumber\\
&&\h^{n\over 4}
\prod_{i=1}^{n}\frac{u^n}{a_i^{n/2}}
\Bigg[
Z\h^{h_{\alpha}\over 4}\prod_{i=1}^ng_i(x{\h}^{-1})\frac{Q({\h}x)}{Q(x)}+  Z\inv {\h}^{-h_{\alpha}\over 4}\prod_{i=1}^ng_i(x)\frac{Q({\h}^{-1}x)}{Q(x)}\Bigg],
\end{eqnarray}
where
\begin{eqnarray}
g_i(x)=(1-a_ix)
\end{eqnarray}
and the eigenvalues of the polynomial operator $Q(x)$ on  $P(s_1,\dots, s_k)\equiv B(v_1)\dots B(v_k)\Omega$ are given by
\begin{eqnarray}
\prod_{i=1}^k(1-x s_i)
\end{eqnarray}
The Bethe ansatz equations in the new variables are as follows:
\begin{eqnarray}
\prod_{\alpha=1}^n\frac{{\h}^{\frac{1}{2}}-a_{\alpha}s_i^{-1}{\h}^{-\frac{1}{2}}
}{1-a_{\alpha}s_i^{-1}} =
Z^{-2}\prod_{j=1,~ i\neq j}^k \frac{s_j^{-1}{\h}^{\frac{1}{2}}-s_i^{-1}{\h}^{-\frac{1}{2}}} {s_j^{-1}{\h}^{-\frac{1}{2}}-s_i^{-1}{\h}^{\frac{1}{2}}}
\end{eqnarray}
Now we want to relate operator $Q(x)$ defined {\it ad hoc} on the eigenvectors with the quantum group-based operators $Q_{\pm}(x)$ from previous section. 

It is intuitively clear from the TQ-relation we derived, that they should differ by multiplication on a certain power series in $x$, so that $Q(x)=F(x)Q_+(x)$.  
We refer the interested reader to read complete details on identification to \cite{fh}, namely 
Theorem 5.9, Corollary 5.10 and finally Theorem 5.11. 
 Here we just assume this relation and simply derive the formula for the corresponding power series $F(x)$. 
If $T_1^f(x)=t(x)\mathcal{T}(u)$, then we have the following system of functional equations:
\begin{eqnarray}
&&t(x)G({\h}^{-1}x)F({\h}x)=F(x), \quad t(x)G(x)F({\h}^{-1}x)=F(x), \\
&&G(x)=\prod_{i=1}^{N}g_i(x).
\end{eqnarray}
While we already know the solution for $t(x)$ from the explicit normalization of the $\mathcal{L}$-operator, the solution for $F$ can be derived from the two equations above:
\begin{eqnarray}
F({\h}^{-2}x)=\frac{G({\h}^{-2}x)}{G({\h}^{-1}x)}F(x)
\end{eqnarray}
Representing $F(x)=\prod_{i=1}^nf_i(x)$, this equation reduces to $f_i({\h}^{-2}x)=\frac{g_i({\h}^{-1}x)}{g_i(\h^{-2}x)}f_i(x)$ and therefore
\begin{eqnarray}
f_i({\h}^{-2}x)=\frac{1-a_ix\h^{-2}}{1-a_ix{\h}^{-1}}f_i(x)
\end{eqnarray}
Hence,
\begin{eqnarray}
f_i(x)=\frac{\prod_{k=1}^{\infty}(1-a_ix{\h}^{2k-2})}{\prod_{k=1}^{\infty}(1-a_ix\h^{2k-1})}=
\frac{(a_ix{\h}^{-2},{\h} ^2)_{\infty}}{ (a_ix{\h}^{-1},\h^2)_{\infty}}
\end{eqnarray}
Note, that this ratio of Pochhammer coefficients because of q-binomial formula can be re-expanded as power series in $x$. In the following when we write it, we interpret it as such an  expansion.

\begin{Proposition}\label{norm}
The eigenvalues of the operator
\begin{eqnarray}
\mathcal{Q}_+(x)=\prod^n_{i=1}\frac{(a_ix{\h}^{-2},{\h} ^2)_{\infty}}{ (a_ix{\h}^{-1},{\h}^2)_{\infty}}Q_+(x),
\end{eqnarray}
where the coefficient in the RHS is understood as power series in $x$-variable, are polynomials with respect to the $x$ variable. Explicitly the eigenvalues are equal to 
$\prod_{i=1}^s\Big(1-{x}{s_i}\Big)$ and they are attained on the vectors
$\{P(s_1,\dots s_k)\}$, provided the Bethe ansatz equations
\begin{eqnarray}
\prod_{p=1}^n\frac{\h^{-1}a_{p}-s_i
}{a_{p}-s_i} =
Z^{-2}\h^{-\frac{n}{2}}\prod_{j=1,~ i\neq j}^k \frac{s_i-s_j{\h}^{-1}} {s_i\h^{-1}-s_j}
\end{eqnarray}
are satisfied.
\end{Proposition}

One can reproduce
$Q_-(x)$ from algebraic Bethe ansatz as well. In order to do that, one has to use lowest weight vector
$\Omega_-=\nu_1\otimes\dots\otimes \nu_n$. It is annihilated by $B(u)$-operators and the space of states is spanned by the operators $C(v_1)\dots C(v_k)$ acting on $\Omega_-$.
The following Theorem is an analogue of Theorems \rf{theig} and \rf{norm}.
\begin{Theorem}
Vectors $\{P_{-}(s_1,\dots s_k)=C(v_1)\dots C(v_k)\Omega_-\}$, where $s_i=v_i^{2}$ are the eigenvectors for $tr(\mathcal{T}(u))$ with eigenvectors $\Lambda(u|\{v_i\},Z^{-1})$ (see \rf{transeig}) , so that $s_i$ satisfy Bethe equations
\begin{eqnarray}
\prod_{p=1}^n\frac{\h^{-1}a_{p}-s_i
}{a_{p}-s_i} =
Z^{2}\h^{-\frac{n}{2}}\prod_{j=1,~ i\neq j}^k \frac{s_i-s_j{\h}^{-1}} {s_i\h^{-1}-s_j}
\end{eqnarray}
The operator
\begin{eqnarray}
\mathcal{Q}_-(x)=\prod^n_{i=1}\frac{(a_ix{\h}^{-2},{\h} ^2)_{\infty}}{ (a_ix{\h}^{-1},{\h}^2)_{\infty}}Q_-(x)
\end{eqnarray}
has eigenvalues
$\prod_{i=1}^n(1-x s_i)$ on the vectors
$P_-(s_1,\dots s_k)$.

\noindent ii) Vectors $P_-(s_1,\dots s_k)$ span the weight subspace $\tilde{\mathbf{H}}_k\subset\mathbf{H}$ corresponding to the eigenvectors of ${\mathcal{H}}$ with eigenvalue $2k-n$. 
\end{Theorem}

The following Proposition gives a normalized version of the quantum Wronskian relation (see Proposition \ref{wr}).

\begin{Theorem}
The Wronskian relation between $\mathcal{Q}_{\pm}$-operators reads as follows:
\begin{eqnarray}
&&Z\h^{\frac{h_{\alpha}}{4}}\mathcal{Q}_{+}({\h}^{\frac{1}{2}} x)
\mathcal{Q}_{-}({\h}^{-\frac{1}{2}}x)-
Z^{-1}{\h}^{-\frac{h_{\alpha}}{4}}
\mathcal{Q}_{+}({\h}^{-\frac{1}{2}}x)
\mathcal{Q}_{-}({\h}^{\frac{1}{2}} x)=\nonumber\\
&&(Z{\h}^{\frac{h_{\alpha}}{4}}-{\h}^{-\frac{h_{\alpha}}{4}}Z^{-1})\prod^n_{i=1}(1-a_i{\h}^{-\frac{1}{2}}x).
\end{eqnarray}
\end{Theorem}

\subsection{Explicit expression for the $Q$-operator via simple root generators}
In order to represent  operator $\mathcal{Q}_{+}(x)$ via Chevalley generators of ${\bf b}_{+}$ on $\pi_1(a_1)\otimes\dots\otimes \pi_1(a_n)$, we have to understand how to compute traces of weighted products of $\mathcal{E}_{\pm}$ in $\rho_{+}$ representation.

Our first ingredient is to compute $tr(e^{\alpha H}\ce^k_{+}\ce^m_-)$ for any $k, m$ as power series in  $e^{-\alpha}$.

\begin{Lemma}
i) Trace $tr(e^{\alpha H}\ce^k_{+}\ce^m_-)$ is zero if $k\neq m$.\\
ii) Assuming $_+\langle 0|$ is dual to the a vacuum vector $c_k=~ _+\langle 0|  \ce^k_{+}\ce^k_-|0\rangle_+=\frac{{\h}^{-k(k+1)/4}[k]_{\sqrt{\h}} !}{({\h}^{\frac{1}{2}}-{\h}^{-\frac{1}{2}})^k}$
\end{Lemma}
\begin{proof} Part i) follows from the cyclic property of the trace and the commutation relations of generators. To prove part ii), one should use that
\begin{eqnarray}
&&\ce_+ \ce_-^k=c
{\h}^{-\frac{1}{2}}\ce_-^{k-1}+
{\h}^{-1}\ce_{-}\ce_{+}\ce_-^{k-1}=
c{\h}^{-\frac{1}{2}}\frac{1-{\h}^{-k}}{1-{\h}^{-1}}\ce_-^{k-1}+
{\h}^{-k}\ce_-^{k}\ce_+=\nonumber\\
&&c{\h}^{-k/2}[k]_{\sqrt{\h}}\ce_-^{k-1}+{\h}^{-k}\ce_-^{k}\ce_+,
\end{eqnarray}
where $c=({\h}^{\frac{1}{2}}-{\h}^{-\frac{1}{2}})\inv$. Therefore $c_k=c {\h}^{-k/2}[k]_{\sqrt{\h}}c_{k-1}$ and this implies $ii)$.
\end{proof}

Finally, one can write the trace of the desired expression:
\begin{eqnarray}
tr(e^{\alpha H}\ce^m_{+}\ce^m_-)=\sum^{\infty}_{k=0}~_+\langle 0|\ce^k_{+}  e^{\alpha H}\ce^m_{+}\ce^m_-\ce^k_-|0\rangle_+
\frac{{\h}^{\frac{k(k+1)}{4}}}{c^k[k]_{\sqrt{\h}}!}.
\end{eqnarray}
Let us express every summand in the expression above as $\sigma_k$. Then
\begin{eqnarray}
&&\sigma_k=\frac{e^{-2\alpha k}{\h}^{\frac{k(k+1)}{4}}}{c^k[k]_{\sqrt{\h}}!}~
_+\langle 0|  \ce^{k+m}_{+}\ce^{k+m}_-|0\rangle_+=\nonumber\\
&&\frac{e^{-2\alpha k}{\h}^{\frac{k(k+1)}{4}}}{c^k[k]_{\sqrt{\h}}!}c^{k+m}[k+m]_{\sqrt{\h}} !{\h}^{-(k+m)(k+m+1)/4}=
\nonumber\\
&&
{\h}^{-\frac{m(m+1)}{4}}\h^{\frac{(m+k)(m+k-1)}{4}}{\h}^{-\frac{k(k-1)}{4}}c^m\frac{(k+m)_{{\h}^{-1}}!}
{(k)_{{\h}^{-1}}!}x^k
={\h}^{-m/2}c^m\frac{(k+m)_{{\h}^{-1}}!}
{(k)_{{\h}^{-1}}!}x^k\nonumber
\end{eqnarray}
where we used the fact that $[s]_{\sqrt{\h}}!={\h}^{\frac{s(s-1)}{4}}(s)_{{\h}^{-1}}$ and $x=e^{-2\alpha}$.
In order to sum over $k$, let us use the quantum binomial formula:
\begin{eqnarray}
\sum_{k\ge 0}\frac{(k+m)_{{\h}^{-1}}!}{(k)_{{\h}^{-1}}!(m)_{{\h}^{-1}}!} x^k=\prod_{k=0}^m\frac{1}{1-{\h}^{-k}x}.
\end{eqnarray}
As a result we obtain the following Proposition.
\begin{Proposition}
The weighted trace of the product of oscillator operators is given by the following formula:
\begin{eqnarray}
tr(e^{\alpha H}\ce^m_{+}\ce^m_-)=\frac{{\h}^{-m/2}(m)_{{\h}^{-1}}!}{({\h}^{\frac{1}{2}}-{\h}^{-\frac{1}{2}})^m}
\prod_{k=0}^m\frac{1}{1-{\h}^{-k}e^{-2\alpha}}, 
\end{eqnarray}
where the RHS is understood as an element of $\mathbb{C}[[e^{-\alpha}]]$.
\end{Proposition}
Now we are ready to write the formula for the $\mathcal{Q}_+$-operator via the trace formula we introduced. Notice, that in the expression for $Q^f_+(x)$ we are taking the trace of the following expression:
\begin{eqnarray}
\exp_{{\h}^{-1}}(x({\h}^{\frac{1}{2}}-{\h}^{-\frac{1}{2}})(e_{\delta-\alpha}\otimes \ce_{+}))R_0(x)\exp_{{\h}^{-1}}(({\h}^{\frac{1}{2}}-{\h}^{-\frac{1}{2}})(e_{\alpha}\otimes \ce_{-}))\mathcal{K}(1\otimes {\h}^{fH/2})\nonumber
\end{eqnarray}
Notice that all other q-exponential terms vanish, since $\tilde{e}_{-\delta}=-\frac{\sqrt{\h}x}{{\h}^{\frac{1}{2}}-{\h}^{-\frac{1}{2}}}$ is  constant in the oscilaltor representation, and therefore all its commutators vanish.
In order to express $R_0(x)$ appropriately, we have to calculate the
$\tilde{e}_{-k\delta}$ generators:
\begin{eqnarray}
&&\mathcal{E}(u)=({\h}^{\frac{1}{2}}-{\h}^{-\frac{1}{2}})\sum_{k\ge 1}
\tilde e_{-k\delta}u^k=\nonumber\\
&&\log(1+({\h}^{\frac{1}{2}}-{\h}^{-\frac{1}{2}})
\tilde{e}'_{-\delta}u)=\log(1-xu)=-\sum_{k\ge 1}\frac{(\sqrt{\h}xu)^k}{k}
\end{eqnarray}
Therefore $\tilde{e}_{-k\delta}=-\frac{({\h}^{1/2}x)^k}{k({\h}^{\frac{1}{2}}-{\h}^{-\frac{1}{2}})}$. Hence
\begin{eqnarray}
R_0(x)=\exp\Big (-\sum_{k\ge 0}\frac{({\h}^{\frac{1}{2}}x)^k}{[2k]_{\sqrt{\h}}}(\tilde{e}_{k\delta}\otimes 1)\Big)
\end{eqnarray}
Let us combine that with the normalizing function for the $Q$-operator  and denote it $W(x)$.  Notice that $W(x)$ is independent of $z$.
Also, let us denote $U={\h}^{(h_{\alpha}\otimes 1)/4}Z$, so that $\mathcal{K}(I\otimes Z^H)=U^{1\otimes H}$. Recall, that the normalization factor is $\mathcal{Z}_{+}(h_{\alpha})=(1-U^{-2})^{-1}$.
Therefore, one can write the expression for $\mathcal{Q}_+(x)$ as follows:
\begin{eqnarray}
&&\mathcal{Q}_+(x)=1+(1-U^{-2})\sum_{m=1}^{\infty}\tilde{W}^Z_mx^m, \\
&&\tilde{W}^Z_m=\sum^m_{k=0}\frac{(\h^{1/2}-\h^{-1/2})^k{\h}^{-k/2}}{(k)_{{\h}^{-1}}!}\prod_{i=0}^k\frac{1}{(1-{\h}^{-i}U^{-2})}e^k_{\delta-\alpha}W_{m-k}e^k_{\alpha}.\nonumber
\end{eqnarray}
Therefore
we have the following Theorem.
\begin{Theorem}$Z$-dependence of the operator $\mathcal{Q}_+(x)$ acting on the representation $\pi_1(a_1)\otimes\dots\otimes \pi_N(a_N)$ can be expressed as follows:
\begin{eqnarray}
&&\mathcal{Q}_+(x)=1+\sum_{m=1}^{\infty}W^Z_mx^m,\nonumber \\
&&W^Z_m=\sum^m_{k=0}\frac{(\h^{1/2}-\h^{-1/2})^k{\h}^{-k/2}}{(k)_{{\h}^{-1}}!\prod_{i=1}^
{k}(1-{\h}^{-i}{\h}^{-h_{\alpha}/2}Z^{-2})}e^k_{\delta-\alpha}
W_{m-k}e^k_{\alpha},
\end{eqnarray}
where $W(x)=\sum^{\infty}_{m=1}x^m W_m$ is the limit $Z\to 0$, which corresponds to the diagonal operator with eigenvalues $\prod_{i}(1-a_ix)$ in the standard basis
of the representation.
\end{Theorem}

\noindent {\bf Remark.} Here we note, that although there is an infinite sum  in the expression for 
$\mathcal{Q}_+(x)$ above, it is necessarily a polynomial in $x$, since we are considering it in the finite dimensional representation $V$ of $\mathbf{b}_+$.\\

Switching to the Drinfeld basis \rf{drinf}, we obtain the following result.

\begin{Corollary}
In the Drinfeld basis the formula for $W_m^Z$ reads as follows:
\begin{eqnarray}
W^Z_m=\sum^m_{k=0}
\frac{(1-\h^{-1})^k{\h}^{-k(k+1)/2}K^{-k}}{(k)_{{\h}^{-1}}!
\prod_{i=1}^{k}
(1-{\h}^{-i}K^{-1}Z^{-2})}
F_0^k
W_{m-k}E_{-1}^k.
\end{eqnarray}
\end{Corollary}

We note that a similar formula can be written for $\mathcal{Q}_-(x)$.

\subsection{Geometric realization of the $Q$-operator}

In this subsection we will relate quantum tautological bundles to the Baxter $\mathcal{Q}_+$-operator, which will allow us to write combinatorial formula for some of them on $K_{\bT}(\N(n))$ in the  basis of fixed points.
Recall from Section 4 that the eigenvalues of $\hat{\tau}(z)$ are given by $\tau(s_1,\dots, s_n)$ evaluated at solutions of Bethe equations from Section 4, which differ from the ones, obtained in section by a substitution $\h \to \h^{-1}$ and $Z^2\to (-1)^nz$. Moreover, we know the explicit form of the from the operator $\mathbf{M}(z)$ (see Theorem \ref{Mop} for the definition) from the section 7.3.7 of \cite{os}. Also, we know from the Theorem \ref{Mth} that 
$$
	\M(z)|_{q=1}=\widehat{\O(1)}(z) \circledast.
	$$
	Thus, the quantum tautological bundle,  corresponding to the top wedge power of the tautological bundle is expressed via the following formula:
\begin{eqnarray}
\widehat{\mathcal{O}(1)}(z) = B(z)\, {\mathcal{O}}(1), \ \ \ B(z)=\sum\limits_{m=0}^{\infty} \dfrac{\hbar^{m(m+1)/2} (\hbar-1)^m K^m }{[m]_{\h}!\prod\limits_{i=1}^{m}(1-(-1)^nz^{-1} K \hbar^i)} F_{0}^{m} E_{0}^{m},
\end{eqnarray}
which coincide with the expression for $W^Z_m$ on $\mathbf{H}_{m}$ weight subspace given the commutation relation of $\mathcal{O}(1)$ with quantum group generators (see \cite{os}). 
This immediately leads to the following result.
\begin{Theorem}
The operator ${\mathcal{Q}_+}(x)$ upon the transformation $\h\to \h^{-1}$, $K\to K^{-1}$ and identification $z=Z^2$ is  equal to the operator of quantum multiplication by the quantum weighted exterior algebra of tautological bundle $\sum_{i=0}^k(-1)^i\widehat{\Lambda^i\mathcal{V}}x^i$. Moreover, the following combinatorial formula, explicitly expressing the dependence of quantum exterior powers tautological bundles on the deformation parameter, is valid:
\begin{eqnarray}
\widehat{\Lambda^m\mathcal{V}}&=&
\sum^m_{k=0}
\frac{(\h-1)^k{\h}^{k(k+1)/2}K^k}{(k)_{{\h}}!
\prod_{i=1}^{k}
(1-(-1)^n{\h}^{i}Kz^{-1})}
F_0^k
(\Lambda^{m-k}\mathcal{V})E^k_{-1}
\end{eqnarray}
\end{Theorem}
\begin{proof}
First of all, the eigenvalues of the corresponding operators coincide. Therefore, we have to show that these operators are diagonalized in the same basis. 

Note, that the eigenvalues of the operator $W^Z_m$ on $\mathbf{H}_m$ are distinct for generic values of $Z, \{a_i\}, \hbar$: they are products of the corresponding Bethe roots. For example, in the limit $Z\to 0$ they are indeed distinct for generic $\{a_i\}$ (they are just the $m$-products of various $a_i$).  Notice, that all other operators $W^Z_r$ commute with $W^Z_m$ on $\mathbf{H}_m$ as a consequence of integrability. Therefore, for generic $Z, \{a_i\}, \hbar$ they coincide with the corresponding tautological bundles. Since both $W^Z_m$ and quantum tautological bundles are rational functions of $Z, \{a_i\}, \hbar$, this means that it is true in general.
\end{proof}


Therefore, we arrive to the following theorem regarding multipication on the quantum tautological bundles.

\begin{Theorem}
The operators of quantum multiplication by $\hat{\tau}(z)$ are the universal elements in  $\fsh[[z]]$, i.e. they do not depend on $k,n$ defining Nakajima variety. 
\end{Theorem}
\begin{proof}
Since quantum exterior powers generate any quantum class $\tau$ under quantum 
multipication, as we know from Theorem \ref{symmeig}, one obtains formulas for multiplication by $\hat{\tau}(z)$ via composition of the corresponding operators from  $\fsh[[z]]$. Thus the theorem is proven.
\end{proof}

\begin{appendices}

\section{Cartan-Weyl basis and R-matrix for $\fsh$. }

According to Khoroshkin and Tolstoy, see e.g. \cite{ktolst}, one can construct an analogue of Cartan-Weyl basis for the above Hopf algebra, based on different ordering of roots:
\begin{eqnarray}
e_{\delta}^{}\!\!\!&:=\!\!\!& [e_{\alpha}^{},e_{\delta-\alpha}^{}]_{\sqrt{\h}}~,
\qquad \qquad\qquad \ \ \
  e_{-\delta}^{} :=
[e_{-\delta+\alpha}^{},e_{-\alpha}^{}]_{\sqrt{\h}^{-1}}\, ,
\\
e_{n\delta+\alpha}^{}\!\!\!&:=\!\!\!& \frac{1}{a}
 [e_{(n-1)\delta+\alpha}^{},e_{\delta}^{}]~,
 \qquad \ \ \;\,
 e_{-n\delta-\alpha}^{} :=  \frac{1}{a}
[e_{-\delta}^{},e_{-(n-1)\delta-\alpha}^{}]\, ,
\nonumber\\
e_{(n+1)\delta-\alpha}^{}\!\!\!&:=\!\!\! & \frac{1}{a}
[e_{\delta}^{},e_{n\delta-\alpha}^{}]~,\;\,
\qquad \ \
 e_{-(n+1)\delta+\alpha}^{} := \frac{1}{a}
[e_{-n\delta+\alpha}^{},e_{-\delta}^{}]\, ,
\nonumber\\
e_{n\delta}'\!\!\!&:=\!\!\! & [e_{\alpha}^{},e_{n\delta-\alpha}^{}]_{\sqrt{\h}}~,
\qquad    \qquad \qquad
  e_{-n\delta}' :=
[e_{-n\delta+\alpha}^{},e_{-\alpha}^{}]_{\sqrt{\h}^{-1}}\, ,\nonumber
\end{eqnarray}
where $n=1,2,\ldots$, and $a={\h}^{\frac{1}{2}}+{\h}^{-\frac{1}{2}}$.
Analogously we set
\begin{eqnarray}
\tilde{e}_{\delta}^{}\!\!\!&:=\!\!\!&
[e_{\delta-\alpha}^{},e_{\alpha}^{}]_q~,\qquad\qquad \qquad
\tilde{e}_{-\delta}^{}:=[e_{-\alpha}^{},e_{-\delta+\alpha}^{}]_{q^{-1}},
\\
\tilde{e}_{(n+1)\delta-\alpha}^{}\!\!\!&:=\!\!\!&\frac{1}{a}
[\tilde{e}_{n\delta-\alpha}^{},\tilde{e}_{\delta}^{}]~,\qquad\;\,
\tilde{e}_{-(n+1)\delta+\alpha}^{}:=\frac{1}{a}
[\tilde{e}_{-\delta}^{},\tilde{e}_{-n\delta+\alpha}^{}]~,
\nonumber\\
\tilde{e}_{n\delta+\alpha}^{}\!\!\!&:=\!\!\!&\frac{1}{a}
[\tilde{e}_{\delta}^{},\tilde{e}_{(n-1)\delta+\alpha}^{}]~,\qquad\;\,
\tilde{e}_{-n\delta-\alpha}^{}:=\frac{1}{a}
[\tilde{e}_{-(n-1)\delta-\alpha}^{},\tilde{e}_{-\delta}^{}]~,
\nonumber\\
\tilde{e}_{n\delta}'\!\!\!&:=\!\!\!&
[e_{\delta-\alpha}^{},\tilde{e}_{(n-1)\delta+\alpha}^{}]_{\sqrt{\h}}~,\qquad\;\;
\tilde{e}_{-n\delta}':=[e_{-\delta+\alpha}^{},
\tilde{e}_{-(n-1)\delta-\alpha}^{}]_{\sqrt{\h}^{-1}},\nonumber
\end{eqnarray}
where $n=1,2,\ldots$. Both versions are related by the involution $\tau$:
\begin{equation}
\begin{array}{rccccl}
\tau(e_{n\delta+\alpha})\!\!\!&=\!\!\!&\tilde{e}_{(n+1)\delta-\alpha},\quad\;
\tau(\tilde{e}_{n\delta+\alpha})\!\!\!&
=\!\!\!&e_{(n+1)\delta-\alpha}\quad &(n\in\mathbb{Z}),\quad
\\[7pt]
\tau(e_{n\delta-\alpha})\!\!\!&=\!\!\!&\tilde{e}_{(n-1)\delta+\alpha},\quad\;
\tau(\tilde{e}_{n\delta-\alpha})\!\!&=\!\!&e_{(n-1)\delta+\alpha}\quad &(n\in\mathbb{Z}),
\\[7pt]
\tau(e'_{n\delta})\!\!\!&=\!\!\!&\tilde{e}'_{n\delta},\qquad\qquad\quad
\tau(\tilde{e}'_{n\delta})\!\!\!&=\!\!\!&e'_{n\delta}\qquad\qquad &(n\neq 0).
\end{array}
\end{equation}
These Cartan-Weyl basis elements allow to represent the R-matrix, providing the
quasitriangular structure for our Hopf algebra, in the triangular form.
In order to describe it, one has to introduce another set of generators $e_{n\delta}$ for pure imaginary roots, so that the relation is described best using generating functions.
\begin{eqnarray}
{\mathcal E}'_{\pm}(u)=({\h}^{\frac{1}{2}}-{\h}^{-\frac{1}{2}})\sum_{n\geq 1}^{}{e'}_{\pm n\delta}u^{\mp n}~,\quad
{\mathcal E}_{\pm}(u)=({\h}^{\frac{1}{2}}-{\h}^{-\frac{1}{2}})\sum_{n\geq 1}^{}e_{\pm n\delta}u^{\mp n}
\end{eqnarray}
so that ${\mathcal E}_{\pm}(u)=\ln(1+{\mathcal E}'_{\pm}(u))$. For $\tilde{e}_{n\delta}$ one can write identical formulas.\\

The R-matrix is an invertible element in $U_{{\h}}(\widehat{\mathfrak{sl}}_2)\otimes U_{{\h}}(\widehat{\mathfrak{sl}}_2)$,  which satisfies the following relations with respect to the coproduct $\Delta_{\sqrt{\h}}$ and opposite coproduct $\tilde{\Delta}_{\sqrt{\h}}=\sigma\Delta_{\sqrt{\h}}$ and $\sigma(a \otimes b)=b\otimes a$:
\begin{eqnarray}
&&\tilde{\Delta}_{\sqrt{\h}}(a)\!\!=\!\!R\Delta_{\sqrt{\h}}(a)R^{-1} \qquad\quad\;\;
\forall\,\,a \in U_{\sqrt{\h}}(\widehat{\mathfrak{sl}}_2)~,
\nonumber\\
&&(\Delta_{\sqrt{\h}}\otimes{\rm id})R\!\!=\!\!R^{13}R^{23}~,\qquad
({\rm id}\otimes\Delta_{\sqrt{\h}})R=R^{13}R^{12},
\end{eqnarray}
The relations above can be understood in the following way:
$R^{12}=\sum a_{i}\otimes b_{i}\otimes{\rm id}$,
$R^{13}=\sum a_{i}\otimes{\rm id}\otimes b_{i}$,
$R^{23}=\sum {\rm id}\otimes a_{i}\otimes b_{i}$ if $R$ has the form
$R=\sum a_{i}\otimes b_{i}$.

Using the definition of the $q$-exponential:
\begin{equation}
\exp_{q}(x):=1+x+\frac{x^{2}}{(2)_{q}!}+\dots+
\frac{x^{n}}{(n)_{q}!}+\dots=
\sum_{n\geq0}\frac{x^{n}}{(n)_{q}!},
\end{equation}
where $(n)_{q}:=\frac{q^{n}-1}{q-1}$.

we give an explicit expression of the universal $R$-matrix:
\begin{equation}
R=R_{+}R_{0}R_{-}\mathcal{K}~.
\end{equation}
Here the factors $\mathcal{K}$ and $R_{\pm}$ have the following form:
\begin{eqnarray}
\mathcal{K}&\!\!=\!\!&{\h}^{\frac{1}{2(\alpha,\alpha)}h_{\alpha}\otimes h_{\alpha}},
\nonumber\\
R_{+}&\!\!=\!\!&\prod_{n\ge 0}^{\rightarrow}\mathcal{R}_{n\delta+\alpha}, \qquad
R_{-}=\prod_{n\ge 1}^{\leftarrow}\mathcal{R}_{n\delta-\alpha}.
\end{eqnarray}
were the elements $\mathcal{R}_{\gamma}$ are given by the formula
\begin{equation}
\mathcal{R}_{\gamma}=\exp_{\sqrt{\h}^{-1}_{\gamma}}
\Big((\h^{1/2}-\h^{-1/2})(e_{\gamma}\otimes e_{-\gamma})\Big)~,
\end{equation}
where ${\h}_{\gamma}={\h}^{(\gamma,\gamma)}$.
The factor $R_{0}$ is defined as follows:
\begin{equation}
R_{0}=\exp\Big((\h^{1/2}-\h^{-1/2})\sum_{n>0}^{}d(n)%
e_{n\delta}\otimes e_{-n\delta}\Big)~,
\end{equation}
where $d(n)$ is given by
\begin{equation}
d(n)=\frac{n({\h}^{\frac{1}{2}}-{\h}^{-\frac{1}{2}})}
{{\h}^{n}-{\h}^{-n}}~.
\end{equation}
In terms of ``dual", tilded generators the universal R-matrix will have the form:
\begin{equation}
R=\tilde{R}_{+}\tilde{R}_{0}\tilde{R}_{-}\tilde{\mathcal{K}}~.
\end{equation}
Here the factors $\t{\mathcal{K}}$ and $\t R_{\pm}$ have the following form
\begin{eqnarray}
\tilde{\mathcal{K}}&\!\!=\!\!&{\h}^{\frac{1}{2(\alpha,\alpha)}h_{\delta-\alpha}\otimes h_{\delta-\alpha}},
\nonumber\\
\tilde{R}_{+}&\!\!=\!\!&\prod_{n\ge 1}^{\rightarrow}\tilde{\mathcal{R}}_{n\delta-\alpha}, \qquad
\tilde{R}_{-}=\prod_{n\ge 0}^{\leftarrow}\tilde{\mathcal{R}}_{n\delta+\alpha}.
\end{eqnarray}
were the elements $\tilde{\mathcal{R}}_{\gamma}$ are given by the formula
\begin{equation}
\tilde{\mathcal{R}}_{\gamma}=\exp_{\sqrt{\h}^{-1}_{\gamma}}
\Big(({\h}^{\frac{1}{2}}-{\h}^{-\frac{1}{2}})(\tilde{e}_{\gamma}\otimes \tilde{e}_{-\gamma})\Big)~,
\end{equation}
and
\begin{equation}
\t{R}_{0}=\exp\Big(({\h}^{\frac{1}{2}}-{\h}^{-\frac{1}{2}})\sum_{n>0}^{}d(n)%
\tilde{e}_{n\delta}\otimes \tilde{e}_{-n\delta}\Big)~.
\end{equation}

\end{appendices}

\end{document}